\documentclass[reqno]{amsproc}
\usepackage{amssymb}
\usepackage{graphicx}
\usepackage{bbold}
\usepackage{euscript}
\usepackage{mathrsfs} %\mathscr T
\usepackage{units}   %\nicefrac %kolejarski ulamek
\usepackage{color}

\makeatletter
\@namedef{subjclassname@2010}{%
\textup{2010} Mathematics Subject Classification}
\makeatother

\numberwithin{equation}{section}

\newtheorem{thm}{Theorem}[section]
\newtheorem{cor}[thm]{Corollary}
\newtheorem{lem}[thm]{Lemma}
\newtheorem{pro}[thm]{Proposition}

\newtheorem*{thm*}{Theorem}

\theoremstyle{remark}
\newtheorem{rem}[thm]{Remark}

\theoremstyle{definition}
\newtheorem{exa}[thm]{Example}

 \DeclareMathOperator{\D}{d}
\DeclareMathOperator{\E}{e}

\newcommand*{\alu}{\EuScript A}
\newcommand*{\bb}{\mathcal B}
\newcommand*{\borel}[1]{{\mathfrak B}(#1)}
\newcommand*{\cbb}{\mathbb C}

\newcommand*{\dz}[1]{{\EuScript D}(#1)}

\newcommand*{\ee}{\mathcal E}

\newcommand*{\fscr}{\mathscr F}

\newcommand*{\gfrak}{\mathfrak{G}}
\newcommand*{\Ge}{\geqslant}
\newcommand*{\hh}{\mathcal H}

\newcommand*{\I}{{\mathrm i}}
\newcommand*{\is}[2]{\langle#1,#2\rangle}
\newcommand*{\jd}[1]{\EuScript N(#1)}
\newcommand*{\kscr}{\mathscr K}
\newcommand*{\Le}{\leqslant}
\newcommand*{\mm}{\mathcal M}
\newcommand*{\nbb}{\mathbb N}
\newcommand*{\ogr}[1]{\boldsymbol B(#1)}
\newcommand*{\ogrp}[1]{\boldsymbol B_{+}(#1)}
\newcommand*{\ob}[1]{{\mathcal R}(#1)}
\newcommand*{\pcal}{{\mathcal P}}
\newcommand*{\zscr}{\mathscr{Z}}

\newcommand*{\rbb}{\mathbb R}
\newcommand*{\smalloplus}{\raise0pt\hbox{$\scriptscriptstyle \oplus$}}

\newcommand*{\yscr}{\mathscr Y}

\newcommand*{\zbb}{\mathbb Z}

\begin{document}
   \title[Composition operators on
entire functions with analytic symbols]{Composition
operators on Hilbert spaces of entire \\ functions
with analytic symbols}
   \author[J.\ Stochel]{Jan Stochel}
\address{Instytut Matematyki, Uniwersytet Jagiello\'nski,
ul.\ \L ojasiewicza 6, PL-30348 Kra\-k\'ow, Poland}
\email{Jan.Stochel@im.uj.edu.pl}
   \author[J.\ B.\ Stochel]{Jerzy Bart{\l}omiej Stochel}
\address{Faculty of Applied Mathematics,
AGH University of Science and Technology, Al. Mickiewicza
30, PL-30059 Krak\'ow, Poland} \email{stochel@agh.edu.pl}
   \thanks{The research of the first author was supported
by the NCN (National Science Center), decision No.
DEC-2013/11/B/ST1/03613. The research of the second
author was supported by the local grant No.
11.11.420.004.}



   \begin{abstract}
Composition operators with analytic symbols on some
reproducing kernel Hilbert spaces of entire functions
on a complex Hilbert space are studied. The questions
of their boundedness, seminormality and positivity are
investigated. It is proved that if such an operator is
bounded, then its symbol is a polynomial of degree at
most $1$, i.e., it is an affine mapping. Fock's type
model for composition operators with linear symbols is
established. As a consequence, explicit formulas for
their polar decomposition, Aluthge transform and
powers with positive real exponents are provided. The
theorem of Carswell, MacCluer and Schuster is
generalized to the case of Segal-Bargmann spaces of
infinite order. Some related questions are also
discussed.
   \end{abstract}
   \maketitle
   \newpage
\setcounter{tocdepth}{2}
\tableofcontents
   \section{Introduction}
   Composition operators in Banach spaces of analytic
functions with analytic symbols have been intensively
studied since at least the 1970's (see the classical
survey article \cite{Nor78}). The literature on this
subject is vast and still growing. Most effort has
been focussed on investigating bounded composition
operators in Banach spaces of analytic functions on
bounded regions in finite dimensional Euclidean
spaces. We mention only a few selected articles:
\cite{Jon01} (composition operators on weighted
Bergman spaces of the unit disk), \cite{Cow-Gun-Ko13}
(weighted composition operators on weighted Hardy
spaces of the unit disk), \cite{Cha-Par14} (weighted
composition operators on the Hardy-Hilbert space of
the unit disc), \cite{MC97} (Fredholm composition
operators over bounded domains in $\cbb^N$) and
\cite{Char13} (composition operators on weighted
Bergman-Orlicz Spaces of the unit ball). It is also
worth mentioning two monographs on analytic
composition operators: \cite{Sh93} (on the Hardy space
$H^2$) and \cite{Cow-MC95} (on the Hardy space $H^p$
of the unit ball).

Since the publication of the article \cite{c-m-s03},
there has been growing interest in investigating
analytic (weighted) composition operators in Banach
spaces of entire functions on finite dimensional
Euclidean spaces. Most work has been concerned with
Fock-type spaces over $\cbb^n$ (see, e.g.,\
\cite{Ue07,Ue07a,Ue10,LeT14,Zha-Pa15} for $n=1$ and
\cite{Ste09,C-I-K10,Du11,Men14,Zha14} for $n\Ge 1$).

The goal of the present paper is to study composition
operators with holomorphic symbols on a reproducing
kernel Hilbert spaces $\varPhi(\hh)$ induced by a
nonconstant entire function $\varPhi$ in $\cbb$ whose
Taylor coefficients at the origin are nonnegative.
This space consists of entire functions on a complex
Hilbert spaces $\hh$ (finite- or
infinite-dimensional). It is worth pointing out that
Frankfurt spaces (see \cite{fr1,fr2,fr3}) and their
multidimensional generalizations (see \cite{FHSz2})
that consist of entire functions which are square
integrable with respect to some symmetric measures fit
well into the scope of our approach (see Example
\ref{phih}). The Segal-Bargmann spaces (nowadays often
called Fock spaces) are the most spectacular examples
of them with $\varPhi=\exp$ (see \cite[Chapter
VI]{seg} and \cite{Barg}; see also
\cite{Barg2,Ja-Ru,jerz0,jerz,Ja-Ru-2,jerz4,jerz3,Ma-St}
for the infinite-dimensional case).

The study of composition operators with holomorphic
symbols in the Segal-Bargmann space $\bb_d$ of finite
order $d$ was initiated by Carswell, MacCluer and
Schuster in \cite{c-m-s03}. They proved, among other
things, that the symbol of a bounded composition
operator on $\bb_d$ is an affine mapping. As shown in
the present paper, this is also true for composition
operators on $\varPhi(\hh)$ (see Proposition
\ref{pro1}). In contrast to the bounded case, densely
defined composition operators on $\bb_d$ may have
non-affine holomorphic symbols (see Example
\ref{ndd}). Here, we concentrate mostly on composition
operators with affine symbols. What is more, we do not
impose any restrictions on the dimension of the
underlying Hilbert space $\hh$.

The paper is organized as follows. We begin by
discussing the selected properties of generalized
inverses of positive bounded Hilbert space operators
(see Section \ref{sek2}). From the point of view of
the present paper the most important property of
generalized inverses is the limit property which is
described in Lemma \ref{cfkz-2}. It is used in the
proof of Theorem \ref{glowne2}. The reproducing kernel
Hilbert space $\varPhi(\hh)$ is introduced and
initially investigated in Section \ref{sekt3}. The
basic properties of (not necessarily bounded)
composition operators $C_{\varphi}$ in $\varPhi(\hh)$
with holomorphic symbols $\varphi$ are established in
Section \ref{sekt4}. A special emphasis is placed on
describing the adjoint operator $C_{\varphi}^*$ (see
Theorem \ref{sprz}). Section \ref{sekt5} is devoted to
showing that members of the class of densely defined
composition operators in $\varPhi(\hh)$ with
holomorphic symbols are maximal with respect to
inclusion of graphs (see Theorem \ref{fipsi}). In
Section \ref{sekt6}, we prove that the symbol of a
bounded composition operator on $\varPhi(\hh)$ is a
polynomial of degree at most $1$, or in other words,
an affine mapping (see Proposition \ref{pro1}).
Proposition \ref{fipsib} shows that the question of
boundedness of a composition operator in
$\varPhi(\hh)$ with an affine symbol reduces to
investigating the case when the linear part of the
symbol is positive. Section \ref{sekt7} provides
Fock's type model for composition operators in
$\varPhi(\hh)$ with linear symbols (see Theorem
\ref{fmod}). The model itself depends on a choice of a
conjugation on the underlying Hilbert space $\hh$. The
reader is referred to Appendix \ref{apap} for more
information on conjugations on Hilbert spaces. Fock's
type model enables us to show that the adjoint of a
composition operator $C_A$ with a linear symbol $A$ is
the composition operator with the symbol $A^*$, which
is the adjoint operator of $A$ (see Theorem
\ref{wnoca}). In particular, this gives us a
description of all selfadjoint composition operators
in $\varPhi(\hh)$ with linear symbols (see Corollary
\ref{sa-sa}). In Section \ref{sekt8}, we completely
characterize the boundedness of composition operators
in $\varPhi(\hh)$ with linear symbols and give
explicit formulae for their norms and spectral radii
(see Theorem \ref{bca}). As a consequence, the
question of when a bounded composition operator on
$\varPhi(\hh)$ with a linear symbol is normaloid
(resp.,\ isometric, coisometric, unitary, an
orthogonal projection, a partial isometry) is
answered. Theorem \ref{gcd3}, which is the main result
of Section \ref{sekt9}, characterizes positivity of
composition operators in $\varPhi(\hh)$ with linear
symbols. This section also contains explicit
descriptions of powers (with positive real exponents),
the polar decomposition and the generalized Aluthge
transform of composition operators in $\varPhi(\hh)$
with linear symbols. Section \ref{sekt10} is devoted
mainly to characterizing seminormality of composition
operators in $\varPhi(\hh)$ with linear symbols (see
Theorem \ref{chca}; see also Theorems \ref{nier} and
\ref{wn2} for related inequalities). The problem of
characterizing subnormality of such operators is
partially solved in Proposition \ref{subnwkw}.
However, in its full generality this problem still
remains unsolved (even in the case of
$\varPhi(z)=z^2$, see \cite[Open Question]{js2}).
Section \ref{sekt11} provides some necessary
conditions for boundedness and seminormality of
composition operators in $\exp(\hh)$ with affine
symbols (see Lemma \ref{n&sc4b} and Proposition
\ref{cohyp}).

In Section \ref{sec9}, we give another and a
significantly shorter proof of the
Carswell-MacCluer-Schuster theorem (see Theorem
\ref{s-b-norm} and Proposition \ref{naszeq}). Our idea
is to reduce the proof to the one-dimensional case,
and then to combine some estimates for the norms of
iterations of the square of the modulus of a
composition operator together with Gelfand's formula
for the spectral radius (see Lemmata \ref{sb2} and
\ref{sb4}). As a consequence, the spectral radius of a
bounded composition operator on the Segal-Bargmann
space of finite order with a holomorphic symbol is
proved to be equal to $1$ (see Theorem
\ref{s-b-prsp}). In turn, such a composition operator
is normaloid if and only if its symbol vanishes at $0$
(see Theorem \ref{semi-lid}), and it is seminormal if
and only if it is normal (see Corollary
\ref{semi-nor}). In Section \ref{Sek13}, we discuss
the relationship between a composition operator
$C_{\varphi}$ acting in the Segal-Bargmann space
$\bb_d$ and the corresponding composition operator
$\tilde C_{\varphi}$ acting in $L^2(\mu_d)$, where
$\varphi$ is a holomorphic symbol and $\mu_d$ is a
Gaussian measure on $\cbb^d$. It is proved that if
$\tilde C_{\varphi}$ is well-defined, then
$C_{\varphi}$ is bounded if and only if $\tilde
C_{\varphi}$ is bounded (see Theorem \ref{bl2G}). It
is worth pointing out that, in contrast to the case of
analytic composition operators on $\bb_d$, there are
bounded cosubnormal (resp.,\ cohyponormal) composition
operators on $L^2(\mu_d)$ with holomorphic symbols
which are not normal (resp.,\ cosubnormal) (see Remark
\ref{mamde}). Theorem \ref{glowne2}, which is the main
result of Section \ref{Sek14}, provides two
characterizations of boundedness of a composition
operator on $\exp(\hh)$ with a holomorphic symbol (see
the equivalences (i)$\Leftrightarrow$(ii) and
(i)$\Leftrightarrow$(iii)) as well as two formulas for
its norm (see \eqref{Jur3}). This can be thought of as
a far-reaching generalization of the
Carswell-MacCluer-Schuster theorem. Let us mention
that the equivalence (i)$\Leftrightarrow$(iii) and the
equality $\|C_{\varphi}\|^2 =
\exp(\|(I-AA^*)^{-1/2}b\|^2)$ have been proved
independently by Trieu Le using a different approach
(cf.\ \cite{LeT}). Our method of proving Theorem
\ref{glowne2} is based on an approximation technique
which enables us to reduce the reasoning to the finite
dimensional case (which, as said above, can in turn be
reduced to the one-dimensional case). If the linear
part $A$ of a holomorphic symbol $\varphi$ of a
bounded composition operator $C_{\varphi}$ on
$\exp(\hh)$ is a strict contraction, then the spectral
radius of $C_{\varphi}$ is equal to $1$ (see
Proposition \ref{spr1}). Example \ref{optim} shows
that there are bounded cohyponormal (hence normaloid)
analytic composition operators on $\exp(\hh)$ which
are not normal and whose spectral radii can take any
value in the interval $(1,\infty)$ (which is never the
case when $\dim \hh < \infty$). In Example \ref{uww}
we illustrate Theorem \ref{glowne2} in the context of
diagonal operators.

The paper is concluded with two appendices. Appendix
\ref{apap} provides proofs of selected properties of
conjugations on Hilbert spaces while Appendix
\ref{apap2} deals with the question of paranormality
of tensor products. Both appendices are strongly
related to the Fock's type model for analytic
composition operators in $\varPhi(\hh)$ with linear
symbols (see Section \ref{sekt7}).

We refer the reader to the monograph \cite{chae} for
more information on holomorphy in normed spaces.
Concerning reproducing kernel Hilbert spaces, we refer
the reader to the classical paper by Aronszajn
\cite{Aronsz} as well as to modern treatises
\cite{FHSz3,FHSz4,FHSz5}.
   \section{\label{sek2}Preliminaries}
   Let $\cbb$, $\rbb$, $\rbb_+$ and $\zbb_+$ denote
the sets of complex numbers, real numbers, nonnegative
real numbers and nonnegative integers, respectively.
Set $\nbb=\zbb_+ \setminus \{0\}$. Let $\hh$ be a
complex Hilbert space. Given a family
$\{M_{\omega}\}_{\omega\in \varOmega}$ of nonempty
subsets of $\hh$, we denote by $\bigvee_{\omega\in
\varOmega} M_{\omega}$ the closed linear span of
$\bigcup_{\omega\in \varOmega} M_{\omega}$. We write
$M^\perp$ for the orthocomplement of a nonempty set $M
\subseteq \hh$. By an {\em operator} in $\hh$ we mean
a linear map $A\colon \hh \supseteq \dz{A} \to \hh$
defined on a vector subspace $\dz{A}$ of $\hh$ called
the {\em domain} of $A$. The kernel, the range, the
adjoint and the closure of $A$ are denoted by
$\jd{A}$, $\ob{A}$, $A^*$ and $\bar A$, respectively
(provided they exist). The {\em graph inner product}
$\is{\cdot}{\mbox{-}}_A$ of $A$ is given by
$\is{f}{g}_A = \is{f}{g} + \is{Af}{Ag}$ for $f,g \in
\dz{A}$. A vector subspace $\ee$ of $\dz{A}$ is called
a {\em core} for $A$ if $\ee$ is dense in $\dz{A}$ in
the {\em graph norm} of $A$ (given by the graph inner
product $\is{\cdot}{\mbox{-}}_A$). If $A$ is closed,
then $\ee$ is a core for $A$ if and only if
$A=\overline{A|_{\ee}}$. An operator in $\hh$ which
vanishes on its domain is called a {\em zero
operator}. A zero operator may not be closed, or even
densely defined.

We say that a densely defined operator $A$ in $\hh$ is
   \begin{enumerate}
   \item[$\bullet$] {\em positive}
if $\is{A\xi}{\xi} \Ge 0$ for all $\xi\in \dz{A}$; then we
write $A\Ge 0$,
   \item[$\bullet$] {\em selfadjoint} if $A=A^*$,
   \item[$\bullet$] {\em hyponormal} if
$\dz{A} \subseteq \dz{A^*}$ and $\|A^*\xi\| \Le \|A\xi\|$
for all $\xi \in \dz{A}$,
   \item[$\bullet$] {\em cohyponormal} if $\dz{A^*} \subseteq
\dz{A}$ and $\|A\xi\| \Le \|A^*\xi\|$ for all $\xi \in
\dz{A^*}$,
   \item[$\bullet$] {\em normal} if $A$ is hyponormal and
cohyponormal,
   \item[$\bullet$] {\em subnormal} if there exist
a complex Hilbert space $\mm$ and a normal operator
$N$ in $\mm$ such that $\hh \subseteq \mm$ (isometric
embedding), $\dz{A} \subseteq \dz{N}$ and $Af = Nf$
for all $f \in \dz{A}$,
   \item[$\bullet$] {\em seminormal} if $A$ is either
hyponormal or cohyponormal.
   \end{enumerate}
Recall that normal operators are subnormal and
subnormal operators are hyponormal, but not conversely
(even in the case of bounded operators, see
\cite{Hal,Fur}). It follows from the von Neumann
theorem that a closed operator $A$ in $\hh$ is
hyponormal (resp.,\ cohyponormal) if and only if $A^*$
is cohyponormal (resp.,\ hyponormal).

We denote by $\ogr{\hh}$ the $C^*$-algebra of all
bounded operators defined on the whole $\hh$. We write
$I=I_{\hh}$ for the identity operator on $\hh$. Set
   \begin{align*}
\ogrp{\hh} = \{A \in \ogr{\hh}\colon A \Ge 0\}.
   \end{align*}
The spectrum and spectral radius of $A \in \ogr{\hh}$
is denoted by $\sigma(A)$ and $r(A)$ respectively. An
operator $A\in \ogr{\hh}$ is said to be {\em
normaloid} if $r(A)=\|A\|$ (cf.\ \cite{Fur}). Clearly,
$A$ is normaloid if and only if $A^*$ is normaloid. An
operator $A\in \ogr{\hh}$ is said to be {\em
paranormal} if $\|Af\|^2 \Le \|f\| \|A^2f\|$ for every
$f \in \hh$. It is well-known that paranormal
operators are normaloid (see \cite{Fur} for more
information on this and related classes of operators).

We refer the reader to \cite{Gui,R-S,Part} for the
foundations of the theory of tensor products of
Hilbert space operators including symmetric tensor
powers of such operators.

Each closed densely defined operator $A$ in $\hh$ has
a unique polar decomposition $A = U|A|$, where $U\in
\ogr{\hh}$ is a partial isometry such that $\jd{U} =
\jd{A}$ and $|A| = (A^*A)^{1/2}$ (cf.\ \cite[Section
8.1]{b-s}). Given two positive selfadjoint operators
$A$ and $B$ in $\hh$, we write $A\preccurlyeq B$ if
$\dz{B^{1/2}} \subseteq \dz{A^{1/2}}$ and $\|A^{1/2}
\xi\| \Le \|B^{1/2} \xi\|$ for all $\xi \in
\dz{B^{1/2}}$. If $A, B \in \ogrp{\hh}$, then $A
\preccurlyeq B$ if and only if $A \Le B$, i.e., $B-A
\Ge 0$.

Suppose $A\in \ogr{\hh }$ is selfadjoint. Since
$A|_{\overline{\ob{A}}}\colon \overline{\ob{A}} \to
\ob{A}$ is a bijection, we may define a generalized
inverse $A^{-1}$ of $A$ by
$A^{-1}=\big(A|_{\overline{\ob{A}}}\big)^{-1}$.
Clearly, we~ have
   \begin{align}  \label{aa-1}
   \begin{gathered}
\text{$\dz{A^{-1}} = \ob{A}$, $\ob{A^{-1}} =
\overline{\ob{A}}$,}
   \\
\text{$AA^{-1}=I_{\ob{A}}$ and $A^{-1}A=P$,}
   \end{gathered}
   \end{align}
where $I_{\ob{A}}$ is the identity operator on
$\ob{A}$ and $P$ is the orthogonal projection of $\hh$
onto $\overline{\ob{A}}$. If $A \in \ogrp{\hh}$, then
we write
   \begin{align} \label{aa-2}
A^{-t} = (A^{t})^{-1}, \quad t \in (0,\infty).
   \end{align}
We show that if $A$ is a positive operator with closed
range, then $A^{-1/2}=(A^{-1})^{1/2}$.
   \begin{lem} \label{a-12}
Let $A\in \ogrp{\hh}$ be an operator with closed
range. Then $\ob{A}=\ob{A^{1/2}}$,
$A^{-1}\in\ogrp{\ob{A}}$, $A^{-1/2} \in\ogrp{\ob{A}}$,
$A^{-1} = A^{-1/2}A^{-1/2}$ and
   \begin{align*}
\is{A^{-1}\xi} {\xi}=\|A^{-1/2}\xi\|^2, \quad \xi \in
\ob{A}.
   \end{align*}
   \end{lem}
   \begin{proof}
Indeed, since
   \begin{align*}
\overline{\ob{A^{1/2}}} = \overline{\ob{A}} = \ob{A}
\subseteq \ob{A^{1/2}} \subseteq
\overline{\ob{A^{1/2}}},
   \end{align*}
we get $\ob{A} = \ob{A^{1/2}}$. Hence $\ob{A}$ reduces
$A$ and $A^{1/2}$, and
   \begin{align*}
A^{-1}=\Big((A^{1/2}|_{\ob{A}})(A^{1/2}|_{\ob{A}})\Big)^{-1}
=A^{-1/2}A^{-1/2}.
   \end{align*}
This and the closed graph theorem imply that
$A^{-1}\in\ogrp{\ob{A}}$, $A^{-1/2} \in\ogrp{\ob{A}}$
and
   \begin{align*}
& \is{A^{-1}\xi}{\xi} = \|A^{-1/2}\xi\|^2, \quad \xi
\in \ob{A}. \qedhere
   \end{align*}
   \end{proof}
   Now we extend the definition of the partial order
$\preccurlyeq$ so as to cover the case of generalized
inverses of positive operators $A \in \ogrp{\hh}$
(note that, in general, $A^{-1}$ may not be densely
defined). Given two operators $A,B \in \ogrp{\hh}$, we
write $B^{-1} \preccurlyeq A^{-1}$ if $\dz{A^{-1/2}}
\subseteq \dz{B^{-1/2}}$ and $\|B^{-1/2} f \| \Le
\|A^{-1/2} f \|$ for all $f \in \dz{A^{-1/2}}$.
Clearly, if $A$ and $B$ are invertible in $\ogr{\hh}$,
then $B^{-1} \preccurlyeq A^{-1}$ if and only if
$B^{-1} \Le A^{-1}$. As shown below, the inequality
$B^{-1} \preccurlyeq A^{-1}$ turns out to be
equivalent to the inequality $A\Le B$.
   \begin{lem} \label{cnier}
If $A,B \in \ogrp{\hh}$ and $\varepsilon \in
(0,\infty)$, then the following conditions are
equivalent{\em :}
   \begin{enumerate}
   \item[(i)] $A \Le B$,
   \item[(ii)] $B^{-1} \preccurlyeq A^{-1}$,
   \item[(iii)] $(\varepsilon + B)^{-1} \Le (\varepsilon +
A)^{-1}$.
   \end{enumerate}
   \end{lem}
   \begin{proof}
(i)$\Rightarrow$(ii) If $E \in \ogrp{\hh}$, then
$P_E\colon \hh\to \overline{\ob{E^{1/2}}}$ stands for
the orthogonal projection of $\hh$ onto
$\overline{\ob{E^{1/2}}}$ and $J_E\colon
\overline{\ob{E^{1/2}}} \to \hh$ is given by
$J_E\xi=\xi$ for all $\xi \in
\overline{\ob{E^{1/2}}}$. Since $\|A^{1/2} \xi\| \Le
\|B^{1/2} \xi\|$ for all $\xi \in \hh$, there exists a
linear contraction $\tilde T\colon
\overline{\ob{B^{1/2}}} \to \overline{\ob{A^{1/2}}}$
such that $\tilde T B^{1/2} = A^{1/2}$. Set
$T=J_A\tilde T P_B$. Then $T \in \ogr{\hh}$, $\|T\|
\Le 1$, $T B^{1/2} = A^{1/2}$ and $T^* = J_B \tilde
T^* P_A$. This implies that $\ob{T^*} = \ob{\tilde
T^*} \subseteq \overline{\ob{B^{1/2}}}$ and $A^{1/2} =
B^{1/2} T^*$, and thus $\ob{A^{1/2}} \subseteq
\ob{B^{1/2}}$ (see also \cite[Theorem 1]{Doug}).
Applying \eqref{aa-1}, we get
   \begin{align} \label{baba}
B^{-1/2} A^{1/2} = B^{-1/2} B^{1/2} T^* = P_{B}
T^*=T^*.
   \end{align}
Hence, by \eqref{aa-1}, we have
   \begin{align*}
\|B^{-1/2} \eta\| = \|(B^{-1/2}A^{1/2})A^{-1/2}\eta\|
\overset{\eqref{baba}}= \|T^* A^{-1/2}\eta\| \Le
\|A^{-1/2}\eta\|, \quad \eta \in \ob{A^{1/2}}.
   \end{align*}
This means that $B^{-1} \preccurlyeq A^{-1}$.

(ii)$\Rightarrow$(i) Since $\ob{A^{1/2}} \subseteq
\ob{B^{1/2}}$ and $\|B^{-1/2} \eta\| \Le \|A^{-1/2}
\eta\|$ for all $\eta\in \ob{A^{1/2}}$, we get
   \begin{align*}
\|B^{-1/2} A^{1/2}A^{-1/2} \eta\| \Le
\|A^{-1/2}\eta\|, \quad \eta \in \ob{A^{1/2}}.
   \end{align*}
Set $T= B^{-1/2} A^{1/2}$. It follows from the above
that $\dz{T}=\hh$, $T \in \ogr{\hh}$ and $\|T\|\Le 1$.
Moreover, we have
   \begin{align*}
B^{1/2} T = (B^{1/2} B^{-1/2}) A^{1/2}
\overset{\eqref{aa-1}}= I_{\ob{B^{1/2}}} A^{1/2} =
A^{1/2}.
   \end{align*}
This implies that $A^{1/2} = T^* B^{1/2}$ and thus
   \begin{align*}
\is{A\xi}{\xi}=\|A^{1/2}\xi\|^2 = \|T^* B^{1/2}\xi\|^2
\Le \|B^{1/2}\xi\|^2=\is{B\xi}{\xi}, \quad \xi \in
\hh.
   \end{align*}

(i)$\Leftrightarrow$(iii) This follows from
\cite[Theorem VI.2.21]{Kat}. Alternatively, since for
any $\varepsilon > 0$, $A\Le B$ if and only if
$A+\varepsilon \Le B+\varepsilon$, we can apply the
equivalence (i)$\Leftrightarrow$(ii) to the operators
$A+\varepsilon$ and $B+\varepsilon$ which are
invertible in $\ogr{\hh}$.
   \end{proof}
The next lemma is a particular case of the Mac
Nerney-Shmul'yan characterization of the range of a
bounded operator.
   \begin{lem}[\mbox{\cite[Theorem 3]{MacN}},
\mbox{\cite[Lemma 3]{Shm}}] \label{cfkz-0} If $A\in
\ogrp{\hh}$ and $\xi \in \hh$, then the following
conditions are equivalent{\em :}
   \begin{enumerate}
   \item[(i)] $\xi \in \ob{A^{1/2}}$,
   \item[(ii)] there exists $c \in \rbb_+$  such that
$|\is{\xi}{h}| \Le c \|A^{1/2}h\|$ for all $h \in
\hh$.
   \end{enumerate}
Moreover, if $\xi \in \ob{A^{1/2}}$, then the smallest
$c\in \rbb_+$ in {\em (ii)} is equal to
$\|A^{-1/2}\xi\|$.
   \end{lem}
   \begin{proof}
We justify only the ``moreover'' part. It is
well-known that the smallest $c\in \rbb_+$ in (ii) is
equal to $\|\eta\|$, where $\eta \in
\overline{\ob{A^{1/2}}}$ is a unique solution of the
equation $\xi = A^{1/2}\eta$. However, this means that
$\eta=A^{-1/2}\xi$.
   \end{proof}
The Mac Nerney-Shmul'yan theorem can be used to
describe the range of the square root of the limit of
a monotonically decreasing net of positive operators.
   \begin{lem} \label{cfkz-2}
If $\{A_{P}\}_{P \in \pcal} \subseteq \ogrp{\hh}$ is a
monotonically decreasing net which converges in the
weak operator topology to $A \in \ogrp{\hh}$ and $\xi
\in \hh$, then the following conditions are
equivalent{\em :}
   \begin{enumerate}
   \item[(i)] $\xi \in \ob{A^{1/2}}$,
   \item[(ii)] for every $P \in \pcal$,
$\xi \in \ob{A_{P}^{1/2}}$ and $c:=\sup_{P\in \pcal}
\|A_{P}^{-1/2} \xi\| < \infty$.
   \end{enumerate}
Moreover, if $\xi \in \ob{A^{1/2}}$, then
$c=\|A^{-1/2}\xi\|$.
   \end{lem}
   \begin{proof}
(i)$\Rightarrow$(ii) By Lemma \ref{cnier}, $A_{P}^{-1}
\preccurlyeq A^{-1}$ and $\ob{A^{1/2}} \subseteq
\ob{A_{P}^{1/2}}$ for all $P \in \pcal$. Hence, if
$\xi \in \ob{A^{1/2}}$, then $\xi \in
\ob{A_{P}^{1/2}}$ and $\|A_{P}^{-1/2}\xi\| \Le
\|A^{-1/2}\xi\|$ for all $P \in \pcal$, which implies
that
   \begin{align} \label{ca-12}
c \Le \|A^{-1/2}\xi\|.
   \end{align}

(ii)$\Rightarrow$(i) Noting that
   \begin{align*}
|\is{\xi}{\eta}| =
|\is{A_{P}^{-1/2}\xi}{A_{P}^{1/2}\eta}| \Le c
\|A_{P}^{1/2}\eta\|, \quad \eta \in \hh, \, P \in
\pcal,
   \end{align*}
and
   \begin{align*}
\lim_{P \in \pcal} (\|A_{P}^{1/2}\eta\|^2 -
\|A^{1/2}\eta\|^2) = \lim_{P \in \pcal}
\is{(A_{P}-A)\eta}{\eta} = 0, \quad \eta \in \hh,
   \end{align*}
we see that
   \begin{align*}
|\is{\xi}{\eta}| \Le c\|A^{1/2}\eta\|, \quad \eta \in
\hh,
   \end{align*}
which together with Lemma \ref{cfkz-0} yields $\xi \in
\ob{A^{1/2}}$ and $\|A^{-1/2}\xi\| \Le c$. This
combined with \eqref{ca-12} completes the proof of the
implication (ii)$\Rightarrow$(i) and the ``moreover''
part.
   \end{proof}
   \section{\label{sekt3}The reproducing kernel Hilbert space
$\varPhi(\hh)$}
   In what follows, we denote by $\fscr$ the set of
all entire functions $\varPhi$ of the form
   \begin{align} \label{wsp}
\varPhi(z) = \sum_{n=0}^\infty a_n z^n, \quad z \in \cbb,
   \end{align}
such that $a_k \Ge 0$ for all $k \in \zbb_+$ and
$a_{n} > 0$ for some $n\in \nbb$. Clearly, if $\varPhi
\in \fscr$, then $\varPhi|_{[0,\infty)}$ is
non-negative, strictly increasing and
$\lim_{[0,\infty) \ni x \to \infty}
\varPhi(x)=\infty$; hence, by Liouville's theorem,
$\limsup_{|z| \to \infty}|\varPhi(z)|=\infty$. Given
$\varPhi \in \fscr$ as in \eqref{wsp}, we set
$\zscr_{\varPhi} = \{k\in \zbb_+\colon a_k > 0\}$ and
$\gfrak_{\varPhi} = \bigcap_{k \in \zscr_{\varPhi}
\setminus \{0\}} G_k$, where $G_k:= \{z\in \cbb\colon
z^k = 1\}$ for $k\in \nbb$. The order of the
multiplicative group $\gfrak_{\varPhi}$ can be
calculated explicitly.
   \begin{lem} \label{gcd=1}
Suppose $Y$ is a nonempty subset of $\nbb$. Then
   \begin{enumerate}
   \item[(i)] $\bigcap_{k \in Y} G_k = G_{\mathrm{gcd}(Y)}$,
where $\mathrm{gcd}(Y)$ is the greatest common divisor
of $Y$,
   \item[(ii)] $\bigcap_{k \in Y} G_k =
\{1\}$ if and only if there exists a nonempty finite
subset of $Y$ which consists of relatively prime
integers.
   \end{enumerate}
   \end{lem}
   \begin{proof}
Clearly, $G_k$ is a multiplicative group of order $k$.
Since $G_k \cap G_l = G_k \cap G_{l-k}$ and
$\mathrm{gcd}(k,l)=\mathrm{gcd}(k,l-k)$ whenever $1
\Le k < l$, we deduce that $G_k \cap G_l =
G_{\mathrm{gcd}(k,l)}$ for all $k,l \in \nbb$. This
implies that $\bigcap_{k \in Y} G_k =
G_{\mathrm{gcd}(Y)}$ (because if $Y = \{k_1, k_2,
\ldots\}$, then the sequence $\{\mathrm{gcd}(k_1,
\dots, k_n)\}_{n=1}^{\infty}$ is monotonically
decreasing and thus $\mathrm{gcd}(Y) = \lim_{n \to
\infty} \mathrm{gcd}(k_1, \dots, k_n)$). Obviously (i)
implies (ii).
   \end{proof}
If $X$ is a nonempty set, then a function $K \colon
X\times X \to \cbb$ is called a kernel on $X$.
Following \cite{Aronsz}, we say that a kernel $K$ on
$X$ is {\em positive definite} if for all finite
sequences $\{\lambda_i\}_{i=1}^n \subseteq \cbb$ and
$\{x_i\}_{i=1}^n \subseteq X$,
   \begin{align*}
\sum_{i,j=1}^n K(x_i,x_j) \lambda_i \bar \lambda_j \Ge 0.
   \end{align*}
From now on, $\hh$ stands for a complex Hilbert space with
inner product $\is{\cdot}{\mbox{-}}$. Let $\varPhi \in
\fscr$. Applying the Schur product theorem (see \cite[p.\
14]{sch} or \cite[Theorem 7.5.3]{Hor-Joh}), we deduce that
the kernel $\hh \times \hh \ni (\xi,\eta) \mapsto
\is{\xi}{\eta}^n \in \cbb$ is positive definite for every
$n\in \zbb_+$, and thus the kernel $K^{\varPhi}\colon \hh
\times \hh \to \cbb$ defined by
   \begin{align}  \label{ajjaj}
K^{\varPhi}(\xi,\eta) = K^{\varPhi,\hh}(\xi,\eta) =
\varPhi(\is{\xi}{\eta}), \quad \xi,\eta \in \hh,
   \end{align}
is positive definite. It is clear that
   \begin{align*}
K^{\varPhi}(\xi,\eta) = \overline{K^{\varPhi}(\eta,\xi)},
\quad \xi,\eta \in \hh.
   \end{align*}
Denote by $\varPhi(\hh)$ the reproducing kernel
Hilbert space with the reproducing kernel
$K^{\varPhi}$ (see \cite[Chap.\ V]{Moo} and
\cite{Aronsz,FHSz3,FHSz4}), i.e., $\varPhi(\hh)$ is a
complex Hilbert space of complex valued functions on
$\hh$ (with the pointwise defined linear operations)
such that $\{K^{\varPhi}_{\xi}\colon \xi \in \hh\}
\subseteq \varPhi(\hh)$ and
   \begin{align} \label{rep}
f(\xi) = \is{f}{K^{\varPhi}_{\xi}}, \quad \xi \in \hh, \, f
\in \varPhi(\hh),
   \end{align}
where
   \begin{align} \label{kfxi}
K^{\varPhi}_{\xi}(\eta) = K^{\varPhi,\hh}_{\xi}(\eta)
= K^{\varPhi}(\eta,\xi), \quad \xi,\eta \in \hh.
   \end{align}
In particular, we have
   \begin{align} \label{rep2}
\varPhi(\is{\xi}{\eta})=K^{\varPhi}(\xi,\eta) =
\is{K^{\varPhi}_{\eta}}{K^{\varPhi}_{\xi}}, \quad \xi, \eta
\in \hh.
   \end{align}
Applying the Cauchy-Schwarz inequality, we infer from
\eqref{rep} and \eqref{rep2} that
   \begin{align} \label{uncov}
\sup\{|f(\xi)|\colon \xi \in \hh,\, \|\xi\| \Le R\} \Le
\|f\| \sqrt{\varPhi(R^2)}, \quad R \in (0,\infty).
   \end{align}
If $\dim \hh \Ge 1$, then by \eqref{kfxi} the
functions $K^{\varPhi}_{\xi}$, $\xi \in \hh$, are
holomorphic. Since the linear span $\kscr^{\varPhi}$
of the set $\{K^{\varPhi}_{\xi}\colon \xi \in \hh\}$
is dense in $\varPhi(\hh)$ (see \cite{Aronsz0} or use
\eqref{rep}), we infer from \eqref{uncov} and
\cite[Theorem 14.16]{chae} that $\varPhi(\hh)$
consists of holomorphic functions.

Let us recall the RKHS test (cf.\ \cite{FHSz,FHSz4}).
   \begin{thm} \label{rkhst}
A function $f\colon \hh \to \cbb$ belongs to
$\varPhi(\hh)$ if and only if there exists $c \in
[0,\infty)$ such that for all $n \in \nbb$,
$\{\lambda_i\}_{i=1}^n \subseteq \cbb$ and
$\{\xi_i\}_{i=1}^n \subseteq \hh$,
   \begin{align} \label{rkhst+}
\Big|\sum_{i=1}^n \lambda_i f(\xi_i)\Big|^2 \Le c
\sum_{i,j=1}^n K^{\varPhi}(\xi_i,\xi_j) \lambda_i \bar
\lambda_j.
   \end{align}
The smallest such $c$ is equal to $\|f\|^2$.
   \end{thm}
The next three propositions will be proved by using
the RKHS test.
   \begin{pro} \label{stale}
Let $\varPhi \in \fscr$. If $\varPhi(0)=0$, then there
is no nonzero constant function in $\varPhi(\hh)$. If
$\varPhi(0)\neq 0$, then constant functions on $\hh$
belong to $\varPhi(\hh)$ and
   \begin{align} \label{njed}
\|\mathbb{1}\| = \frac{1}{\sqrt{\varPhi(0)}},
   \end{align}
where $\mathbb{1}(\xi) = 1$ for every $\xi \in \hh$.
   \end{pro}
   \begin{proof}
If $\varPhi(0)=0$, then by \eqref{rep2} we have
$\|K^{\varPhi}_0\|^2 = \varPhi(0) =0$, which together with
\eqref{rep} implies that $f(0)=0$ for every $f\in
\varPhi(\hh)$. Hence there is no nonzero constant function
in $\varPhi(\hh)$. In turn, if $\varPhi(0)\neq 0$, then
   \begin{align*}
\Big|\sum_{i=1}^n \lambda_i\Big|^2 \Le \frac{1}{\varPhi(0)}
\sum_{k=0}^\infty a_k \sum_{i,j=1}^n \is{\xi_i}{\xi_j}^k
\lambda_i \bar \lambda_j = \frac{1}{\varPhi(0)}
\sum_{i,j=1}^n K^{\varPhi}(\xi_i,\xi_j) \lambda_i \bar
\lambda_j
   \end{align*}
for all finite sequences $\{\lambda_i\}_{i=1}^n
\subseteq \cbb$ and $\{\xi_i\}_{i=1}^n \subseteq \hh$,
where $\{a_k\}_{k=0}^\infty$ is as in \eqref{wsp}. It
follows from Theorem \ref{rkhst} that $\mathbb{1} \in
\varPhi(\hh)$ (hence, any constant function on $\hh$
belongs to $\varPhi(\hh)$) and $\|\mathbb{1}\|\Le
1/\sqrt{\varPhi(0)}$. To prove the reverse inequality,
set $c=\|\mathbb{1}\|^2$ and apply \eqref{rkhst+} to
$f=\mathbb{1}$, $n=1$, $\lambda_1=1$ and $\xi_1=0$.
What we get is $\|\mathbb{1}\|\Ge
1/\sqrt{\varPhi(0)}$. This completes the proof.
   \end{proof}
   \begin{pro}
If $\varPhi \in \fscr$ is as in \eqref{wsp}, then for
every $n\in \nbb$ and for all $\eta_1, \ldots, \eta_n
\in \hh$ and $k_1, \ldots, k_n \in \nbb$ such that
$a_{k_1+ \ldots+k_n} > 0$, $\is{\cdot}{\eta_1}^{k_1}
\cdots \is{\cdot}{\eta_n}^{k_n} \in \varPhi(\hh)$~ and
   \begin{align*}
\|\is{\cdot}{\eta_1}^{k_1} \cdots
\is{\cdot}{\eta_n}^{k_n}\|^2 \Le
\frac{\|\eta_1\|^{2k_1} \cdots
\|\eta_n\|^{2k_n}}{a_{k_1+ \ldots+k_n}}.
   \end{align*}
   \end{pro}
   \begin{proof}
Set $k=k_1+ \ldots+k_n$. Then
   \begin{align*}
\Big|\sum_{i=1}^m \lambda_i \is{\xi_i}{\eta_1}^{k_1}
\cdots \is{\xi_i}{\eta_n}^{k_n}\Big|^2 & =
\Big|\Big\langle \sum_{i=1}^m \lambda_i \xi_i^{\otimes
k}, \eta_1^{\otimes k_1} \otimes \cdots \otimes
\eta_n^{\otimes k_n}\Big\rangle\Big|^2
   \\
& \Le \|\eta_1\|^{2k_1} \cdots \|\eta_n\|^{2k_n}
\sum_{i,j=1}^m \is{\xi_i}{\xi_j}^k \lambda_i \bar
\lambda_j
   \\
& \Le \frac{\|\eta_1\|^{2k_1} \cdots
\|\eta_n\|^{2k_n}}{a_k} \sum_{i,j=1}^n
K^{\varPhi}(\xi_i,\xi_j) \lambda_i \bar \lambda_j
   \end{align*}
for all $m\in \nbb$, $\{\lambda_i\}_{i=1}^m \subseteq
\cbb$ and $\{\xi_i\}_{i=1}^m \subseteq \hh$, which by
Theorem \ref{rkhst} completes the proof.
   \end{proof}
   \begin{pro} \label{isintin+}
Suppose $\varPhi \in \fscr$, $W\in \ogr{\hh}$ is a
coisometry and $f\colon \hh \to \cbb$ is a function
such that $f\circ W \in \varPhi(\hh)$. Then $f \in
\varPhi(\hh)$ and $\|f\circ W\|=\|f\|$.
   \end{pro}
   \begin{proof}
By Theorem \ref{rkhst}, we have
   \begin{align*}
\Big|\sum_{i=1}^n \lambda_i f(\xi_i)\Big|^2 & =
\Big|\sum_{i=1}^n \lambda_i (f\circ W)
(W^*\xi_i)\Big|^2
   \\
& \Le \|f\circ W\|^2 \sum_{i,j=1}^n
K^{\varPhi}(W^*\xi_i,W^*\xi_j) \lambda_i \bar
\lambda_j
   \\
& \hspace{-.9ex}\overset{\eqref{ajjaj}} = \|f\circ
W\|^2 \sum_{i,j=1}^n K^{\varPhi}(\xi_i,\xi_j)
\lambda_i \bar \lambda_j,
   \end{align*}
for all $n \in \nbb$, $\{\lambda_i\}_{i=1}^n \subseteq
\cbb$ and $\{\xi_i\}_{i=1}^n \subseteq \hh$. Applying
Theorem \ref{rkhst} again shows that $f \in
\varPhi(\hh)$. Hence, by Corollary \ref{unit}, we have
$\|f\circ W\|=\|f\|$.
   \end{proof}
Now we give some examples of reproducing kernel
Hilbert spaces of the form $\varPhi(\hh)$ for
$\hh=\cbb$ that can be regarded as closed subspaces of
$L^2$-spaces. We refer the reader to \cite{FHSz2} for
the study of the multidimensional case.
   \begin{exa} \label{phih}
Let $\nu$ be a positive Borel measure on $\rbb_+$ such
that
   \begin{align} \label{maj1}
\text{$\int_{\rbb_+} t^{n} \D \nu(t) < \infty$ and
$\nu((c,\infty)) > 0$ for all $n \in \zbb_+$ and $c\in
(0,\infty)$.}
   \end{align}
Define the positive Borel measure $\mu$ on $\cbb$ by
   \begin{align*}
\mu(\varDelta) = \frac{1}{2\pi} \int_0^{2\pi}
\int_{\rbb_+} \chi_{\varDelta} (r\E^{\I \theta}) \D
\nu(r) \D \theta, \quad \varDelta \text{ is a Borel
subset of } \cbb.
   \end{align*}
It is a matter of routine to show that $p\in L^2(\mu)$
for every complex polynomial $p$ in one complex
variable. Since, by \eqref{maj1}, $\int_{\rbb_+} t^{n}
\D \nu(t) \in (0,\infty)$ for every $n \in \zbb_+$, we
can define the function $\varPhi\colon \cbb \to \cbb$
by
   \begin{align} \label{pser}
\varPhi(z) = \sum_{n=0}^{\infty}
\frac{1}{\int_{\rbb_+} t^{2n} \D \nu(t)} \, z^n, \quad
z \in \cbb.
   \end{align}
It follows from \eqref{maj1} and \cite[Exercise 4(e),
Chapter 3]{Rud1} that the radius of convergence of the
power series in \eqref{pser} is equal to
   \begin{align*}
\frac{1}{\limsup_{n\to \infty}
\frac{1}{\sqrt[n]{\int_{\rbb_+} t^{2n} \D \nu(t)}}} =
\lim_{n\to \infty} \sqrt[n]{\int_{\rbb_+} t^{2n} \D
\nu(t)} = \infty.
   \end{align*}
As a consequence, $\varPhi \in \fscr$. It turns out
that the reproducing kernel Hilbert space
$\varPhi(\cbb)$ can be described as follows
   \begin{align} \label{maj3}
\varPhi(\cbb) = \big\{f \colon f \text{ is an entire
function on $\cbb$ and } f \in L^2(\mu)\big\}.
   \end{align}
This fact was proved in \cite{fr1,fr2} (see also
\cite{fr3}) under the additional assumption that
$\nu(\{0\})=0$. However, arguing as in \cite[Example
18]{J-St} (with emphases put on \cite[Eq.\
(40)]{J-St}), one can show that \eqref{maj3} remains
true without assuming that $\nu(\{0\})=0$. This means
that the right-hand side of \eqref{maj3} is a
reproducing kernel Hilbert space with the reproducing
kernel $\cbb \times \cbb \ni (\xi,\eta) \mapsto
\varPhi(\xi \bar \eta) \in \cbb$, where $\varPhi$ is
given by \eqref{pser}. In particular, if $\nu$ is the
positive Borel measure on $\rbb_+$ given by
   \begin{align*}
\nu(\varDelta) = 2 \int_{\varDelta} s \E^{-s^2} \D s,
\quad \varDelta \text{ is a Borel subset of } \rbb_+,
   \end{align*}
then $\int_{\rbb_+} t^{2n} \D \nu(t)=n!$ for every
$n\in \zbb_+$. This implies that $\varPhi=\exp$ and
consequently $\exp(\cbb)$ can be regarded as the
Segal-Bargmann space $\bb_1$ (cf.\ Section~
\ref{sec9}).

The results of the present paper can be applied to
$\varPhi(\cbb)$ with $\varPhi$ given by \eqref{pser}.
In particular, by Proposition \ref{pro1} and Corollary
\ref{wn3}, bounded hyponormal composition operators on
$\varPhi(\cbb)$ with holomorphic symbols are always
normal.
   \end{exa}
It follows from Proposition \ref{stale} that
   \begin{align*}
\dim \varPhi(\{0\}) =
   \begin{cases}
0 & \text{if } \varPhi(0) = 0,
   \\[1ex]
1 & \text{if } \varPhi(0) \neq 0,
   \end{cases}
\quad \varPhi \in \fscr.
   \end{align*}
From now on, to avoid trivial cases, we make the
following
   \begin{align} \label{SA}
   \begin{minipage}{70ex} {\sc Standing assumption:}
$\dim \hh \Ge 1$.
   \end{minipage}
   \end{align}
   \section{\label{sekt4}Introducing $C_{\varphi}$}
   Given a holomorphic mapping $\varphi\colon \hh \to
\hh$, we define the operator $C_{\varphi}$ in
$\varPhi(\hh)$, called a {\em composition operator}
with a {\em symbol} $\varphi$, by
   \begin{align*}
\dz{C_{\varphi}}&=\{f\in\varPhi(\hh)\colon f \circ \varphi
\in \varPhi(\hh)\},
   \\
C_{\varphi} f &= f \circ \varphi, \quad f \in
\dz{C_{\varphi}}.
   \end{align*}
We begin by stating two simple properties of composition
operators.
   \begin{pro} \label{stale2}
Let $\varphi\colon \hh \to \hh$ be a holomorphic mapping
and $\varPhi \in \fscr$. Then
   \begin{enumerate}
   \item[(i)] $C_{\varphi}$ is closed,
   \item[(ii)] if
$\varPhi(0)\neq 0$ and $C_{\varphi} \in
\ogr{\varPhi(\hh)}$, then $r(C_{\varphi}) \Ge 1$ and thus
$\|C_{\varphi}\| \Ge 1$.
   \end{enumerate}
   \end{pro}
   \begin{proof}
(i) This can be deduced from the reproducing property
\eqref{rep}.

(ii) By Proposition \ref{stale}, the function
$\mathbb{1}$ is an eigenvector of $C_{\varphi}$
corresponding to the eigenvalue $1$.
   \end{proof}
It is worth mentioning that the assertion (ii) of
Proposition \ref{stale2} is no longer true if
$\varPhi(0) = 0$ (e.g., one can apply Theorem
\ref{bca} to $\varPhi(z)=z\exp(z)$).

For self-containedness, we include the proof of the
following description of the adjoint of $C_{\varphi}$
which is true for composition operators acting in
arbitrary reproducing kernel Hilbert spaces with
arbitrary symbols.
   \begin{thm}\label{sprz}
Let $\varphi\colon \hh \to \hh$ be a holomorphic
mapping and $\varPhi \in \fscr$. Then
   \begin{enumerate}
   \item[(i)] if $C_{\varphi}$ is densely defined, then
$\kscr^{\varPhi}$ is a core for $C_{\varphi}^*$ and
   \begin{align} \label{op*}
C_{\varphi}^*(K^{\varPhi}_{\xi}) =
K^{\varPhi}_{\varphi(\xi)}, \quad \xi \in \hh,
   \end{align}
   \item[(ii)] if there exists an operator
$J_{\varphi}$ in $\varPhi(\hh)$ such that
   \begin{align} \label{Jurek}
\text{$\dz{J_{\varphi}} = \kscr^{\varPhi}$ and $J_{\varphi}
(K^{\varPhi}_{\xi}) = K^{\varPhi}_{\varphi(\xi)}$ for all
$\xi \in \hh$,}
   \end{align}
then $C_{\varphi} = J_{\varphi}^*$; such $J_{\varphi}$
is unique, and it is closable if and only if
$C_{\varphi}$ is densely defined.
   \end{enumerate}
Moreover, if $C_{\varphi}$ is densely defined, then
$J_{\varphi} := C_{\varphi}^*|_{\kscr^{\varPhi}}$ is
closable and satisfies \eqref{Jurek}.
   \end{thm}
   \begin{proof}
(i) Suppose $C_{\varphi}$ is densely defined. Since for
every $f \in \dz{C_{\varphi}}$,
    \begin{align*}
\is{C_{\varphi}f}{K^{\varPhi}_{\xi}}\overset{\eqref{rep}}=
(C_{\varphi}f)(\xi)=f(\varphi(\xi))\overset{\eqref{rep}}
=\is{f}{K^{\varPhi}_{\varphi(\xi)}}, \quad \xi \in \hh,
    \end{align*}
we deduce that $\kscr^{\varPhi} \subseteq
\dz{C_{\varphi}^*}$ and \eqref{op*} holds. Set $J_{\varphi}
= C_{\varphi}^*|_{\kscr^{\varPhi}}$. Then clearly
$J_{\varphi}$ is closable and densely defined, and thus, by
the von Neumann theorem, $J_{\varphi}^*$ is densely defined
and $J_{\varphi}^{**}=\overline{J_{\varphi}}$. If $f \in
\dz{J_{\varphi}^*}$, then
   \begin{align*}
(J_{\varphi}^*f)(\xi) \overset{\eqref{rep}}=
\is{J_{\varphi}^*f}{K^{\varPhi}_{\xi}} =
\is{f}{J_{\varphi}K^{\varPhi}_{\xi}} \overset{\eqref{op*}}=
\is{f}{K^{\varPhi}_{\varphi(\xi)}} \overset{\eqref{rep}}=
f(\varphi(\xi)), \quad \xi\in \hh,
   \end{align*}
which implies that $f \in\dz{C_{\varphi}}$ and
$J_{\varphi}^* f = C_{\varphi}f$. Hence $J_{\varphi}^*
\subseteq C_{\varphi}$. This yields
   \begin{align*}
\overline{J_{\varphi}} \subseteq C_{\varphi}^*
\subseteq J_{\varphi}^{**} = \overline{J_{\varphi}},
   \end{align*}
which means that $\kscr^{\varPhi}$ is a core for
$C_{\varphi}^*$.

(ii) Take $f \in \varPhi(\hh)$. If $h = \sum_{i=1}^n
\lambda_i K^{\varPhi}_{\xi_i}$ for some finite sequences
$\{\lambda_i\}_{i=1}^n \subseteq \cbb$ and
$\{\xi_i\}_{i=1}^n \subseteq \hh$, then, by \eqref{rep} and
\eqref{Jurek}, we have
   \begin{align} \label{tak2}
\is{J_{\varphi} h}{f} = \sum_{i=1}^n \lambda_i
\overline{f(\varphi(\xi_i))}.
   \end{align}
Note that $f$ is in $\dz{J_{\varphi}^*}$ if and only
if there exists $c \in (0,\infty)$ such that
$|\is{J_{\varphi} h}{f}|^2 \Le c \|h\|^2$ for all $h
\in \dz{J_{\varphi}}$, or equivalently, by
\eqref{rep2} and \eqref{tak2}, if and only if
   \begin{align*}
\Big|\sum_{i=1}^n \bar \lambda_i f(\varphi(\xi_i))\Big|^2
\Le c \sum_{i,j=1}^n K^{\varPhi}(\xi_j,\xi_i) \lambda_i
\bar \lambda_j, \quad \{\lambda_i\}_{i=1}^n \subseteq \cbb,
\, \{\xi_i\}_{i=1}^n \subseteq \hh, \, n \in \nbb.
   \end{align*}
In view of Theorem \ref{rkhst}, the latter is
equivalent to $f \circ \varphi \in \varPhi(\hh)$. This
shows that $\dz{C_{\varphi}} = \dz{J_{\varphi}^*}$.
Moreover, by \eqref{rep} and \eqref{tak2}, we have
   \begin{align*}
\is{J_{\varphi} h}{f} = \is{h}{f\circ \varphi}, \quad h \in
\dz{J_{\varphi}}, \, f \in \dz{C_{\varphi}},
   \end{align*}
which implies that $C_{\varphi} = J_{\varphi}^*$. The
uniqueness of an operator $J_{\varphi}$ satisfying
\eqref{Jurek} is obvious. The fact that such $J_{\varphi}$
is closable if and only if $C_{\varphi}$ is densely defined
is a direct consequence of the von Neumann theorem. This
completes the proof.
   \end{proof}
   \begin{cor}\label{hypo}
Suppose $\varPhi \in \fscr$ and $\varphi\colon \hh \to \hh$
is a holomorphic mapping such that $C_\varphi$ is
hyponormal. Then $\varphi(0)=0$.
   \end{cor}
   \begin{proof}
Since
   \begin{align*}
K^{\varPhi}_0(\varphi(\xi)) = \varPhi(\is{\varphi(\xi)}{0})
= \varPhi(0) = \varPhi(\is{\xi}{0}) = K^{\varPhi}_0(\xi),
\quad \xi \in \hh,
   \end{align*}
we deduce that $K^{\varPhi}_{0} \in \dz{C_{\varphi}}$ and
$C_{\varphi} K^{\varPhi}_0 = K^{\varPhi}_0$. By Theorem
\ref{sprz}(i), we have
   \begin{align*}
\varPhi(\|\varphi(0)\|^2) = \|K^{\varPhi}_{\varphi(0)}\|^2
= \|C_{\varphi}^* K^{\varPhi}_{0}\|^2 \Le \|C_{\varphi}
K^{\varPhi}_{0}\|^2 = \|K^{\varPhi}_{0}\|^2 = \varPhi (0).
   \end{align*}
Since the function $\varPhi|_{[0,\infty)}$ is strictly
increasing, we get $\|\varphi(0)\|=0$, which completes
the proof.
   \end{proof}
Regarding Theorem \ref{sprz}(ii), it is worth pointing
out that if the composition operator $C_{\varphi}$ is
not densely defined, then the operator $J_{\varphi}$
satisfying \eqref{Jurek} may not exist. This is
illustrated by Proposition \ref{nonzero} below which
deals with the question of when $C_{\varphi}$ is a
zero operator. We also give an exact description of
the norm of the composition operator with a constant
symbol.
   \begin{pro} \label{nonzero}
Let $\varphi\colon \hh \to \hh$ be a holomorphic
mapping and $\varPhi \in \fscr$. Then
   \begin{enumerate}
   \item[(i)] if $\varPhi(0) \neq 0$ and $C_{\varphi}$
is densely defined, then $C_{\varphi}$ is not a zero
operator,
   \item[(ii)] if $\varPhi(0) = 0$, then $C_{\varphi}$
is a densely defined zero operator if and only if
$\varphi(\xi) = 0$ for every $\xi \in \hh$; if this is
the case, then $\dz{C_{\varphi}}=\varPhi(\hh)$,
   \item[(iii)] if $\varPhi(0) = 0$ and $\varphi(\xi) = a$
for every $\xi \in \hh$ and for some $a \in \hh
\setminus \{0\}$, then $C_{\varphi}$ is a zero
operator, $\dz{C_{\varphi}} =
\{K^{\varPhi}_{a}\}^\perp$, $K^{\varPhi}_{a} \neq 0$
and there is no operator $J_{\varphi}$ in
$\varPhi(\hh)$ satisfying \eqref{Jurek},
   \item[(iv)] if $\varPhi(0) \neq 0$ and $\varphi(\xi) = a$
for every $\xi \in \hh$ and for some $a \in \hh$, then
$C_{\varphi} \in \ogr{\varPhi(\hh)}$, $\mathbb{1} \in
\varPhi(\hh)$, $C_{\varphi} f = f(a) \cdot \mathbb{1}$
for every $f \in \varPhi(\hh)$, and
   \begin{align*}
\|C_{\varphi}\| =
\sqrt{\frac{\varPhi(\|a\|^2)}{\varPhi(0)}}.
   \end{align*}
   \end{enumerate}
   \end{pro}
   \begin{proof}
(i)\&(ii) Suppose $C_{\varphi}$ is densely defined and
$C_{\varphi} f = 0$ for every $f \in
\dz{C_{\varphi}}$. Then $C_{\varphi}^*$ vanishes on
$\varPhi(\hh)$ and consequently
   \begin{align} \label{Jurobudz}
\sum_{n=0}^\infty a_n \|\varphi(\xi)\|^{2n}
\overset{\eqref{rep2}}= \|K^{\varPhi}_{\varphi (\xi)}\|^2
\overset{\eqref{op*}}= \|C_{\varphi}^*(K^{\varPhi}_{\xi})
\|^2=0, \quad \xi \in \hh,
   \end{align}
where $\{a_n\}_{n=0}^\infty$ is as in \eqref{wsp}. If
$\varPhi(0)\neq 0$, we arrive at a contradiction
(because $\varPhi(0)=a_0$). If $\varPhi(0) = 0$, then
\eqref{Jurobudz} implies that $\varphi(\xi) = 0$ for
every $\xi \in \hh$.

Suppose now that $\varPhi(0) = 0$ and $\varphi(\xi) =
0$ for every $\xi \in \hh$. Then, by \eqref{rep2},
$K^{\varPhi}_{0}=0$ and thus, by \eqref{rep}, $f(0)=0$
for every $f \in \varPhi(\hh)$, which means that
$\dz{C_{\varphi}} = \varPhi(\hh)$ and $C_{\varphi} f =
0$ for every $f \in \varPhi(\hh)$.

   (iii) If $f \in \dz{C_{\varphi}}$, then $f(a) \cdot
\mathbb{1} = C_{\varphi} f \in \varPhi(\hh)$ and thus,
by Proposition \ref{stale}, $f(a)=0$, or equivalently,
by \eqref{rep}, $f \in \{K^{\varPhi}_{a}\}^\perp$. The
converse implication holds as well. Since, by
\eqref{rep2}, $\|K^{\varPhi}_{0}\|^2 = \varPhi(0)=0$
and $\|K^{\varPhi}_{a}\|^2 = \varPhi(\|a\|^2) > 0$, we
see that $K^{\varPhi}_{0}=0$ and
$K^{\varPhi}_{\varphi(0)} \neq 0$. Hence, there is no
operator $J_{\varphi}$ satisfying \eqref{Jurek}.

   (iv) Using Proposition \ref{stale}, we see that
$\mathbb{1} \in \varPhi(\hh)$ and
   \begin{align*}
C_{\varphi} f = f(a) \cdot \mathbb{1}
\overset{\eqref{rep}} = \is{f}{K^{\varPhi}_a} \cdot
\mathbb{1}, \quad f \in \varPhi(\hh),
   \end{align*}
which in the standard operator theory notation means
that $C_{\varphi} = \mathbb{1} \otimes K^{\varPhi}_a$
(cf.\ \cite[p.\ 72]{Con}). This together with
\eqref{rep2}, \eqref{njed} and \cite[Proposition
16.3(d)]{Con} yields
   \begin{align*}
& \|C_{\varphi}\| = \|\mathbb{1}\| \|K^{\varPhi}_a\| =
\sqrt{\frac{\varPhi(\|a\|^2)}{\varPhi(0)}}. \qedhere
   \end{align*}
   \end{proof}
Under the assumptions of Proposition
\ref{nonzero}(iii), the orthocomplement of the domain
of $C_{\varphi}$ is one-dimensional and there is no
operator $J_{\varphi}$ in $\varPhi(\hh)$ satisfying
\eqref{Jurek}. However, if $\varPhi(\hh)$ is infinite
dimensional, then there are plenty of operators $T$ in
$\varPhi(\hh)$ such that $\dz{T} = \kscr^{\varPhi}$
and $C_{\varphi} = T^*$ (compare with Theorem
\ref{sprz}(ii)). This is a consequence of the
following more general result.
   \begin{pro}
Let $A$ be a zero operator in a complex Hilbert space
$\hh$ such that $\dz{A} = \{e\}^{\perp}$ for some
nonzero $e \in \hh$. Then
   \begin{enumerate}
   \item[(i)] if $T$ is a densely defined
operator in $\hh$, then $A=T^*$ if and only if there
exists a discontinuous linear functional $\tau\colon
\dz{T} \to \cbb$ such that $T\xi = \tau (\xi) e$ for
every $\xi \in \dz{T}$,
   \item[(ii)]  if $\hh$ is infinite dimensional and
$\EuScript E$ is a dense vector subspace of $\hh$,
then there exists an operator $T$ in $\hh$ such that
$\dz{T} = \EuScript E$ and $A=T^*$.
   \end{enumerate}
   \end{pro}
   \begin{proof} (i) Suppose $A=T^*$. If $\eta \in
\{e\}^{\perp} = \dz{T^*}$, then $\is{T\xi}{\eta} = 0$
for all $\xi \in \dz{T}$. Hence $\{e\}^{\perp}
\subseteq \ob{T}^{\perp}$, or equivalently, $\ob{T}
\subseteq \cbb \cdot e$. Thus for every $\xi\in
\dz{T}$ there exists a unique $\tau(\xi) \in \cbb$
such that $T\xi = \tau(\xi) e$. Clearly, $\tau$ is a
linear functional which is not continuous (otherwise,
$\dz{T^*} = \hh$, which contradicts our assumption).
To prove the ``if'' part note that $\eta \in \dz{T^*}$
if and only if the functional $\dz{T} \ni \xi \mapsto
\is{T\xi}{\eta} = \tau (\xi) \is{e}{\eta} \in \cbb$ is
continuous.

   (ii) This is a direct consequence of (i) and the
fact that each infinite dimensional normed space
admits a discontinuous linear functional.
   \end{proof}
   \section{\label{sekt5}Maximality of $C_\varphi$}
   In this section we show that members of the class
of densely defined composition operators in
$\varPhi(\hh)$ with holomorphic symbols are maximal
with respect to inclusion of graphs (see Theorem
\ref{fipsi} below). First, we need the following two
lemmata.
   \begin{lem} \label{symtensor}
If $a,b \in \hh$ and $n\in \nbb$, then the following
conditions are equivalent{\em :}
   \begin{enumerate}
   \item[(i)] $\is{\eta}{a}^n =  \is{\eta}{b}^n$
for all $\eta \in \hh$,
   \item[(ii)] there exists $\delta \in \cbb$
such that $\delta^n = 1$ and $a = \delta \cdot b$.
   \end{enumerate}
   \end{lem}
   \begin{proof}
Suppose (i) holds. It is clear that
$\{a\}^{\perp}=\{b\}^{\perp}$. Hence
   \begin{align*}
\cbb \cdot a = \{a\}^{\perp\perp} = \{b\}^{\perp\perp} =
\cbb \cdot b.
   \end{align*}
If $a=0$, then $b=0$ and thus $\delta =1$ does the
job. Otherwise, $a \neq 0$, $b\neq 0$ and $a=\delta
\cdot b$ for some $\delta \in \cbb\setminus \{0\}$.
Substituting $\eta = b$ into (i), we get $\bar
\delta^n \|b\|^{2n} = \|b\|^{2n}$, which implies that
(ii) holds. The reverse implication
(ii)$\Rightarrow$(i) is obvious.
   \end{proof}
   \begin{lem} \label{spojne}
If $\varphi\colon \hh \to \cbb$ is holomorphic, then the
set $\varphi^{-1}(\cbb\setminus \{0\})$ is connected.
   \end{lem}
   \begin{proof}
Take $a,b \in \varphi^{-1}(\cbb \setminus \{0\})$ such that
$a \neq b$. Let $\hh_{a,b}$ be the vector space spanned by
$\{a, b\}$. Since $\hh_{a,b}$ is finite dimensional and
$\varphi|_{\hh_{a,b}}$ is holomorphic, we infer from
\cite[Proposition 2.2.3]{Chirka} that
$\varOmega_{a,b}:=\{\xi \in \hh_{a,b} \colon
\varphi|_{\hh_{a,b}} (\xi) \neq 0\}$ is connected and thus
there exists a continuous path $f\colon [0,1] \to
\varOmega_{a,b}$ such that $f(0)=a$ and $f(1)=b$. Since
$\varOmega_{a,b} \subseteq \varphi^{-1}(\cbb \setminus
\{0\})$, we deduce that $\varphi^{-1}(\cbb \setminus
\{0\})$ is path-connected and hence connected.
    \end{proof}
   Now we are ready to prove the main result of this
section.
   \begin{thm}\label{fipsi}
Let $\varPhi \in \fscr$ and let $\varphi, \psi\colon
\hh \to \hh$ be holomorphic mappings. Assume that the
operator $C_{\varphi}$ is densely defined. Then the
following conditions are equivalent{\em :}
   \begin{enumerate}
   \item[(i)] $C_{\varphi} \subseteq C_{\psi}$,
   \item[(ii)] $C_{\varphi} = C_{\psi}$,
   \item[(iii)] there exists $\alpha \in \gfrak_{\varPhi}$
such that $\varphi(\xi) = \alpha \cdot \psi(\xi)$ for every
$\xi \in \hh$.
   \end{enumerate}
   \end{thm}
   \begin{proof}
Since $\varPhi \in \fscr$, the set $\zscr_{\varPhi}
\setminus \{0\}$ is nonempty.

(i)$\Rightarrow$(iii) It follows from (i) that
$C_{\psi}^* \subseteq C_{\varphi}^*$ and thus, by
Theorem \ref{sprz}(i), $\kscr^{\varPhi} \subseteq
\dz{C_{\varphi}^*} \cap \dz{C_{\psi}^*}$ and
   \begin{align*}
K^{\varPhi}_{\psi(\xi)} = C_{\psi}^* K^{\varPhi}_{\xi} =
C_{\varphi}^* K^{\varPhi}_{\xi} =
K^{\varPhi}_{\varphi(\xi)}, \quad \xi \in \hh.
   \end{align*}
This implies that
   \begin{align*}
\varPhi(\is{\eta}{\psi(\xi)}) =
\is{K^{\varPhi}_{\psi(\xi)}}{K^{\varPhi}_{\eta}} =
\is{K^{\varPhi}_{\varphi(\xi)}}{K^{\varPhi}_{\eta}} =
\varPhi(\is{\eta}{\varphi(\xi)}), \quad \xi, \eta \in \hh.
   \end{align*}
Replacing $\eta$ by $z\eta$, we get
   \begin{align*}
\sum_{n\in \zscr_{\varPhi}} a_n \is{\eta}{\psi(\xi)}^n
z^n= \sum_{n\in \zscr_{\varPhi}} a_n
\is{\eta}{\varphi(\xi)}^n z^n, \quad z \in \cbb, \,
\xi, \eta \in \hh,
   \end{align*}
where the sequence $\{a_k\}_{k=0}^\infty$ is as in
\eqref{wsp}. Hence, we have
   \begin{align}   \label{nnn}
\is{\eta}{\psi(\xi)}^n = \is{\eta}{\varphi(\xi)}^n,
\quad \xi, \eta \in \hh, n \in \zscr_{\varPhi}
\setminus \{0\}.
   \end{align}
This implies that $\varphi^{-1}(\hh\setminus \{0\}) =
\psi^{-1}(\hh\setminus \{0\})$. If
$\varphi^{-1}(\hh\setminus \{0\}) = \varnothing$, then
$\alpha=1$ does the job. Suppose
$\varphi^{-1}(\hh\setminus \{0\}) \neq \varnothing$.
By \eqref{nnn} and Lemma \ref{symtensor}, for every
$\xi \in \varphi^{-1}(\hh\setminus \{0\})$, there
exists a unique $\delta(\xi) \in \gfrak_{\varPhi}$
such that
   \begin{align} \label{mmp}
\psi(\xi) = \delta(\xi) \cdot \varphi(\xi), \quad \xi \in
\varphi^{-1}(\hh\setminus \{0\}).
   \end{align}

Our next aim is to show that the open set
$\varphi^{-1}(\hh\setminus \{0\})$ is connected. Take
$a,b \in \varphi^{-1}(\hh\setminus \{0\})$ such that
$a\neq b$. Note that there exists $\eta \in \hh$ such
that $\is{\varphi(a)}{\eta} \neq 0$ and
$\is{\varphi(b)}{\eta} \neq 0$. Indeed, if
$\is{\varphi(a)}{\varphi(b)} = 0$, then we set
$\eta=\varphi(a) + \varphi(b)$. Otherwise, we set
$\eta = \varphi(b)$. Since $\is{\varphi(\cdot)}{\eta}$
is a holomorphic function, we infer from Lemma
\ref{spojne} that there exists a continuous path
$f\colon [0,1] \to \{\xi \in \hh\colon
\is{\varphi(\xi)}{\eta} \neq 0\}$ such that $f(0)=a$
and $f(1)=b$. Since $\{\xi \in \hh\colon
\is{\varphi(\xi)}{\eta} \neq 0\} \subseteq
\varphi^{-1}(\hh\setminus \{0\})$, we conclude that
$\varphi^{-1}(\hh \setminus \{0\})$ is path-connected
and thus connected.

Now we show that $\delta(\cdot)$ is a holomorphic function.
Indeed, \eqref{mmp} yields
   \begin{align*}
\delta(\xi) =
\frac{\is{\psi(\xi)}{\eta}}{\is{\varphi(\xi)}{\eta}},
\quad \xi \in \varOmega _{\eta}:= \{\zeta \in \hh
\colon \is{\varphi(\zeta)}{\eta} \neq 0\}, \, \eta \in
\hh.
   \end{align*}
Since $\varphi^{-1}(\hh\setminus \{0\}) =
\bigcup_{\eta \in \hh} \varOmega_{\eta}$ and each
$\varOmega_{\eta}$ is open, we deduce that the
function $\delta(\cdot)$ is continuous and
$G$-holomorphic on $\varphi^{-1}(\hh \setminus
\{0\})$; thus by \cite[Theorem 14.9]{chae} it is
holomorphic.

Recall that $\delta(\xi) \in \gfrak_{\varPhi}$ for all
$\xi \in \varphi^{-1}(\hh\setminus \{0\})$. Hence
$|\delta(\cdot)|=1$ on $\varphi^{-1}(\hh\setminus
\{0\})$. As $\delta(\cdot)$ is holomorphic and
$\varphi^{-1}(\hh\setminus \{0\})$ is connected, the
maximum modulus principle (cf.\ \cite[Corollary
13.9]{chae}) implies that there exists $\alpha\in
\cbb$ such that $\delta(\xi)=\bar \alpha$ for every
$\xi \in \varphi^{-1}(\hh\setminus \{0\})$. Since
$\varphi^{-1}(\hh\setminus \{0\}) =
\psi^{-1}(\hh\setminus \{0\})$, we infer from
\eqref{mmp} that (iii) holds.

(iii)$\Rightarrow$(ii) Note that
   \begin{align*}
\varPhi(\is{\eta}{\varphi(\xi)}) = \sum_{n \in
\zscr_{\varPhi}} a_n \bar \alpha^n
\is{\eta}{\psi(\xi)}^n = \sum_{n=0}^\infty a_n
\is{\eta}{\psi(\xi)}^n =
\varPhi(\is{\eta}{\psi(\xi)}), \quad \xi, \eta \in
\hh.
   \end{align*}
Therefore, by \eqref{rep2} and \eqref{op*}, we have
   \begin{align} \label{Zakop}
\is{C_{\varphi}^*
(K^{\varPhi}_{\xi})}{K^{\varPhi}_{\eta}} =
\is{K^{\varPhi}_{\varphi(\xi)}}{K^{\varPhi}_{\eta}} =
\is{K^{\varPhi}_{\psi(\xi)}}{K^{\varPhi}_{\eta}},
\quad \xi, \eta \in \hh.
   \end{align}
This implies that there exists an operator $J_{\psi}$
in $\varPhi(\hh)$ which satisfies \eqref{Jurek} with
$\psi$ in place of $\varphi$, and such that $J_{\psi}
\subseteq C_{\varphi}^*$. Hence $J_{\psi}$ is
closable. By Theorem \ref{sprz}(ii), $C_{\psi}$ is
densely defined. It follows from \eqref{op*} (applied
to $\psi$) and \eqref{Zakop} that
   \begin{align*}
C_{\varphi}^*(K^{\varPhi}_{\xi}) =
C_{\psi}^*(K^{\varPhi}_{\xi}), \quad \xi \in \hh.
   \end{align*}
Since, by Theorem \ref{sprz}(i), $\kscr^{\varPhi}$ is
a core for $C_{\varphi}^*$ and $C_{\psi}^*$, we deduce
that $C_{\varphi}^* = C_{\psi}^*$. By Proposition
\ref{stale2}(i) and the von Neumann theorem, we have
$C_{\varphi} = C_{\varphi}^{**} = C_{\psi}^{**} =
C_{\psi}$.

(ii)$\Rightarrow$(i) Evident.
   \end{proof}
   \section{\label{sekt6}Boundedness of $C_\varphi$
(only necessity)}
   We begin by proving a kind of ``cancellation''
lemma.
   \begin{lem} \label{limsup}
If $\varPhi \in \fscr$ and $f, g \colon \hh \to
[0,\infty)$ are such that $\liminf_{\|\xi\| \to
\infty} g(\xi) > 0$ and $\limsup_{\|\xi\| \to
\infty} \frac{\varPhi(f(\xi))}{\varPhi(g(\xi))} <
\infty$, then $\limsup_{\|\xi\| \to \infty}
\frac{f(\xi)}{g(\xi)} < \infty$.
   \end{lem}
   \begin{proof}
By assumption there exist $R_0,\varepsilon \in (0,
\infty)$ and $c \in [1,\infty) $ such that
   \begin{align} \label{atuc}
g(\xi) \Ge \varepsilon \text{ and } c \Ge
\frac{\varPhi(f(\xi))}{\varPhi(g(\xi))} \text{ for all
} \xi \in \hh \text{ such that } \|\xi\| \Ge R_0.
   \end{align}
Let $\varPhi$ be as in \eqref{wsp}. Then \eqref{atuc}
implies that
   \begin{align} \label{skar}
a_0(c-1) \Ge \sum_{n=1}^\infty a_n g(\xi)^n
\bigg(\bigg(\frac{f(\xi)}{g(\xi)}\bigg)^n - c\bigg),
\quad \xi \in \hh, \, \|\xi\| \Ge R_0.
   \end{align}
Suppose the contrary, that $\limsup_{\|\xi\| \to
\infty} \frac{f(\xi)}{g(\xi)} = \infty$. Then there
exists a sequence $\{\xi_k\}_{k=1}^\infty \subseteq
\hh$ such that $\lim_{k\to\infty}\|\xi_k\| =\infty$
and $\lim_{k\to \infty}\frac{f(\xi_k)}{g(\xi_k)} =
\infty$. Take any real number $\vartheta > c$. Then,
by \eqref{atuc}, there exists an integer $k\Ge 1$ such
that $\|\xi_k\| \Ge R_0$, $g(\xi_k) \Ge \varepsilon$
and $\frac{f(\xi_k)}{g(\xi_k)} \Ge \vartheta$. This
and \eqref{skar} yield
   \begin{align*}
a_0(c-1) \Ge \sum_{n=1}^\infty a_n g(\xi_k)^n
\bigg(\bigg(\frac{f(\xi_k)}{g(\xi_k)}\bigg)^n -
c\bigg) \Ge (\vartheta - c) (\varPhi(\varepsilon) -
\varPhi(0)).
   \end{align*}
Since $\varPhi(\varepsilon) - \varPhi(0) > 0$, we
deduce that $\vartheta \Le c + \frac
{a_0(c-1)}{\varPhi(\varepsilon) - \varPhi(0)}$ for all
$\vartheta > c$, which is a contradiction.
   \end{proof}
Now we show that if a composition operator
$C_{\varphi}$ is densely defined and bounded, then its
symbol $\varphi$ is a polynomial of degree at most $1$
(or, in other words, $\varphi$ is affine), i.e.,
$\varphi = A+b$ for some $A \in \ogr{\hh}$ and $b\in
\hh$, where $(A + b)(\xi) = A(\xi) + b$ for $\xi \in
\hh$ (see \cite{c-m-s03} for a very particular case of
analytic composition operators on the Segal-Bargmann
space of finite order). As shown in Example \ref{ndd},
$C_{\varphi}$ may be densely defined even if the
symbol $\varphi$ is not a polynomial.
    \begin{pro} \label{pro1}
Suppose $\varPhi\in \fscr$, $\varphi\colon \hh \to \hh$ is
a holomorphic mapping and $\dz{C_{\varphi}}=\varPhi(\hh)$.
Then $C_{\varphi}$ is bounded and there exists a unique
pair $(A,b) \in \ogr{\hh} \times \hh$ such that $\varphi =
A + b$.
    \end{pro}
   \begin{proof}
By Proposition \ref{stale2} and the closed graph
theorem, the operator $C_{\varphi}$ and, consequently,
$C_{\varphi}^*$ are bounded. Since $\varPhi(x)
> 0$ for all $x\in (0,\infty)$ and
$\|K^{\varPhi}_{\xi}\|^2= \varPhi(\|\xi\|^2)$ for all
$\xi \in \hh$, we see that $K^{\varPhi}_{\xi} \neq 0$
for all $\xi \in \hh \setminus \{0\}$. This together
with Theorem \ref{sprz}(i) implies that
   \begin{align*}
\frac{\varPhi(\|\varphi(\xi)\|^2)}{\varPhi(\|\xi\|^2)}
\overset{\eqref{rep2}}=
\frac{\|K^{\varPhi}_{\varphi(\xi)}\|^2}{\|K^{\varPhi}_{\xi}\|^2}
= \bigg\|C_{\varphi}^*
\frac{K^{\varPhi}_{\xi}}{\|K^{\varPhi}_{\xi}\|}\bigg\|^2
\Le \|C_{\varphi}\|^2, \quad \xi\in \hh\setminus \{0\}.
   \end{align*}
Hence, by Lemma \ref{limsup}, $\limsup_{\|\xi\|
\to \infty} \frac{\|\varphi(\xi)\|}{\|\xi\|} <
\infty$. Applying \cite[Theorem 13.8]{chae}, we
get the required representation $\varphi = A +
b$.
   \end{proof}
   The study of boundedness of a composition operator
with an affine symbol reduces to investigating the
case when the linear part of the symbol is positive.
We begin by stating a lemma which is a direct
consequence of Proposition \ref{isintin+}.
   \begin{lem} \label{isintin}
Suppose $\varPhi \in \fscr$, $\varphi\colon \hh \to
\hh$ is a holomorphic mapping and $W\in \ogr{\hh}$ is
a coisometry. Then $C_W C_{\varphi} = C_{\varphi \circ
W}$.
   \end{lem}
   \begin{pro} \label{fipsib}
Let $\varPhi \in \fscr$, $\varphi=A+b$ and $\psi =
|A^*| + b$, where $A\in\ogr{\hh}$ and $b\in\hh$. Then
   \begin{enumerate}
   \item[(i)] $\dz{C_{\varphi}} = \dz{C_{\psi}}$,
   \item[(ii)] $C_{\varphi} \in \ogr{\varPhi(\hh)}$
if and only if $C_{\psi} \in \ogr{\varPhi(\hh)}$,
   \item[(iii)]
$\|C_{\varphi}\|=\|C_{\psi}\|$ provided $C_{\varphi}
\in \ogr{\varPhi(\hh)}$.
   \end{enumerate}
   \end{pro}
   \begin{proof}
Let $A=U|A|$ be the polar decomposition of $A$. Then
$A^*=U^*|A^*|$ is the polar decomposition of $A^*$.
Consider two cases.

{\sc Case 1.} $\dim \jd{A^*} \Le \dim \jd{A}$.

Then there exists an isometry $\widetilde W \in
\ogr{\hh}$ which extends
$U^*|_{\overline{\ob{|A^*|}}}$. Hence $A^* =
\widetilde W |A^*|$ and so $A = |A^*|W$, where
$W:=(\widetilde W)^*$ is a coisometry. Therefore $\psi
\circ W = \varphi$, which by Lemma \ref{isintin}
yields $C_W C_{\psi} = C_{\varphi}$. Since by
Corollary \ref{unit}, $C_W \in \ogr{\varPhi(\hh)}$ is
an isometry, we see that (i), (ii) and (iii) hold.

{\sc Case 2.} $\dim \jd{A} \Le \dim \jd{A^*}$.

Let $P\in\ogr{\hh}$ be the orthogonal projection of
$\hh$ onto $\overline{\ob{A}}$. We will show that
there exists a coisometry $V \in \ogr{\hh}$ such that
   \begin{align}  \label{UVP}
UV=P.
   \end{align}
Indeed, since $\dim \jd{A} \Le \dim \jd{A^*}$, there
exists an isometry $\tilde{U} \in \ogr{\hh}$ which
extends $U|_{\overline{\ob{|A|}}}$. Then $V:=\tilde
U^*$ is a partial isometry with initial space
$\ob{\tilde U}$ and final space $\hh$ which extends
$U^*|_{\overline{\ob{A}}}$. Hence,
$V\big(\ob{\tilde{U}} \ominus \overline{\ob{A}}\,\big)
= \hh \ominus \overline{\ob{|A|}} = \jd{U}$. It is now
easy to see that $V$ satisfies \eqref{UVP}. As
$A=|A^*|U$ and $|A^*|(I_{\hh}-P)=0$, we have
   \begin{align*}
(\varphi \circ V)(\xi) = |A^*|U V \xi + b
\overset{\eqref{UVP}}= |A^*|P \xi + b = |A^*| \xi + b
= \psi(\xi), \quad \xi \in \hh.
   \end{align*}
This, the fact that $V$ is a coisometry and Lemma
\ref{isintin} imply that $C_V C_{\varphi} = C_{\psi}$.
Since, by Corollary \ref{unit}, $C_V \in
\ogr{\varPhi(\hh)}$ is an isometry, we deduce that the
conditions (i), (ii) and (iii) are satisfied. This
completes the proof.
   \end{proof}
   \section{\label{sekt7}Fock's type model for $C_A$}
   Recall that, by Theorem \ref{sprz},
$\kscr^{\varPhi}$ is a core for $C_{\varphi}^*$
whenever $C_{\varphi}$ is densely defined. In this
section we discuss the question of whether
$\kscr^{\varPhi}$ is a core for $C_{\varphi}$ and the
related question of whether $C_{\varphi}^* =
C_{{\varphi}^*}$ for $\varphi = A \in \ogr{\hh}$. To
answer these two questions, we build a Fock's type
model for $C_A$ (cf.\ Theorem \ref{fmod}). We begin by
proving an auxiliary lemma.
   \begin{lem} \label{lem1}
If $\varPhi \in \fscr$ and $A \in \ogr{\hh}$, then
   \begin{enumerate}
   \item[(i)] $\kscr^{\varPhi} \subseteq \dz{C_A}$
and $C_A(K_\xi^{\varPhi}) = K_{A^*\xi}^{\varPhi}$ for all
$\xi \in \hh$,
   \item[(ii)] $C_{A}^* = \overline{C_{A^*}|_{\kscr^{\varPhi}}}$,
   \item[(iii)] $(C_A|_{\kscr^{\varPhi}})^* = C_{A^*}$.
   \end{enumerate}
   \end{lem}
   \begin{proof}
   (i) If $\xi, \eta \in \hh$, then by \eqref{kfxi}, we
   have
   \begin{align*}
(K_\xi^{\varPhi}\circ A)(\eta) = \varPhi(\is{A\eta}{\xi}) =
\varPhi(\is{\eta}{A^*\xi}) = K_{A^*\xi}^{\varPhi}(\eta),
   \end{align*}
which implies that $K_\xi^{\varPhi} \in \dz{C_A}$ and
$C_A(K_\xi^{\varPhi}) = K_{A^*\xi}^{\varPhi}$.

(ii) Applying \eqref{op*} and (i) (the latter to
$A^*$), we see that $C_{A}^*|_{\kscr^{\varPhi}} =
C_{A^*}|_{\kscr^{\varPhi}}$. Taking closures and using
Theorem \ref{sprz}(i), we obtain (ii).

(iii) Applying (ii) to $A^*$, we get $C_{A^*}^* =
\overline{C_{A}|_{\kscr^{\varPhi}}}$. Since $C_{A^*}$
is closed (see Proposition \ref{stale2}) and densely
defined, the von Neumann theorem yields
   \begin{align*}
& (C_A|_{\kscr^{\varPhi}})^* = (C_{A^*}^*)^* =
C_{A^*}. \qedhere
   \end{align*}
   \end{proof}
For $n\in \zbb_+$, we denote by $\hh^{\odot n}$ the
$n$-th symmetric tensor power of a complex Hilbert
space $\hh$ and by $A^{\odot n}$ the $n$-th symmetric
tensor power of $A \in \ogr{\hh}$. In particular, we
follow the convention that $\hh^{\odot 0}=\cbb$,
$\xi^{\otimes 0}=1$ for all $\xi \in \hh$ and
$A^{\odot 0}=I_{\cbb}$. Given $\varPhi \in \fscr$, we
write $\varGamma_{\varPhi}(A) = \bigoplus_{n \in
\zscr_{\varPhi}} A^{\odot n}$ for $A\in \ogr{\hh}$.
The operator $\varGamma_{\varPhi}(A)$ is closed and
densely defined as an orthogonal sum of a number of
bounded operators $A^{\odot n}\in \ogr{\hh^{\odot
n}}$. The mapping $\varGamma_{\varPhi}(\cdot)$ is used
below to build a Fock's type model for $C_A$.
   \begin{thm}\label{fmod}
Suppose $\varPhi \in \fscr$, $Q$ is a conjugation on
$\hh$ and $A \in \ogr{\hh}$. Then there exists a
unitary isomorphism $U_{\varPhi,Q}\colon \varPhi(\hh)
\to \bigoplus_{n \in \zscr_{\varPhi}} \hh^{\odot n}$,
which does not depend on $A$, such that
   \begin{align}   \label{import}
C_{A}^*= U_{\varPhi,Q}^{-1} \varGamma_{\varPhi}
(\varXi_{Q}(A)) U_{\varPhi,Q},
   \end{align}
where $\varXi_{Q}(A)=QAQ$ $($see Appendix {\em
\ref{apap}}\/$)$.
   \end{thm}
   \begin{proof}
Set $\mm=\bigoplus_{n \in \zscr_{\varPhi}} \hh^{\odot n}$
and $T=\varGamma_{\varPhi}(\varXi_{Q}(A))$. Let
$\{a_n\}_{n=0}^{\infty}$ be as in \eqref{wsp}. Since
   \begin{align} \label{CCC}
\sum_{n \in \zscr_{\varPhi}} \|\sqrt{a_n} \, \xi^{\otimes
n}\|^2 = \varPhi(\|\xi\|^2), \quad \xi \in \hh,
   \end{align}
we can define the function $Y\colon \hh \to \mm$ by
   \begin{align} \label{defY}
Y(\xi) = \bigoplus_{n \in \zscr_{\varPhi}} \sqrt{a_n} \;
\xi^{\otimes n}, \quad \xi \in \hh.
   \end{align}
Denote by $\yscr$ the linear span of $\{Y(\xi)\colon \xi
\in \hh\}$. By \eqref{CCC}, we have
   \begin{align*}
\sum_{n \in \zscr_{\varPhi}} \|\sqrt{a_n} \,
\varXi_Q(A)^{\odot n} \xi^{\otimes n}\|^2 =
\varPhi(\|\varXi_Q(A)\xi\|^2) < \infty, \quad \xi \in \hh,
   \end{align*}
which implies that $\yscr \subseteq \dz{T}$. We will show
that $\yscr$ is a core for $T$. To this end, we take $h \in
\dz{T}$ which is orthogonal to $\yscr$ with respect to the
graph inner product $\is{\cdot}{\mbox{-}}_T$. Since
$h=\bigoplus_{n \in \zscr_{\varPhi}} h_n$ with $h_n \in
\hh^{\odot n}$, we get
   \begin{align*}
0 & = \is{h}{Y(\bar \lambda\xi)} + \is{Th}{TY(\bar
\lambda\xi)}
   \\
& = \sum_{n\in \zscr_{\varPhi}} \sqrt{a_n} \;
\Big(\is{h_n}{\xi^{\otimes n}} + \is{\varXi_Q(A)^{\odot n}
h_n}{\varXi_Q(A)^{\odot n}\xi^{\otimes n}}\Big) \lambda^n,
\quad \lambda \in \cbb, \, \xi \in \hh.
   \end{align*}
By the identity theorem for power series, this implies that
   \begin{align*}
\is{(I+(\varXi_Q(A)^{\odot n})^*\varXi_Q(A)^{\odot
n})h_n}{\xi^{\otimes n}} = 0, \quad \xi \in \hh, \, n \in
\zscr_{\varPhi}.
   \end{align*}
Using the polarization formula (cf.\ \cite[Theorem
4.6]{chae}), one can show that the set $\{\xi^{\otimes
n}\colon \xi \in \hh\}$ is total in $\hh^{\odot n}$.
Hence, we have
   \begin{align*}
(I+(\varXi_Q(A)^{\odot n})^*\varXi_Q(A)^{\odot n})h_n = 0,
\quad n \in \zscr_{\varPhi}.
   \end{align*}
As a consequence, we see that $h_n=0$ for all
$n\in\zbb_+$, which yields $h=0$. This combined with
the fact that $(\dz{T}, \is{\cdot}{\mbox{-}}_T)$ is a
Hilbert space (cf.\ \cite[Theorem 5.1]{Weid}) implies
that $\yscr$ is a core for $T$. Since $T$ is densely
defined, we get
   \begin{align} \label{yden}
\overline{\yscr}=\mm.
   \end{align}

It follows from \eqref{defY} that
   \begin{align} \label{UEQ}
\is{Y(\xi)}{Y(\eta)} = \varPhi(\is{\xi}{\eta}) =
\varPhi(\is{Q\eta}{Q\xi}) \overset{\eqref{rep2}}=
\is{K^{\varPhi}_{Q\xi}}{K^{\varPhi}_{Q\eta}}, \quad \xi, \,
\eta \in \hh.
   \end{align}
Since $\kscr^{\varPhi}$ is dense in $\varPhi(\hh)$ and
$Q$ is surjective, we deduce from \eqref{yden} and
\eqref{UEQ} that there exists a unique unitary
isomorphism $U=U_{\varPhi,Q}\colon \varPhi(\hh) \to
\mm$ such that (compare with \cite[Proposition
1]{jerz2})
   \begin{align} \label{defU}
UK^{\varPhi}_{Q\xi} = Y(\xi), \quad \xi \in \hh.
   \end{align}
Applying Theorem \ref{sprz} and the equalities
\eqref{defY} and \eqref{defU} , we see that
$\kscr^{\varPhi} \subseteq \dz{C_{A}^*}$,
$C_{A}^*(\kscr^{\varPhi}) \subseteq \kscr^{\varPhi}$
and
   \begin{align*}
T U(K^{\varPhi}_{Q\xi}) & = T Y(\xi) = \bigoplus_{n
\in\zscr_{\varPhi}} \sqrt{a_n} \; (\varXi_Q(A)\xi)^{\otimes
n}
   \\
&= Y(\varXi_Q(A)\xi) = U K^{\varPhi}_{AQ\xi} = U C_A^*
(K^{\varPhi}_{Q\xi}), \quad \xi \in \hh.
   \end{align*}
This combined with the surjectivity of $Q$, the
injectivity of $U$ and \eqref{defU} implies that $U
\kscr^{\varPhi} =\yscr$, $T\yscr \subseteq \yscr$ and
$T|_{\yscr} U = U C_{A}^*|_{\kscr^{\varPhi}}$. Hence
$U^{-1}T|_{\yscr} U = C_{A}^*|_{\kscr^{\varPhi}}$
which together with Theorem \ref{sprz}, Lemma
\ref{lem1} and the fact that $\yscr$ is a core for $T$
yields
   \begin{align*}
C_{A}^* = \overline{C_{A^*}|_{\kscr^{\varPhi}}}= \overline{
C_{A}^*|_{\kscr^{\varPhi}}} = U^{-1}\overline{T|_{\yscr}} U
= U^{-1} T U.
   \end{align*}
This completes the proof.
   \end{proof}
   Before answering the two questions posed at the
beginning of this section, we state and prove one more
lemma.
   \begin{lem} \label{lem2}
Suppose $\varPhi \in \fscr$ and $A \in \ogr{\hh}$.
Then the following conditions are equivalent{\em :}
   \begin{enumerate}
   \item[(i)] $C_A^*=C_{A^*}$,
   \item[(ii)] $\kscr^{\varPhi}$ is a core for $C_A$,
   \item[(iii)] $\jd{I + C_AC_{A^*}} = \{0\}$,
   \item[(iv)] if $f \in \dz{C_A}$ and $f(AA^*\xi)=-f(\xi)$
for all $\xi \in \hh$, then $f=0$.
   \end{enumerate}
   \end{lem}
   \begin{proof}
(i)$\Leftrightarrow$(ii) If $C_A^*=C_{A^*}$, then
Proposition \ref{stale2}(i), Lemma \ref{lem1}(iii) and
the von Neumann theorem yield
   \begin{align*}
C_A=C_A^{**} = C_{A^*}^* =
(C_A|_{\kscr^{\varPhi}})^{**} =
\overline{C_A|_{\kscr^{\varPhi}}}.
   \end{align*}
Similar reasoning gives the reverse implication.

(i)$\Leftrightarrow$(iii) This is a direct consequence
of the inclusion $C_{A}^* \subseteq C_{A^*}$ (apply
Lemma \ref{lem1}(ii)) and the following general fact:
if $S$ and $T$ are closed densely defined operators in
a complex Hilbert space $\hh$ such that $S \subseteq
T$, then $S=T$ if and only if $\jd{I+S^*T}=\{0\}$.

(ii)$\Leftrightarrow$(iv) Note that $\kscr^{\varPhi}$
is a core for $C_A$ if and only if $\kscr^{\varPhi}$
is dense in $\dz{C_A}$ with respect to the graph norm
or equivalently if and only if the only function $f$
in $\dz{C_A}$ such that $\is{f}{K_\xi^{\varPhi}} +
\is{C_Af}{C_A K_\xi^{\varPhi}} = 0$ for all $\xi \in
\hh$ is the zero function. Since, by Lemma \ref{lem1}
and Theorem \ref{sprz}(i),
   \begin{align*}
\is{f}{K_\xi^{\varPhi}} + \is{C_Af}{C_A
K_\xi^{\varPhi}} & \overset{\eqref{rep}}= f(\xi) +
\is{C_Af}{K_{A^*\xi}^{\varPhi}}
   \\
& \hspace{1ex}= f(\xi) + \is{f}{K_{AA^*\xi}^{\varPhi}}
\overset{\eqref{rep}}= f(\xi) + f(AA^*\xi), \quad \xi
\in \hh,
   \end{align*}
we conclude that the conditions (ii) and (iv) are
equivalent.
   \end{proof}
   Now we are in a position to answer the
aforementioned questions.
   \begin{thm} \label{wnoca}
Suppose $\varPhi \in \fscr$ and $A \in \ogr{\hh}$. Then
   \begin{enumerate}
   \item[(i)] $C_{A}^*=C_{A^*}$,
   \item[(ii)] $\kscr^{\varPhi}$ is a core for $C_A$.
   \end{enumerate}
   \end{thm}
   \begin{proof}
(i) First note that $(B^*)^{\odot n} = (B^{\odot
n})^*$ for every $B\in \ogr{\hh}$ and for all $n\in
\zbb_+$. Fix any conjugation $Q$ on $\hh$. Since $C_A$
is closed and densely defined (cf.\ Lemma \ref{lem1}),
we infer from Theorem \ref{fmod} and Proposition
\ref{abcon}(v) that
   \begin{align*}
C_A = (C_A^*)^* &\overset{\eqref{import}}=
U_{\varPhi,Q}^{-1}\bigg(\bigoplus_{n \in \zscr_{\varPhi}}
\varXi_Q(A)^{\odot n}\bigg)^* U_{\varPhi,Q}
   \\
&\hspace{1.5ex} = U_{\varPhi,Q}^{-1}\bigg(\bigoplus_{n \in
\zscr_{\varPhi}} \varXi_Q(A^*)^{\odot n}\bigg)
U_{\varPhi,Q}\overset{\eqref{import}}= (C_{A^*})^*.
   \end{align*}
Applying the above to $A^*$ in place of $A$ (or taking
adjoints), we get $C_{A}^*=C_{A^*}$.

(ii) Apply Lemma \ref{lem2}.
   \end{proof}
   \begin{cor} \label{osfn}
Suppose $\varPhi \in \fscr$ and $A \in \ogr{\hh}$.
Then $C_A$ is unitarily equivalent to $\bigoplus_{n
\in \zscr_{\varPhi}} C^{\langle n\rangle}_A$, where
$C^{\langle 0 \rangle}_A$ is the identity operator on
$\cbb$ and for every $n\in \nbb$, $C^{\langle n
\rangle}_A$ denotes the composition operator in
$\varPhi_n(\hh)$ with the symbol $A$ and
$\varPhi_n(z)=z^n$ for $z\in \cbb$. Moreover, if $Q$
is a conjugation on $\hh$, then for every $n\in
\zbb_+$, $C^{\langle n \rangle}_{\varXi_{Q}(A^*)}$ is
unitarily equivalent to $A^{\odot n}$.
   \end{cor}
   \begin{proof}
Applying Theorems \ref{fmod} and \ref{wnoca}(i) to
$\varPhi$ (resp.,\ $\varPhi_n$ with $n\in \nbb$), we
deduce that $C_A$ is unitarily equivalent to
$\bigoplus_{n \in \zscr_{\varPhi}}
\varXi_{Q}(A^*)^{\odot n}$ (resp.,\ $C^{\langle n
\rangle}_A$ is unitarily equivalent to
$\varXi_{Q}(A^*)^{\odot n}$ for every $n\in \nbb$).
This and Proposition \ref{abcon} yield the
``moreover'' part and the unitary equivalence of $C_A$
and $\bigoplus_{n \in \zscr_{\varPhi}} C^{\langle n
\rangle}_A$.
   \end{proof}
   \begin{cor} \label{sa-sa}
Suppose $\varPhi \in \fscr$ and $A \in \ogr{\hh}$. Then
$C_A$ is selfadjoint if and only if there exists $\alpha
\in \gfrak_{\varPhi}$ such that $A^* = \alpha A$.
   \end{cor}
   \begin{proof}
This is a direct consequence of Theorem \ref{wnoca}(i)
and Theorem \ref{fipsi}.
   \end{proof}
   \section{\label{sekt8}Boundedness and partial isometricity
of $C_A$}
   The following lemma will be used to calculate the
norm of the composition operator $C_A$, where $A \in
\ogr{\hh}$ (see Theorem \ref{bca}). For
self-containedness, we include its proof.
   \begin{lem} \label{lempom}
If $A \in \ogr{\hh}$ and $n\in \zbb_+$, then $A^{\odot
n} \in \ogr{\hh^{\odot n}}$ and $\|A^{\odot
n}\|=\|A\|^n$.
   \end{lem}
   \begin{proof}
It suffices to consider the case of $n\in \nbb$. Since
$A^{\odot n} \subseteq A^{\otimes n} \in
\ogr{\hh^{\otimes n}}$, we see that $A^{\odot n} \in
\ogr{\hh^{\odot n}}$ and $\|A^{\odot n}\|\Le
\|A^{\otimes n}\|=\|A\|^n$. However
   \begin{align*}
\|Af\|^n = \|A^{\odot n} f^{\otimes n}\| \Le
\|A^{\odot n}\| \|f\|^n, \quad f \in \hh,
   \end{align*}
which implies that $\|A\|^n \Le \|A^{\odot n}\|$. This
completes the proof.
   \end{proof}
   Theorem \ref{bca} below provides necessary and
sufficient conditions for $C_A$ to be bounded and the
explicit formulas for the norm and the spectral radius
of $C_A$. Given $m\in \zbb_+$ and $n\in \zbb_+ \cup
\{\infty\}$, we define the function $q_{m,n}\colon
[0,\infty) \to [0,\infty]$ by
   \begin{align*}
q_{m,n}(\vartheta)
=\vartheta^m\max\{1,\vartheta^{n-m}\}, \quad \vartheta
\in [0,\infty),
   \end{align*}
where $\vartheta^0=1$ for $\vartheta \in [0,\infty)$,
$\vartheta^{\infty} = \infty$ for $\vartheta \in
(1,\infty)$, $\vartheta^\infty = 0$ for $\vartheta \in
[0,1)$ and $1^\infty=1$ (cf.\ \cite[Lemma 2.1]{js1}).
   \begin{thm}\label{bca}
Suppose $\varPhi \in \fscr$ and $A\in \ogr{\hh}$.
Set\/\footnote{\;Note that $0$ is a zero of $\varPhi$
of multiplicity $m$ and $\infty$ is a pole of
$\varPhi$ of order $n$.} $m=\min \zscr_{\varPhi}$ and
$n=\sup \zscr_{\varPhi}$. Then
   \begin{enumerate}
   \item[(i)] if $n < \infty$, then $C_A
\in \ogr{\varPhi(\hh)}$,
   \item[(ii)] if $n =  \infty$, then
$C_A\in \ogr{\varPhi(\hh)}$ if and only if $\|A\|\Le
1$.
   \end{enumerate}
Moreover, if $C_A\in \ogr{\varPhi(\hh)}$, then
   \begin{align}  \label{normca}
\|C_A\| & = q_{m,n}(\|A\|),
   \\ \label{prsp}
r(C_A) & = q_{m,n}(r(A)).
   \end{align}
   \end{thm}
   \begin{proof}
Clearly, $C_A \in \ogr{\hh}$ if and only if $C_A^* \in
\ogr{\hh}$, or equivalently by Theorem \ref{fmod},
Lemma \ref{lempom} and Proposition \ref{abcon}, if and
only if $\sup_{k \in \zscr_{\varPhi}}
\|\varXi_{Q}(A)^{\odot k}\| = \sup_{k \in
\zscr_{\varPhi}} \|A\|^k < \infty$. This implies (i)
and (ii). Moreover, we have
   \begin{align*}
\|C_A\| = \|C_A^*\| \overset{\eqref{import}} = \sup_{k
\in \zscr_{\varPhi}} \|\varXi_{Q}(A)^{\odot k}\| =
\sup_{k \in \zscr_{\varPhi}} \|A\|^k = q_{m,n}(\|A\|).
   \end{align*}

To prove \eqref{prsp}, assume that $C_A\in
\ogr{\varPhi(\hh)}$. Since $C_{A^k}=C_A^k \in
\ogr{\varPhi(\hh)}$ for every $k\in \nbb$, we infer
from \eqref{normca} that
   \begin{align*}
\|C_A^k\|^{1/k} = q_{m,n}(\|A^k\|^{1/k}), \quad k \in
\nbb.
   \end{align*}
If $n < \infty$, then $q_{m,n}$ is continuous on
$[0,\infty)$. This, combined with Gelfand's formula
for spectral radius, yields
   \begin{align}   \label{kilo}
r(C_A) = \lim_{k\to\infty} q_{m,n}(\|A^k\|^{1/k}) =
q_{m,n}(r(A)).
   \end{align}
In turn, if $n = \infty$, then $q_{m,n}$ is continuous
on $[0,1]$ and, by (ii), $\|A^k\|^{1/k} \in [0,1]$ for
every $k\in \nbb$. Passing to the limit, as in
\eqref{kilo}, we get \eqref{prsp}.
   \end{proof}
   \begin{cor} \label{normaloid}
Suppose $\varPhi \in \fscr$, $A\in \ogr{\hh}$ and $C_A
\in \ogr{\varPhi(\hh)}$. Then
   \begin{enumerate}
   \item[(i)] if $A$ is normaloid, then so is $C_{A}$,
   \item[(ii)] if $\varPhi(0) \neq 0$ and $\|A\| \Le 1$,
then $C_A$ is normaloid and $r(C_A)=\|C_A\|=1$,
   \item[(iii)] if either $\varPhi(0) = 0$ or $\|A\| > 1$,
then $C_A$ is normaloid if and only if $A$ is
normaloid.
   \end{enumerate}
   \end{cor}
   \begin{proof}
The assertions (i) and (ii) follow directly from
Theorem \ref{bca}. In view of (i), it remains to prove
the ``only if'' part of (iii). Assume $C_A$ is
normaloid. Using Theorem \ref{bca}, we show that in
each of the three possible cases $A$ is normaloid.
Indeed, if $\varPhi(0) \neq 0$ and $\|A\|
> 1$, then $m=0$ and $1 \Le n < \infty$, and  thus,
by \eqref{normca} and \eqref{prsp}, $\max\{1, r(A)^n\}
= \|A\|^n$, which implies that $r(A) = \|A\|$. In
turn, if $\varPhi(0) = 0$ and $\|A\|\Le 1$, then $1
\Le m \Le n \Le \infty$ and so $r(A)^m = \|A\|^m$,
which gives $r(A) = \|A\|$. Finally, if $\varPhi(0) =
0$ and $\|A\| > 1$, then $1 \Le m \Le n < \infty$ and
thus
   \begin{align*}
r(A)^m \max\{1,r(A)^{n-m}\}=\|A\|^n,
   \end{align*}
which yields $r(A) = \|A\|$. This completes the proof.
   \end{proof}
   \begin{cor} \label{unit}
Suppose $\varPhi \in \fscr$ and $A \in \ogr{\hh}$.
Then $C_A$ is an isometry $($resp.,\ a coisometry, a
unitary operator$)$ if and only if $A$ is a coisometry
$($resp.,\ an isometry, a unitary operator$)$.
   \end{cor}
   \begin{proof}
By Theorem \ref{bca}, we may assume that $C_A \in
\ogr{\varPhi(\hh)}$. We will consider only the case
when $C_A$ is an isometry, leaving the remaining cases
to the reader. It follows from Theorem \ref{wnoca}
that $C_A$ is an isometry if and only if
$C_{AA^*}=C_A^*C_A = I_{\varPhi(\hh)} = C_I$, or
equivalently, by Theorem \ref{fipsi}, if and only if
there exists $\alpha \in \gfrak_{\varPhi}$ such that
$AA^* = \alpha \cdot I$. Since $|\alpha| = 1$ and
$AA^* \Ge 0$, we have $\alpha = 1$ (because, by
\eqref{SA}, $\hh\neq \{0\}$). This completes the
proof.
   \end{proof}
   \begin{cor}\label{opr}
Let $\varPhi \in \fscr$ and $P \in \ogr{\hh}$. Then
$C_P$ is an orthogonal projection if and only if there
exists $\alpha \in \gfrak_{\varPhi}$ such that $\alpha
P$ is an orthogonal projection.
   \end{cor}
   \begin{proof}
If $C_P$ is an orthogonal projection, then by Theorem
\ref{wnoca}, $C_{P^*P} \supseteq C_P C_P^* = C_P$.
Hence, by Theorem \ref{fipsi}, there exists $\beta \in
\gfrak_{\varPhi}$ such that $P = \beta P^*P$. Set
$\alpha=\bar \beta$ and $Q=\alpha P$. Since
$|\alpha|=1$, we see that $Q^*Q = Q$.

To prove the reverse implication, set $Q=\alpha P$.
Since $P,Q$ are contractions, we infer from Theorem
\ref{bca} that $C_P, C_Q \in \ogr{\varPhi(\hh)}$. This
and Theorem \ref{wnoca} imply that $C_Q^* = C_Q =
C_Q^2$. By Theorem \ref{fipsi}, $C_P = C_Q$, which
completes the proof.
   \end{proof}
   \begin{cor} \label{partis}
Let $\varPhi \in \fscr$ and $A \in \ogr{ \hh}$. Then
$C_A$ is a partial isometry if and only if $A$ is a
partial isometry.
   \end{cor}
   \begin{proof}
Here we use the well-known characterizations of
partial isometries, see \cite[Problem 127 and
Corollary 2]{Hal}. Suppose $C_A$ is a partial
isometry. By Theorem \ref{wnoca}, we have
   \begin{align*}
C_A C_A^* = C_{A^*A}.
   \end{align*}
As $C_A C_A^*$ is an orthogonal projection, we deduce
from the above equality and Corollary \ref{opr} that
there exists $\alpha \in \gfrak_{\varPhi}$ such that
$\alpha A^*A$ is an orthogonal projection. Since
orthogonal projections are positive, we conclude that
$\alpha = 1$ whenever $A \neq 0$ (the case of $A=0$ is
obvious). Reversing the above reasoning and using
Theorem \ref{bca} we complete the proof.
   \end{proof}
   \section{\label{sekt9}Positivity and the polar
decomposition of $C_A$}
   We begin by giving necessary and sufficient
conditions for the composition operator $C_A$ to be
positive. In particular, we will show that if $C_{A}$
is positive, then $C_{A}$ is selfadjoint. First, we
state a simple lemma.
   \begin{lem} \label{alai}
Suppose $T\in \ogr{\hh}$ is a nonzero operator and
$\alpha_1, \alpha_2 \in \cbb$ are such that
$|\alpha_j| = 1$ and $\alpha_j T \Ge 0$ for $j=1,2$.
Then $\alpha_1=\alpha_2$.
   \end{lem}
   \begin{proof}
Since $\alpha_j T$ is selfadjoint, we see that
$\alpha_j^2T = T^*$, which yields
$\alpha_1^2=\alpha_2^2$. If $\alpha_2=-\alpha_1$, then
$0 \Le \alpha_2 T = - \alpha_1 T \Le 0$ and thus
$T=0$, a contradiction.
   \end{proof}
   \begin{thm} \label{gcd3}
Suppose $\varPhi \in \fscr$ and $A\in\ogr{\hh}$. Then
the following conditions are equivalent{\em :}
   \begin{enumerate}
   \item[(i)] $C_A \Ge 0$,
   \item[(ii)] there exists $\alpha \in \gfrak_{\varPhi}$
such that $\alpha A \Ge 0$,
   \item[(iii)] there exists $B\in \ogr{\hh}$ such that
$B \Ge 0$ and $C_A=C_B$.
   \end{enumerate}
Moreover, if $A \Ge 0$, then $C_A$ is selfadjoint and
$C_A=C_{A^{1/2}}^*C_{A^{1/2}}$.
   \end{thm}
   \begin{proof}
(i)$\Rightarrow$(ii) We split the proof of this
implication into two steps.

{\sc Step 1.} Suppose $C_A \Ge 0$ and $n\in
\zscr_{\varPhi} \setminus \{0\}$. Then there exists
$\alpha \in G_n$ such that $\alpha A \Ge 0$.

Indeed, it follows from Corollary \ref{osfn} and
Theorem \ref{bca} that $C^{\langle n \rangle}_A \in
\ogr{\varPhi_n(\hh)}$ and $C^{\langle n \rangle}_A \Ge
0$. Hence, by Corollary \ref{sa-sa} (applied to
$\varPhi_n$), there exists $\beta \in
\gfrak_{\varPhi_n}=G_n$ such that $A^* = \beta A$.
Then there exists $k\in \{0, \ldots, n-1\}$ such that
$\beta=\exp(\I \frac{k}{n} 2 \pi)$. Set $z = \exp(\I
\frac{k}{n} \pi)$ and $B=zA$. It is easily seen that
$B^* = B$. Let $Q$ be a conjugation on $\hh$. Since
$C^{\langle n \rangle}_A \Ge 0$, we infer from Theorem
\ref{fmod} (applied to $\varPhi_n$) that
   \begin{align} \label{xiqa}
\varXi_{Q}(A)^{\odot n} \Ge 0.
   \end{align}
This implies that
   \begin{align} \label{xiqa2}
\is{A\xi}{\xi}^n = \is{(Q\xi)^{\otimes
n}}{\varXi_{Q}(A)^{\odot n} (Q\xi)^{\otimes n}} \Ge 0,
\quad \xi \in \hh.
   \end{align}
Now we consider three cases.

{\sc Case 1.} $n$ is odd and $k$ is even.

Then, $z^n = (-1)^k=1$ and thus
   \begin{align*}
\is{B\xi}{\xi}^n = \is{A\xi}{\xi}^n
\overset{\eqref{xiqa2}}\Ge 0, \quad \xi \in \hh.
   \end{align*}
Since $B=B^*$ and $n$ is odd, we see that $B\Ge 0$.
Hence $\alpha A \Ge 0$ with $\alpha=z \in G_n$.

{\sc Case 2.} $n$ and $k$ are odd.

Then $\alpha:=(-z) \in G_n$, and so
   \begin{align*}
\is{(-B)\xi}{\xi}^n = \is{A\xi}{\xi}^n
\overset{\eqref{xiqa2}}\Ge 0, \quad \xi \in \hh.
   \end{align*}
Hence, as in Case 1, we see that $-B \Ge 0$ and
consequently $\alpha A \Ge 0$.

{\sc Case 3.} $n$ is even.

We can assume that $A \neq 0$. Since $B$ is
selfadjoint and
   \begin{align*}
(-1)^k \is{B\xi}{\xi}^n = \is{A\xi}{\xi}^n
\overset{\eqref{xiqa2}} \Ge 0, \quad \xi \in \hh,
   \end{align*}
we deduce that $k$ is even. Hence $z\in G_n$. We show
that either $B\Ge 0$ or $-B \Ge 0$. In view of
Proposition \ref{abcon}(vi), this reduces to showing
that either $T\Ge 0$ or $-T \Ge 0$, where
$T:=\varXi_{Q}(B)$. To prove the latter, first observe
that by Proposition \ref{abcon}(v), $T^*=T$. Let $E$
be the spectral measure of $T$. Then the closed vector
spaces $\hh_{-} := \ob{E((-\infty,0))}$, $\hh_{0} :=
\ob{E(\{0\})}$ and $\hh_{+} := \ob{E((0,\infty))}$
reduce $T$. Moreover, $T_{-}:= - T|_{\hh_{-}} \Ge 0$,
$\jd{T_{-}}=\{0\}$, $T_{+}:= T|_{\hh_{+}} \Ge 0$, $\jd
{T_{+}}=\{0\}$ and
   \begin{align} \label{Kr12}
T = (-T_{-}) \oplus 0|_{\hh_{0}} \oplus T_{+}.
   \end{align}
We will prove that either $\hh_{-}=\{0\}$ or
$\hh_{+}=\{0\}$. Indeed, otherwise there exist $e_{-}
\in \hh_{-}$ and $e_{+} \in \hh_{+}$ such that
   \begin{align} \label{pm}
\is{T_{-}e_{-}}{e_{-}} = 1=\is{T_{+}e_{+}}{e_{+}}.
   \end{align}
Take $\lambda_{1,-}, \lambda_{1,+},\lambda_{2,-},
\lambda_{2,+} \in \cbb$. Set $\xi_{i,\pm} =
\lambda_{i,\pm} \cdot e_{\pm}$ and $\xi_i= \xi_{i,-} +
\xi_{i,+}$ for $i=1,2$. Clearly $\xi_{i,\pm} \in
\hh_{\pm}$ for $i=1,2$. Since $z\in G_n$, we see that
$T^{\odot n} = \varXi_{Q}(A)^{\odot n}$ and thus, by
\eqref{xiqa}, \eqref{Kr12} and \eqref{pm}, we have
   \allowdisplaybreaks
   \begin{align} \notag
0 &\Le \Big\langle T^{\odot n}\Big(\sum_{i=1}^2
\xi_i^{\otimes n}\Big), \sum_{j=1}^2 \xi_j^{\otimes
n}\Big\rangle
   \\ \notag
&= \sum_{i,j=1}^2 \is{((-T_{-}) \oplus 0|_{\hh_{0}}
\oplus T_{+})(\xi_{i,-} \oplus 0 \oplus
\xi_{i,+})}{\xi_{j,-} \oplus 0 \oplus \xi_{j,+}}^n
   \\ \notag
& =\sum_{i,j=1}^2 (\is{T_{+}\xi_{i,+}}{\xi_{j,+}} -
\is{T_{-}\xi_{i,-}}{\xi_{j,-}}\big)^n
   \\ \label{pds}
& =\sum_{i,j=1}^2 (\lambda_{i,+}\bar \lambda_{j,+} -
\lambda_{i,-}\bar \lambda_{j,-})^n.
   \end{align}
Set $\lambda_{1,-} = w$, $\lambda_{1,+} = 1$,
$\lambda_{2,-} = 1$ and $\lambda_{2,+} = - \overline
w$ with $w=\exp(\I \frac{\pi}{n})$. It follows from
\eqref{pds} that
   \begin{align*}
0 \Le \sum_{i,j=1}^2 (\lambda_{i,+}\bar \lambda_{j,+}
- \lambda_{i,-}\bar \lambda_{j,-})^n = (-1)^n 2^{n+1}
\mathfrak{Re}(w^n) =- 2^{n+1} < 0,
   \end{align*}
which gives a contradiction. Hence either
$\hh_{-}=\{0\}$ or $\hh_{+}=\{0\}$, which together
with \eqref{Kr12} implies that either $T\Ge 0$ or $-T
\Ge 0$, or equivalently that either $B\Ge 0$ or $-B
\Ge 0$. If $B\Ge 0$, then $\alpha A \Ge 0$ with
$\alpha=z \in G_n$. Otherwise, $\alpha A \Ge 0$ with
$\alpha=-z \in G_n$. This completes the proof of Step
1.

{\sc Step 2.} Suppose $C_A \Ge 0$. Then there exists
$\alpha \in \gfrak_{\varPhi}$ such that $\alpha A \Ge
0$.

Indeed, by Step 1, for every $n\in
\zscr_{\varPhi}\setminus \{0\}$ there exists $\alpha_n
\in G_n$ such that $\alpha_n A \Ge 0$. Assuming that
$A \neq 0$ (which is no loss of generality), we deduce
from Lemma \ref{alai} that $\alpha_n=\alpha$ for every
$n\in \zscr_{\varPhi} \setminus \{0\}$, where
$\alpha:=\alpha_{\inf\zscr_{\varPhi}\setminus \{0\}}$.
Hence $\alpha \in \bigcap_{n\in
\zscr_{\varPhi}\setminus \{0\}} G_n =
\gfrak_{\varPhi}$ and $\alpha A \Ge 0$, which
completes the proof of Step 2, and thus of the
implication (i)$\Rightarrow$(ii).

(ii)$\Leftrightarrow$(iii) Apply Lemma \ref{lem1}(i)
and Theorem \ref{fipsi}.

(iii)$\Rightarrow$(i) Without loss of generality we
can assume that $A\Ge 0$. By Lemma \ref{lem1},
$C_{A^{1/2}}$ is closed and densely defined. Hence,
$C_{A^{1/2}}^*C_{A^{1/2}}$ is positive and selfadjoint
(cf.\ \cite[Theorem 5.39]{Weid}). In view of Theorem
\ref{wnoca}(i), we have
   \begin{align*}
C_{A^{1/2}}^*C_{A^{1/2}} = C_{A^{1/2}}C_{A^{1/2}}
\subseteq C_A = C_A^*.
   \end{align*}
By maximality of selfadjoint operators (cf.\
\cite[Theorem 5.31]{Weid}), we deduce that $C_A =
C_{A^{1/2}}^*C_{A^{1/2}}$, which implies that $C_A \Ge
0$. This also proves the ``moreover'' part.
   \end{proof}
   \begin{cor} \label{asp1}
Suppose $A\in \ogr{\hh}$ and $n \in \nbb$. Then the
following conditions are equivalent{\em :}
   \begin{enumerate}
   \item[(i)] $A^{\odot n} \Ge 0$,
   \item[(ii)] there exists $\alpha \in G_n$ such
that $\alpha A \Ge 0$,
   \item[(iii)] there exists $B \in \ogr{\hh}$ such
that $B \Ge 0$ and $B^{\odot n} = A^{\odot n}$,
   \item[(iv)] $A^{\otimes n} \Ge 0$.
   \end{enumerate}
   \end{cor}
   \begin{proof}
Applying Theorems \ref{fmod} and \ref{gcd3} to
$\varPhi=\varPhi_n$ and using Proposition \ref{abcon}
(see also Corollary \ref{osfn}), we deduce that the
conditions (i)-(iii) are equivalent.

(iv)$\Rightarrow$(i) Obvious.

(ii)$\Rightarrow$(iv) Since $\alpha A \Ge 0$ and
$\alpha \in G_n$, we get
   \begin{align*}
& A^{\otimes n} = (\alpha A)^{\otimes n} = (((\alpha
A)^{1/2})^{\otimes n})^*((\alpha A)^{1/2})^{\otimes n}
\Ge 0. \qedhere
   \end{align*}
   \end{proof}
   \begin{cor}
Suppose $Y$ is a nonempty subset of $\nbb$ and $z\in
\cbb$. Then the following conditions are
equivalent{\em :}
   \begin{enumerate}
   \item[(i)] $z^n \Ge 0$ for every $n \in Y$,
   \item[(ii)]  there exists $\alpha \in \cbb$ such
that $\alpha^{\,\mathrm{gcd}(Y)}=1$ and $\alpha z \Ge
0$,
   \item[(iii)] there exists $b \in \rbb_+$ such that
$b^n = z^n$ for every $n \in Y$.
   \end{enumerate}
   \end{cor}
   \begin{proof}
This can be deduced from Corollary \ref{asp1}, Lemma
\ref{alai} and Lemma \ref{gcd=1}.
   \end{proof}
   \begin{cor} \label{f-f}
Let $\varPhi \in \fscr$ and $A \in \ogr{\hh}$ be such
that $A \Ge 0$. If $f \in \varPhi(\hh)$ satisfies the
following equality
   \begin{align} \label{jur1+}
f(A\xi) = - f(\xi), \quad \xi \in \hh,
   \end{align}
then $f=0$.
   \end{cor}
   \begin{proof}
It follows from \eqref{jur1+} that $f \in \dz{C_A}$
and $C_A f = -f$. Since, by Theorem \ref{gcd3}, $C_A
\Ge 0$, we deduce that $f=0$.
   \end{proof}
Corollary \ref{f-f} is no longer true if $A$ is not
positive (even if $\dim \hh < \infty$). Regarding
Theorem \ref{gcd3}, it is worth mentioning that, in
general, the positivity of $C_A$ does not imply the
positivity of $A$ (e.g., if $\varPhi(z)=z^2$ for $z\in
\cbb$ and $A=-I_{\hh}$, then, by Theorem \ref{fipsi}
and Lemma \ref{lem1}(i), $C_A = C_{I_{\hh}} =
I_{\varPhi(\hh)} \Ge 0$).

Now we give an explicit description of powers
$C_A^{t}$ of $C_A$ with positive real exponents $t$ in
the case when $A\Ge 0$.
   \begin{thm} \label{alpharoot}
Let $\varPhi \in \fscr$, $A\in \ogr{\hh}$ and $t \in
(0,\infty)$. Suppose $A\Ge 0$. Then
   \begin{enumerate}
   \item[(i)] $C_A$ is selfadjoint and $C_A \Ge 0$,
   \item[(ii)] $C_A^{t} = C_{A^{t}}$,
   \item[(iii)] $\dz{C_{A^{t}}} \subseteq \dz{C_{A^{s}}}$
for every $s \in (0,t)$.
   \end{enumerate}
   \end{thm}
   \begin{proof}
(i) According to Corollary \ref{asp1}, $A^{\odot n}
\Ge 0$ and consequently $\big(A^{\odot n}\big)^{t} =
(A^{t})^{\odot n}$ for all $n\in \zbb_+$. By Theorem
\ref{gcd3}, $C_A$ is selfadjoint and $C_A \Ge 0$.

(ii) Fix a conjugation $Q$ on $\hh$. It follows from
Proposition \ref{abcon}(vi) that $\varXi_Q(A) \Ge 0$.
Applying Theorem \ref{qaq}(v), we deduce that
$\varXi_{Q}(A)^{t}=\varXi_{Q}(A^{t})$. Since
   \begin{align*}
\big(\bigoplus_{\omega \in \varOmega}
S_{\omega}\big)^{t} = \bigoplus_{\omega \in \varOmega}
S_{\omega}^{t},
   \end{align*}
whenever $\{S_{\omega}\}_{\omega \in \varOmega}$ is a
family of positive selfadjoint operators, we deduce
that
   \allowdisplaybreaks
   \begin{align*}
C_A^{t} \overset{\eqref{import}}=
\Big(U_{\varPhi,Q}^{-1}\varGamma_{\varPhi}(\varXi_{Q}(A))
U_{\varPhi,Q}\Big)^{t} & =
U_{\varPhi,Q}^{-1}\big(\varGamma_{\varPhi}(\varXi_{Q}(A))\big)^{t}
U_{\varPhi,Q}
      \\
& = U_{\varPhi,Q}^{-1}
\varGamma_{\varPhi}(\varXi_{Q}(A)^{t}) U_{\varPhi,Q}
   \\
&= U_{\varPhi,Q}^{-1}
\varGamma_{\varPhi}(\varXi_{Q}(A^{t})) U_{\varPhi,Q}
   \\
& \hspace{-1ex}\overset{\eqref{import}}= C_{A^{t}}.
   \end{align*}

(iii) This is a direct consequence of (ii) and
\cite[Lemma A.1]{St-Seb}.
   \end{proof}
As shown below, the polar decomposition of $C_A$ can
be explicitly written in terms of the polar
decomposition of $A^*$.
   \begin{thm} \label{PolDec}
Suppose $\varPhi \in \fscr$ and $A \in \ogr{\hh}$. Let
$A=U |A|$ be the polar decomposition of $A$. Then $C_A
= C_{U} C_{|A^*|}$ is the polar decomposition of
$C_A$. In particular, $|C_A| = C_{|A^*|}$.
   \end{thm}
   \begin{proof}
In view of Corollary \ref{partis}, $C_U$ is a partial
isometry. It follows from Proposition \ref{stale2}(i),
Lemma \ref{lem1}(i), \cite[Theorem 5.39]{Weid} and
Theorem \ref{wnoca}(i) that $C_A^* C_A$ is positive
and selfadjoint, and
   \begin{align*}
C_A^* C_A = C_{A^*} C_A \subseteq C_{|A^*|^2} =
(C_{|A^*|^2})^*.
   \end{align*}
Hence, by maximality of selfadjoint operators, $C_A^*
C_A = C_{|A^*|^2}$. Applying Theorem
\ref{alpharoot}(ii), we see that $|C_A| = C_{|A^*|}$.

Since $A^* = U^*|A^*|$ is the polar decomposition of
$A^*$, we get $A = |A^*| U$. Hence
   \begin{align} \label{cacas}
C_U C_{|A^*|} \subseteq C_A.
   \end{align}
However, $\dz{C_{|A^*|}} = \dz{|C_{A}|}=\dz{C_A}$, and
thus by \eqref{cacas}, $C_U C_{|A^*|} = C_A$. It
remains to show that $\jd{C_U} = \jd{C_{A}}$. To prove
the last equality, observe that for $f\in
\varPhi(\hh)$, $f\in \jd{C_U}$ if and only if $f$
vanishes on $\ob{U}$, or equivalently, by the equality
$\ob{U} = \overline{\ob{A}}$, if and only if $f \in
\jd{C_{A}}$.
   \end{proof}
   Theorem \ref{PolDec} can be used do describe the
generalized Aluthge transform of $C_{A}$. The Aluthge
transform was first studied in \cite{Alu} and has been
studied extensively since, mostly in the context of
$p$-hyponormal operators (see e.g.\ \cite{J-K-P}).
Given a complex Hilbert space $\hh$ and an operator $T
\in \ogr{\hh}$ with the polar decomposition $T=U|T|$,
we write
   \begin{align*}
\alu_{s,t}(T) = |T|^{s} U |T|^{t}, \quad s, t \in
(0,\infty).
   \end{align*}
We call $\alu_{s,t}(T)$ the {\em $(s,t)$-Aluthge
transform} of $T$. Clearly, $\alu_{s,t}(T) \in
\ogr{\hh}$. In case $s=t=\frac{1}{2}$, we get the
usual Aluthge transform of $T$.
   \begin{thm}
Let $\varPhi \in \fscr$ and $A \in \ogr{\hh}$ be such
that $C_{A} \in \ogr{\varPhi(\hh)}$ and let $s,t\in
(0,\infty)$. Then
$\alu_{s,t}(C_{A})=C_{\alu_{s,t}(A^*)}^*$.
   \end{thm}
   \begin{proof}
Let $A=U |A|$ be the polar decomposition of $A$. By
Theorem \ref{PolDec}, $C_A = C_{U} C_{|A^*|}$ is the
polar decomposition of $C_A$. Hence, by Theorem
\ref{alpharoot}, we have
   \begin{align} \label{alu1}
\alu_{s,t}(C_{A}) = C_{|A^*|}^{s} C_{U} C_{|A^*|}^{t}
= C_{|A^*|^{s}} C_{U} C_{|A^*|^{t}} =
C_{|A^*|^{t}U|A^*|^{s}}.
   \end{align}
Since $A^*=U^*|A^*|$ is the polar decomposition of
$A^*$, we deduce from Theorem \ref{wnoca} that
   \begin{align}  \label{alu2}
C_{|A^*|^{t}U|A^*|^{s}} = C_{|A^*|^{s}U^*|A^*|^{t}}^*
= C_{\alu_{s,t}(A^*)}^*.
   \end{align}
Combining \eqref{alu1} and \eqref{alu2} completes the
proof.
   \end{proof}
   \section{\label{sekt10}Seminormality of $C_A$ and related
questions}
   The following theorem, which is interesting in
itself, will be used to characterize seminormality of
composition operators $C_A$.
   \begin{thm} \label{nier}
If $\varPhi \in \fscr$ and $A,B \in \ogr{\hh}$, then
the following conditions are equivalent{\em :}
   \begin{enumerate}
   \item[(i)] $\dz{C_B} \subseteq \dz{C_A}$ and $\|C_A f\|
\Le \|C_B f\|$ for all $f \in \dz{C_B}$,
   \item[(ii)] $\|C_A f\|  \Le \|C_B f\|$ for all
$f \in \kscr^{\varPhi}$,
   \item[(iii)] $\|A^*\xi\| \Le \|B^*\xi\|$ for all $\xi  \in
\hh$.
   \end{enumerate}
   \end{thm}
   \begin{proof}
(i)$\Rightarrow$(ii) Obvious (due to Lemma
\ref{lem1}(i)).

(ii)$\Rightarrow$(i) This can be deduced from Theorem
\ref{wnoca}(ii) and Proposition \ref{stale2}(i).

(ii)$\Rightarrow$(iii) It follows from \eqref{rep2}
and Lemma \ref{lem1}(i) that
   \begin{align*}
\varPhi(\|A^*\xi\|^2) = \|C_A K_\xi^{\varPhi}\|^2 \Le
\|C_B K_\xi^{\varPhi}\|^2 = \varPhi(\|B^*\xi\|^2),
\quad \xi \in \hh.
   \end{align*}
Hence, by the strict monotonicity of $\varPhi$, (iii)
holds.

(iii)$\Rightarrow$(i) Fix a conjugation $Q$ on $\hh$,
and set $E=\varXi_Q(A^*)$ and $F=\varXi_Q(B^*)$ (see
Appendix \ref{apap}). It follows from (iii) that $\|E
\xi\| \Le \|F \xi\|$ for all $\xi \in \hh$. Hence
$E^*E \Le F^*F$. Applying \cite[Proposition 2.2]{js2},
we deduce that
   \begin{align*}
(E^*E)^{\otimes n} \Le (F^*F)^{\otimes n}, \quad n \in
\zbb_+.
   \end{align*}
This implies that
   \begin{align*}
\|E^{\odot n} \xi\|^2 = \is{(E^*E)^{\otimes n}
\xi}{\xi} \Le \is{(F^*F)^{\otimes n} \xi}{\xi}=
\|F^{\odot n} \xi\|^2, \quad \xi \in \hh^{\odot n}, \;
n \in \zbb_+.
   \end{align*}
As a consequence, we see that
$\dz{\varGamma_{\varPhi}(F)} \subseteq
\dz{\varGamma_{\varPhi}(E)}$ and
$\|\varGamma_{\varPhi}(E) \xi\| \Le
\|\varGamma_{\varPhi}(F) \xi\|$ for all $\xi \in
\dz{\varGamma_{\varPhi}(F)}$. Applying Theorems
\ref{fmod} and \ref{wnoca}(i) completes the proof.
   \end{proof}
Regarding Theorem \ref{nier}, we note that if
$\dz{C_B} \subseteq \dz{C_A}$, then there exists
$\alpha \in \rbb_+$ such that $\|C_A f\| \Le \alpha
(\|f\| + \|C_B f\|)$ for all $f \in \dz{C_B}$ (use
Proposition \ref{stale2}(i) and the closed graph
theorem).

The question of seminormality of $C_A$ can now be
fully answered (see also Corollary \ref{sa-sa} for the
characterization of selfadjoint composition operators
of the form $C_A$).
   \begin{thm} \label{chca}
If $\varPhi \in \fscr$ and $A \in \ogr{\hh}$, then the
following conditions are equivalent{\em :}
   \begin{enumerate}
   \item[(i)] $C_A$ is cohyponormal $($resp.,\ hyponormal\/$)$,
   \item[(ii)] $A$ is hyponormal $($resp.,\ cohyponormal\/$)$.
   \end{enumerate}
   \end{thm}
   \begin{proof}
Apply Theorem \ref{wnoca}(i) and Theorem \ref{nier} to
the pairs $(A,A^*)$ and $(A^*,A)$, respectively.
   \end{proof}
As a consequence of the above considerations, we
obtain the following characterizations of
hyponormality, normality, unitarity, isometricity and
coisometricity of $C_{A+b}$ (see \cite[Proposition
5.1]{LeT} for a particular case of Corollary \ref{wn3}
where $\varPhi=\exp$ and the composition operator
$C_{\varphi}$ is assumed to be bounded).
   \begin{cor} \label{wn3}
If $\varPhi \in \fscr$, $A \in \ogr{\hh}$ and $b\in
\hh$, then $C_{A+b}$ is hyponormal $($resp.,\ normal,
unitary, isometric, coisometric$)$ if and only if
$b=0$ and $A$ is cohyponormal $($resp.,\ normal,
unitary, coisometric, isometric$)$.
   \end{cor}
   \begin{proof}
The characterizations of hyponormality, normality,
unitarity and isometricity of $C_{A+b}$ follow easily
from Corollaries \ref{hypo} and \ref{unit} and Theorem
\ref{chca}.

Suppose now that $C_{A+b}$ is a coisometry. Then, by
Theorem \ref{sprz}(i), we have
   \begin{align*}
\varPhi(0) = K^{\varPhi}_0(0) = (C_{A+b}C_{A+b}^*
K^{\varPhi}_0)(0) = (C_{A+b} K^{\varPhi}_{b})(0) =
K^{\varPhi}_{b}(b) = \varPhi(\|b\|^2),
   \end{align*}
which implies that $b=0$. Hence, the characterization
of coisometricity of $C_{A+b}$ follows directly from
Corollary \ref{unit}.
   \end{proof}
In the next corollary, we restrict our attention to
operators $A\in \ogr{\hh}$ such~ that
   \begin{align} \label{klop}
   \begin{minipage}{70ex} {\em
either the planar measure of the spectrum of $A$ is
equal to $0$ or one of the operators $\mathfrak{Re} A$
or $\mathfrak{Im} A$ is compact.}
   \end{minipage}
   \end{align}
Clearly, any compact operator $A\in \ogr{\hh}$
satisfies both requirements of \eqref{klop}.
   \begin{cor}\label{dim-fin}
Suppose $\varPhi \in \fscr$ and $A \in \ogr{\hh}$
satisfies \eqref{klop}. Then $C_A$ is seminormal if
and only if $A$ is normal.
   \end{cor}
   \begin{proof}
Apply Theorem \ref{chca}, the fact that bounded
seminormal operators with compact real or imaginary
parts are normal (cf.\ \cite[Problem 207]{Hal}), and
Putnam's inequality (cf.\ \cite[Theorem 1]{Put}).
   \end{proof}
   If $\varPhi^{\prime}(0) \neq 0$, then we can also
characterize subnormal composition operators.
   \begin{pro} \label{subnwkw}
Suppose $\varPhi \in \fscr$ is such that
$\varPhi^{\prime}(0) \neq 0$. If $A \in \ogr{\hh}$,
then $C_{A}$ $($resp.,\ $C_{A}^*$$)$ is subnormal if
and only if $A^*$ $($resp.,\ $A$$)$ is subnormal.
   \end{pro}
   \begin{proof}
This is a direct consequence of Theorems \ref{fmod},
\ref{wnoca}(i) and \ref{qaq}(i), and the fact that if
$T\in \ogr{\hh}$ is subnormal, then $T^{\odot n}$ is
subnormal for every $n\in \zbb_+$ (cf.\ \cite[Theorem
2.4]{js2}).
   \end{proof}
It is an open question as to whether the ``only if''
part of Proposition \ref{subnwkw} remains true if
$\varPhi^{\prime}(0)=0$ (the ``if'' part is true
independently of whether $\varPhi^{\prime}(0)=0$ or
not). If $\varPhi(z)=z^2$ for $z\in \cbb$, then the
question reduces to that in \cite[p.\ 139]{js2}.

Below we give necessary and sufficient conditions for
positive operators $A, B \in \ogr{\hh}$ to satisfy the
inequality $C_A \preccurlyeq C_B$.
   \begin{thm} \label{wn2}
Let $\varPhi \in \fscr$ and let $A, B \in \ogrp{\hh}$.
Then the following conditions are equivalent{\em :}
   \begin{enumerate}
   \item[(i)] $C_A \preccurlyeq C_B$,
   \item[(ii)] $\is{C_A f}{f} \Le \is{C_B f}{f}$ for all
$f \in \kscr^{\varPhi}$,
   \item[(iii)] $A \Le B$.
   \end{enumerate}
   \end{thm}
   \begin{proof}
By Theorem \ref{gcd3} both operators $C_A$ and $C_B$
are positive and selfadjoint. It follows from Lemma
\ref{lem1}(i) and Theorem \ref{wnoca}(i) that
$\kscr^{\varPhi}$ is invariant for $C_{A^{1/2}}$ and
$C_{A^{1/2}}^* C_{A^{1/2}}f = C_A f$ for all $f\in
\kscr^{\varPhi}$. This implies that $\|C_{A^{1/2}}
f\|^2 = \is{C_A f}{f}$ for all $f \in
\kscr^{\varPhi}$. The same is true for $B$. Hence,
applying Theorems \ref{alpharoot} and \ref{nier}
completes the proof.
   \end{proof}
We conclude this section by calculating the limit of a
net of composition operators which are orthogonal
projections.
   \begin{pro} \label{suport}
If $\varPhi \in \fscr$ and $\pcal \subseteq \ogr{\hh}$
is an upward-directed partially ordered set of
orthogonal projections, then
   \begin{align*}
\lim_{P\in \pcal} C_Pf = C_{Q}f, \quad f \in
\varPhi(\hh),
   \end{align*}
where $Q$ is the orthogonal projection of $\hh$ onto
$\bigvee_{P\in \pcal} \ob{P}$.
   \end{pro}
   \begin{proof}
It follows from Corollary \ref{opr} and Theorem
\ref{wn2} that $\{C_P\}_{P\in \pcal}$ is a
monotonically increasing net of orthogonal
projections. Hence, $\{C_P\}_{P\in \pcal}$ converges
in the strong operator topology to the orthogonal
projection $T$ of $\varPhi(\hh)$ onto $\bigvee_{P\in
\pcal} \ob{C_P}$ (cf.\ \cite[Theorem 4.3.4]{Mlak}). If
$P\in \pcal$, then $P \Le Q$ and thus $C_P \Le C_{Q}$,
which implies that $\ob{T}=\bigvee_{P\in \pcal}
\ob{C_P} \subseteq \ob{C_Q}$. Take $f \in \ob{C_Q}
\ominus \ob{T}$. Then $f = C_Q g$ for some $g \in
\varPhi(\hh)$. Since $QP=P$ for every $P \in \pcal$,
we have
   \begin{align*}
g(P\xi)= g(QP\xi) \overset{\eqref{rep}} = \is{C_Q
g}{K^{\varPhi}_{P\xi}} \overset{\eqref{op*}}=
\is{f}{C_P K^{\varPhi}_{\xi}} = 0, \quad P \in \pcal,
\, \xi \in \hh,
   \end{align*}
and thus $g(\eta)=0$ for every $\eta \in
\overline{\bigcup_{P\in \pcal} \ob{P}}$. As
$\bigcup_{P\in \pcal} \ob{P}$ is a vector space
(because $\pcal$ is upward-directed), we see that
$g|_{\ob{Q}}=0$, which yields $f=C_Q g =0$. This
implies that $\ob{C_Q} = \ob{T}$ and thus $C_Q = T$,
which completes the proof.
   \end{proof}
   \section{\label{sekt11}Boundedness and seminormality
of $C_{A+b}$ in $\exp(\hh)$ (necessity)}
   From now on, we abbreviate $K^{\exp}$ and
$\kscr^{\exp}$ to $K$ and $\kscr$ respectively. The
corresponding reproducing kernel Hilbert space
$\exp(\hh)$ is called the Segal-Bargmann space of
order $\dim \hh$ (see e.g., \cite{jerz,jerz2}).

We begin by providing explicit formulae for
$C_{\varphi} K_{\xi}$ and $C_{\varphi}^*$ if $\varphi$
is affine. Let us note that part (iii) of Lemma
\ref{cfkz} appeared in \cite[Lemma 2]{c-m-s03} under
the assumption that $\hh$ is finite dimensional and
$C_{\varphi}$ is bounded.
   \begin{lem} \label{cfkz}
Suppose $\varphi = A + b$ with $A \in \ogr{\hh}$ and
$b \in \hh$. Then
   \begin{enumerate}
   \item[(i)] $\kscr \subseteq \dz{C_{\varphi}}
\cap \dz{C_{\varphi}^*}$,
   \item[(ii)] $C_{\varphi} K_{\xi} = \exp(\is{b}{\xi})
K_{A^*\xi}$ for every $\xi \in \hh$,
   \item[(iii)] $(C_{\varphi}^* f)(\xi) = \exp(\is{\xi}{b})
f(A^* \xi)$ for all $\xi \in \hh$ and $f\in
\dz{C_{\varphi}^*}$,
   \item[(iv)] $(C_{\varphi}^*C_{\varphi} f)(\xi) =
\exp(\is{\xi}{b}) f(AA^* \xi + b)$ for all $\xi \in \hh$
and $f \in \dz{C_{\varphi}^*C_{\varphi}}$.
   \end{enumerate}
   \end{lem}
   \begin{proof}
As $(K_{\xi} \circ \varphi) (\eta) = \exp(\is{b}{\xi})
K_{A^*\xi}(\eta)$ for all $\xi, \eta \in \hh$, we
infer from Theorem \ref{sprz}(i) that (i) and (ii)
hold. By (i) and Theorem \ref{sprz}(i), we have
$(C_{\varphi}^* K_\eta)(\xi) = K_b(\xi) K_{\eta}(A^*
\xi)$ for all $\xi, \eta\in \hh$. Since $\kscr$ is a
core for $C_{\varphi}^*$ (cf.\ Theorem \ref{sprz}(i))
and the norm convergence implies the pointwise one in
$\exp(\hh)$, we obtain (iii). The condition (iv) is a
direct consequence of (iii).
   \end{proof}
The next result is of some general interest (see also
\cite{LeT} for a recent independent approach in the
case of bounded operators). Before stating it, we
define the generalized inverse of a (not necessarily
bounded) selfadjoint operator $B$ in a complex Hilbert
space $\hh$. Since $\overline{\ob{B}}$ reduces $B$ and
$B_1:=B|_{\overline{\ob{B}}}$ is a selfadjoint
operator in $\overline{\ob{B}}$ such that $\jd{B_1} =
\{0\}$ and $\ob{B}=\ob{B_1}$, we can define $B^{-1} =
B_1^{-1}$. The generalized inverse $B^{-1}$ is a
selfadjoint operator in $\overline{\ob{B}}$. If
moreover $B\Ge 0$, then we set $B^{-1/2} =
(B^{1/2})^{-1}$. Clearly, $B^{-1/2} = (B^{-1})^{1/2}$.
   \begin{lem} \label{range}
Let $B$ be a selfadjoint operator in a complex Hilbert
space $\hh$, $e \in \hh$ and $c \in \rbb$. Then the
following conditions are equivalent{\em :}
   \begin{enumerate}
   \item[(i)] $\is{B \xi}{\xi} - 2 \mathfrak{Re}
\is{\xi}{e} + c \Ge 0$ for all $\xi\in \dz{B}$,
   \item[(ii)] $B \Ge 0$, $e \in \ob{B^{1/2}}$
and $\|B^{-1/2} e \|^2 \Le c$.
   \end{enumerate}
Moreover, $\|B^{-1/2} e \|^2$ is the least constant
$c$ for which {\em (i)} holds.
   \end{lem}
   \begin{proof}
(i)$\Rightarrow$(ii) Since (i) holds for $\xi=0$, we
see that $c\Ge 0$. Substituting $t\xi$ in place of
$\xi$ into (i), we get
   \begin{align*}
t^2\is{B \xi}{\xi} - 2 t \mathfrak{Re} \is{\xi}{e} + c
\Ge 0, \quad t \in \rbb, \, \xi\in \dz{B}.
   \end{align*}
This implies that $B \Ge 0$ and $(\mathfrak{Re}
\is{\xi}{e})^2 \Le c \is{B \xi}{\xi}$ for all $\xi \in
\dz{B}$. Next, substituting $z \xi$ in place of $\xi$
into the last inequality with appropriate $z \in
\cbb$, we get
   \begin{align} \label{nier2}
|\is{\xi}{e}| \Le \sqrt{c} \, \|B^{1/2} \xi\|, \quad
\xi \in \dz{B}.
   \end{align}
Since $\dz{B}$ is a core for $B^{1/2}$ (cf.\
\cite[Theorem 4.5.1]{b-s}), we deduce that
\eqref{nier2} holds for all $\xi \in \dz{B^{1/2}}$.
This implies that there exists a continuous linear
functional $\varLambda\colon \overline{\ob{B^{1/2}}}
\to \cbb$ such that $\|\varLambda\| \Le \sqrt{c}$ and
$\varLambda(B^{1/2}\xi) = \is{\xi}{e}$ for all $\xi
\in \dz{B^{1/2}}$. By the Riesz representation
theorem, there exists a vector $\eta \in
\overline{\ob{B^{1/2}}}$ such that $\|\eta\| =
\|\varLambda\|$ and $\varLambda(B^{1/2}\xi) =
\is{B^{1/2}\xi}{\eta}$ for all $\xi\in \dz{B^{1/2}}$.
Hence $\|\eta\| \Le \sqrt{c}$ and $\is{\xi}{e} =
\is{B^{1/2}\xi}{\eta}$ for all $\xi\in \dz{B^{1/2}}$.
Since $B^{1/2}$ is selfadjoint, we see that $\eta \in
\dz{B^{1/2}}$ and $e = B^{1/2} \eta$. This implies
(ii).

(ii)$\Rightarrow$(i) It follows from (ii) that
   \begin{align*}
\is{B \xi}{\xi} - 2 \mathfrak{Re} \is{\xi}{e} +
\|B^{-1/2} e \|^2 & = \|B^{1/2} \xi\|^2 - 2
\mathfrak{Re} \is{B^{1/2}\xi}{B^{-1/2}e} + \|B^{-1/2}
e \|^2
   \\
& = \|B^{1/2} \xi - B^{-1/2}e\|^2 \Ge 0, \quad \xi \in
\dz{B},
   \end{align*}
which completes the proof.
   \end{proof}
Below we provide necessary conditions for the
boundedness of $C_{A+b}$ in $\exp(\hh)$. It turns out
that they are also sufficient (cf.\ Theorem
\ref{glowne2}).
   \begin{lem}\label{n&sc4b}
Suppose $C_{\varphi} \in \ogr{\exp(\hh)}$, where
$\varphi = A + b$ with $A \in \ogr{\hh}$ and $b \in
\hh$. Then
   \begin{enumerate}
   \item[(i)] $\|A\| \Le 1$ and $b \in \ob{(I-AA^*)^{1/2}}$,
   \item[(ii)] $\exp(\|(I-AA^*)^{-1/2} b\|^2) \Le
\|C_{\varphi}\|^2$.
   \end{enumerate}
   \end{lem}
   \begin{proof}
It follows from Proposition \ref{nonzero}(i) that
$\|C_{\varphi}\| > 0$. By \eqref{rep2} and Lemma
\ref{cfkz}(ii), we have
   \begin{align*}
\exp(2 \mathfrak{Re}\is{\xi}{b} + \|A^*\xi\|^2) =
\|C_{\varphi} K_{\xi}\|^2 \Le \|C_{\varphi}\|^2
\exp(\|\xi\|^2), \quad \xi \in \hh,
   \end{align*}
which yields
   \begin{align*}
\is{(I-AA^*) \xi}{\xi} - 2 \mathfrak{Re} \is{\xi}{b} +
2 \log \|C_{\varphi}\| \Ge 0, \quad \xi\in \hh.
   \end{align*}
   This, combined with Lemma \ref{range}, implies (i)
and (ii).
   \end{proof}
   \begin{cor}\label{n&sc4b-cor}
Under the assumptions of Lemma {\em \ref{n&sc4b}}, the
following holds{\em :}
   \begin{enumerate}
   \item[(i)] if $\xi  \in \hh$ is such that $\|A^*\xi\|
= \|\xi\|$, then $\xi \in \{b\}^{\perp}$,
   \item[(ii)] $\jd{\lambda I - A^*} \subseteq \{b\}^{\perp}$
for every $\lambda \in \mathbb T$, where $\mathbb T =
\{\lambda \in \cbb \colon |\lambda|=1\}$,
\item[(iii)] if $b \in \jd{\lambda I - A^*}$ for some
$\lambda \in \mathbb T$, then $b=0$.
   \end{enumerate}
   \end{cor}
   \begin{proof}
   (i) By Lemma \ref{n&sc4b}(i), $\|A\| \Le 1$ and
thus the equality $\|A^*\xi\| = \|\xi\|$ holds if and
only if $\xi \in \jd{(I-AA^*)^{1/2}}$. Since
$\jd{(I-AA^*)^{1/2}} \perp \ob{(I - AA^*)^{1/2}}$ and
$b \in \ob{(I-AA^*)^{1/2}}$ (cf.\ Lemma
\ref{n&sc4b}(i)), we get (i).

   The conditions (ii) and (iii) follow from (i).
   \end{proof}
Now we give necessary conditions for the
cohyponormality of $C_{A+b}$ in $\exp(\hh)$.
   \begin{pro}\label{cohyp}
Suppose $C_{\varphi} \in \ogr{\exp(\hh)}$ and $\lambda
\in \mathbb T$, where $\varphi = A + b$ with $A \in
\ogr{\hh}$ and $b \in \hh$. Then
   \begin{enumerate}
   \item[(i)] if $C_{\varphi}$ is cohyponormal, then
$A$ is hyponormal, $(I-A^*)b \in \ob{[A^*,A]^{1/2}}$
and $\|[A^*,A]^{-1/2} (I-A^*)b\| \Le \|b\|$, where
$[A^*,A]:=A^*A - AA^*$,
   \item[(ii)]
$\big($$C_{\varphi}$ is cohyponormal and $b \in
\jd{\lambda I - A^*}$$\big)$ $\Leftrightarrow$ $($$A$
is hyponormal and $b=0$$)$,
   \item[(iii)] if $A$ satisfies \eqref{klop}, then
$C_{\varphi}$ is seminormal\/\footnote{\;Consult
Corollary \ref{wn3}.} if and only if both $A$ is
normal and $b=0$, or equivalently, if and only if
$C_{\varphi}$ is normal.
   \end{enumerate}
   \end{pro}
   \begin{proof}
(i) Since $C_{\varphi}$ is cohyponormal, we infer from
Lemma \ref{cfkz} that
   \begin{align*}
\exp(\|\varphi(\xi)\|^2) \overset{\eqref{rep2}}=
\|K_{\varphi(\xi)}\|^2 & \overset{\eqref{op*}} =
\|C_{\varphi}^* K_{\xi}\|^2
   \\
& \hspace{1ex}\Ge \|C_{\varphi} K_{\xi}\|^2
   \\
& \hspace{1ex} = |\exp(\is{b}{\xi})|^2 \exp(\|A^*\xi\|^2)
   \\
& \hspace{1ex} = \exp(2 \mathfrak{Re}\is{b}{\xi} +
\|A^*\xi\|^2), \quad \xi \in \hh,
   \end{align*}
which leads to
   \begin{align*}
\is{[A^*,A] \xi}{\xi} - 2 \mathfrak{Re}
\is{\xi}{(I-A^*)b} + \|b\|^2 \Ge 0, \quad \xi \in \hh.
   \end{align*}
This, together with Lemma \ref{range}, yields (i).

(ii) Apply Corollary \ref{n&sc4b-cor}(iii) and Theorem
\ref{chca}.

   (iii) Assume $A$ satisfies \eqref{klop}. Suppose
first that $C_{\varphi}$ is cohyponormal. Then, by
(i), $A$ is hyponormal. If the planar measure of the
spectrum of $A$ equals $0$, then by Putnam's
inequality (cf.\ \cite[Theorem 1]{Put}) $A$ is normal.
In turn, if one of the operators $\mathfrak{Re} A$ or
$\mathfrak{Im} A$ is compact, then by \cite[Problem
207]{Hal} $A$ is again normal. In both cases, (i)
implies that $b \in \jd{I - A^*}$. This and (ii) yield
$b=0$. Suppose now that $C_{\varphi}$ is hyponormal.
Then, by Corollary \ref{hypo}, $b=0$. Applying
Corollary \ref{dim-fin} we deduce that $A$ is normal.
Finally, if $A$ is normal and $b=0$, then Corollary
\ref{wn3} implies that $C_{\varphi}$ is normal. This
completes the proof.
   \end{proof}
   Using Proposition \ref{cohyp}(i) and Corollary
\ref{n&sc4b-cor}(iii), we get the following.
   \begin{cor} \label{b=0}
Suppose $C_{\varphi} \in \ogr{\exp(\hh)}$, where
$\varphi = A + b$ with $A \in \ogr{\hh}$ and $b \in
\hh$. If $C_{\varphi}$ is cohyponormal and $b\neq 0$,
then $A$ is not normal.
   \end{cor}
In view of Propositions \ref{pro1} and
\ref{cohyp}(iii), if $\hh$ is finite dimensional, then
bounded seminormal composition operators $C_{\varphi}$
on $\exp(\hh)$ are always normal. This is no longer
true if $\dim \hh = \infty$ (see Example \ref{optim}).
Note also that the situation described in Corollary
\ref{b=0} can really happen (see Example \ref{optim}).
   \section{\label{sec9}Composition operators in Segal-Bargmann
spaces of finite order}
   Given $d \in \nbb$, we denote by $\mu_d$ the Borel
probability measure on $\cbb^d$ defined by
   \begin{align*}
\mu_d(\varDelta) = \frac{1}{\pi^d} \int_{\varDelta}
\exp(-\|\xi\|^2) \D V_d(\xi), \quad \varDelta\text{ -- a
Borel subset of } \cbb^d,
   \end{align*}
where $\|\xi\|^2 = |\xi_1|^2 + \ldots + |\xi_d|^2$ for
$\xi = (\xi_1, \ldots, \xi_d) \in \cbb^d$ and $V_d$
stands for the $2d$-dimensional Lebesgue measure on
$\cbb^d$. Denote by $\bb_d$ the Segal-Bargmann space
of order $d$, i.e.,
   \begin{align*}
\bb_d = \{f \colon f \text{ is an entire function on }
\cbb^d \text{ and } f \in L^2(\mu_d)\}.
   \end{align*}
It is well-known that $\bb_d = \exp(\cbb^d)$, where
$\cbb^d$ is equipped with the standard inner product
(cf.\ \cite[Section 1C]{Barg})
   \begin{align*}
\text{$\is{\xi}{\eta} = \xi_1\bar \eta_1 + \ldots +
\xi_d\bar \eta_d$ for $\xi = (\xi_1, \ldots, \xi_d) \in
\cbb^d$ and $\eta = (\eta_1, \ldots, \eta_d) \in \cbb^d$.}
   \end{align*}

We begin by calculating the norm of $C_{\varphi}$ in
the case of $d=1$ and $|A|<1$. First we give an upper
estimate for $\|C_{\varphi}\|$ (as mentioned in the
proof Lemma \ref{sb4}, the operator $D$ appearing in
Lemma \ref{sb2} below is of the form
$D=C_{\varphi}^*C_{\varphi}$ with $\alpha=|A|^2$,
where $\varphi(z) = A z + b$ for $z \in \cbb$).
   \begin{lem} \label{sb2}
Fix $\alpha \in [0,1)$ and $b \in \cbb$. Let $D$ be an
operator in $\bb_1$ given by
   \begin{align*}
(Df)(z) = f(\alpha z + b) \exp(z \bar b), \quad z \in \cbb,
\, f \in \bb_1.
   \end{align*}
Then $D \in \ogr{\bb_1}$ and
   \begin{align*}
\|D\| \Le
\frac{\exp\big(\frac{|b|^2}{1-\alpha}\big)}{\sqrt{1-\alpha^2}}.
   \end{align*}
   \end{lem}
   \begin{proof}
By \eqref{rep} and \eqref{rep2}, $|f(z)|^2 \Le \|f\|^2
\E^{|z|^2}$ for all $z \in \cbb$ and $f\in \bb_1$.
This and the fact that the Lebesgue measure $V_1$ is
translation-invariant yield
   \allowdisplaybreaks
   \begin{align*}
\pi \int_{\cbb} |Df|^2 \D\mu_1 & = \int_{\cbb}
|f(\alpha z + b)|^2 \E^{2 \mathfrak{Re} (z \bar b)}
\E^{-|z|^2} \D V_1(z)
   \\
& \Le \|f\|^2 \int_{\cbb} \E^{|\alpha z + b|^2 + 2
\mathfrak{Re} (z \bar b) - |z|^2} \D V_1(z)
   \\
& = \|f\|^2 \exp\Big(\frac{2|b|^2}{1-\alpha}\Big)
\int_{\cbb} \E^{-(1-\alpha^2) \left |z -
\frac{b}{1-\alpha}\right |^2} \D V_1(z)
   \\
& = \|f\|^2 \exp\Big(\frac{2|b|^2}{1-\alpha}\Big)
\int_{\cbb} \E^{-(1-\alpha^2) |z|^2} \D V_1(z)
   \\
& = \pi \|f\|^2
\frac{\exp\Big(\frac{2|b|^2}{1-\alpha}\Big)}{ 1-\alpha^2},
\quad f \in \bb_1,
   \end{align*}
which completes the proof.
   \end{proof}
The following lemma can be proved by a simple induction
argument.
   \begin{lem} \label{sb3}
If $D$ is as in Lemma {\em \ref{sb2}}, then
   \begin{align*}
(D^n f)(z) = f\Big(\alpha^n z + b_n \Big) \E^{z \bar b_n}
\exp\Big(\frac{|b|^2}{1-\alpha}\Big(n-1 -
\frac{\alpha-\alpha^{n}}{1-\alpha}\Big)\Big),
   \end{align*}
for all $z \in \cbb$, $f \in \bb_1$ and $n \in \nbb $,
where $b_n = \frac{1-\alpha^n}{1-\alpha} b$ for $n \in
\nbb$.
    \end{lem}
Now we can calculate the norm of $C_{\varphi}$ when $d=1$
and $|A|<1$.
   \begin{lem} \label{sb4}
Let $A \in \cbb$ be such that $|A|<1$ and let $b \in \cbb$.
Set $\varphi(z) = Az +b$ for $z\in \cbb$. Then $C_{\varphi}
\in \ogr{\bb_1}$ and
   \begin{align} \label{cfnorma}
\|C_{\varphi}\|^2 = \exp\Big(\frac{|b|^2}{1-|A|^2}\Big).
   \end{align}
   \end{lem}
   \begin{proof}
By Proposition \ref{stale2}(i) and Lemma
\ref{cfkz}(i), the operator $C_{\varphi}^*C_{\varphi}$
is closed and densely defined. Applying Lemma
\ref{cfkz}(iv) and Lemma \ref{sb2} with
$\alpha=|A|^2$, we deduce that
$C_{\varphi}^*C_{\varphi}=D$ is bounded. Hence
$C_{\varphi} \in \ogr{\bb_1}$.

Now we show that \eqref{cfnorma} holds. Since $D$ is
selfadjoint, $\|D\|=r(D)$. Applying Lemmata \ref{sb2} and
\ref{sb3}, we deduce that for every $n\in \nbb$,
   \begin{align*}
\|D^n\|^{1/n} \Le \frac{1}{(1-\alpha^{2n})^{1/2n}}
\exp\Big(\frac{|b|^2(1-\alpha^{n})}{n(1-\alpha)^2}\Big)
\exp\Big(\frac{|b|^2}{1-\alpha}\Big(\frac{n-1}{n} -
\frac{\alpha-\alpha^{n}}{n(1-\alpha)}\Big)\Big).
   \end{align*}
Using Gelfand's formula for the spectral radius, we
conclude that
   \begin{align*}
\|C_{\varphi}\|^2 = \|D\| = r(D) \Le
\exp\Big(\frac{|b|^2}{1-|A|^2}\Big).
   \end{align*}
The reverse inequality follows from Lemma
\ref{n&sc4b}(ii).
   \end{proof}
We are now in a position to discuss the general
$d$-dimensional case.
   \begin{thm} \label{s-b-norm}
Let $\varphi\colon \cbb^d \to \cbb^d$ be a holomorphic
mapping $($$d \in \nbb$$)$. Then $C_{\varphi} \in
\ogr{\bb_d}$ if and only if there exist $A \in
\ogr{\cbb^d}$ and $b\in \cbb^d$ such that $\varphi = A
+ b$, $\|A\| \Le 1$ and $b \in \ob{I-AA^*}$. Moreover,
if $C_{\varphi} \in \ogr{\bb_d}$, then
   \begin{align} \label{norm}
1 \Le \|C_{\varphi}\|^2 = \exp(\is{(I-AA^*)^{-1}b}{b}).
   \end{align}
   \end{thm}
   \begin{proof}
By Proposition \ref{pro1}, there is no loss of
generality in assuming that $\varphi = A + b$, where
$A \in \ogr{\cbb^d}$ and $b \in \cbb^d$. Thus, the
``only if'' part follows from Lemma \ref{n&sc4b}(i)
(recall that $\ob{E^{1/2}} = \ob{E}$ whenever $E$ is a
bounded positive operator on a finite dimensional
Hilbert space). To prove the converse implication and
\eqref{norm}, we assume that $\|A\| \Le 1$ and $b \in
\ob{I-AA^*}$. If $\ob{I-AA^*}=\{0\}$, then $A$ is
unitary, and, by Corollary \ref{unit},
$C_{\varphi}=C_A$ is unitary and \eqref{norm} holds.
Suppose now that $\ob{I-AA^*} \neq \{0\}$, or
equivalently that $|A^*| \neq I$. Set $\psi =
|A^*|+b$. Since $|A^*|$ is a positive contraction,
there exist a monotonically increasing sequence
$\{t_n\}_{n=1}^d \subseteq [0,1]$ and an orthonormal
basis $\{g_n\}_{n=1}^d$ of $\cbb^d$ such that $|A^*|
g_n = t_n g_n$ for all $n \in \{1, \ldots, d\}$. Then
$\ob{I-AA^*}$ is the linear span of $\{g_n\}_{n=1}^k$
and thus $b=\sum_{n=1}^k \beta_n g_n$ with $\beta_n =
\is{b}{g_n}$ for $n\in \{1, \ldots, k\}$, where
$k=\max\{n\colon t_n < 1\}$. Define the unitary
operator $T \in \ogr{\cbb^d}$ and the holomorphic
mapping $\rho\colon \cbb^d \to \cbb^d$ by
$T\big(\sum_{n=1}^d \lambda_n g_n\big) = (\lambda_1,
\ldots, \lambda_d)$ and $\rho(\lambda_1, \ldots,
\lambda_d) = (t_1 \lambda_1 + \beta_1, \ldots, t_d
\lambda_d + \beta_d)$ for $(\lambda_1, \ldots,
\lambda_d) \in \cbb^d$ with $\beta_n=0$ for $n \Ge
k+1$. Clearly, by Theorem \ref{wnoca} and Corollary
\ref{unit}, $C_{T}$ is unitary and
$C_{T}^{*}=C_{T^{-1}}$. Since $\rho =T\circ \psi \circ
T^{-1}$, we deduce that $C_{\rho} = C_{T}^* C_{\psi}
C_{T}$. This means that $C_{\psi} \in \ogr{\bb_d}$ if
and only if $C_{\rho} \in \ogr{\bb_d}$, and if this is
the case, then $\|C_{\psi}\| = \|C_{\rho}\|$. Let
$\rho_n\colon \cbb \to \cbb$ be the polynomial given
by $\rho_n(\lambda)=t_n \lambda + \beta_n$ for
$\lambda\in \cbb$ and $n\in \{1,\ldots,d\}$. By Lemma
\ref{sb4}, $C_{\rho_n} \in \ogr{\bb_1}$ for every $n
\in \{1, \ldots, k\}$. Clearly, $C_{\rho_n} =
I_{\bb_1}$ for every $n \in \{k+1, \ldots, d\}$. Note
that if $(f_1, \ldots, f_d) \in \bb_1^d$, then $f_1
\otimes \cdots \otimes f_d \in \dz{C_{\rho}}$ (because
$(f_1 \otimes \cdots \otimes f_d) \circ \rho =
C_{\rho_1} f_1 \otimes \cdots \otimes C_{\rho_d} f_d
\in \bb_d$) and
   \begin{align} \label{tensor}
C_{\rho} (f_1 \otimes \cdots \otimes f_d) = C_{\rho_1}
f_1 \otimes \cdots \otimes C_{\rho_d} f_d,
   \end{align}
where $(f_1 \otimes \cdots \otimes f_d)(\lambda_1,
\ldots, \lambda_d) = f_1(\lambda_1 ) \cdots
f_d(\lambda_d)$ for all $(\lambda_1, \dots, \lambda_d)
\in \cbb^d$. Since\footnote{\;Recall that if $\mm_j$
is a reproducing kernel Hilbert space on $X_j$ with
the reproducing kernel $K_j\colon X_j\times X_j \to
\cbb$ for $j=1, \ldots, n$, then $\mm_1 \otimes \cdots
\otimes \mm_n$ is a reproducing kernel Hilbert space
on $X=X_1 \times \cdots \times X_n$ with the
reproducing kernel $K\colon X\times X \to \cbb$
defined by $K((x_1, \ldots, x_n),(y_1, \ldots,
y_n))=K_1(x_1,y_1) \cdots K_n(x_n,y_n)$ for $(x_1,
\ldots, x_n),(y_1, \ldots, y_n) \in X$ (cf.\
\cite[Theorem I, page 361]{Aronsz}).} $\bb_d = \bb_1
\otimes \cdots \otimes \bb_1$, Proposition
\ref{stale2}(i) and \eqref{tensor} yield $C_{\rho} =
C_{\rho_1} \otimes \cdots \otimes C_{\rho_d}$. This
implies that $C_{\psi} \in \ogr{\bb_d}$ and
   \begin{align*}
\|C_{\psi}\| = \|C_{\rho}\| = \prod_{n=1}^k
\|C_{\rho_n}\| \overset{\eqref{cfnorma}}=
\prod_{n=1}^k
\exp\Big(\frac{|\beta_n|^2}{1-t_n^2}\Big) =
\exp(\is{(I-AA^*)^{-1}b}{b}).
   \end{align*}
Applying Proposition \ref{fipsib} completes the proof.
   \end{proof}
The following proposition sheds more light on the
relationships between \cite[Theorems 1 and 2]{c-m-s03}
and Theorem \ref{s-b-norm}.
   \begin{pro}  \label{naszeq}
If $\hh$ is a finite dimensional complex Hilbert
space, $A\in \ogr{\hh}$ is a contraction and $b\in
\hh$, then the following conditions are equivalent{\em
:}
   \begin{enumerate}
   \item[(i)] $\is{A\xi}{b}=0$ for every $\xi \in \hh$
such that $\|A\xi\|=\|\xi\|$,
   \item[(ii)] $A^*b \in \ob{I-A^*A}$,
   \item[(iii)] $b \in \ob{I-AA^*}$.
   \end{enumerate}
   \end{pro}
   \begin{proof}
Since $\|A\|\Le 1$, we deduce from the Cauchy-Schwarz
inequality that
   \begin{align} \label{csin}
\text{$\forall \xi \in \hh\colon \|A\xi\|=\|\xi\| \iff
\xi \in \jd{I-A^*A}$.}
   \end{align}

(i)$\Rightarrow$(ii) It follows from \eqref{csin} that
$\is{\xi}{A^*b}=0$ for all $\xi \in \jd{I-A^*A}$, or
equivalently that $A^*b \perp \jd{I-A^*A}$. Therefore
$A^*b \in \overline{\ob{I-A^*A}} = \ob{I-A^*A}$.

(ii)$\Rightarrow$(iii) Let $\eta \in \hh$ be such that
$A^*b = (I-A^*A)\eta$. If $\xi\in \jd{I-AA^*}$, then
   \begin{align*}
\is{b}{\xi} & = \is{A^*b}{A^*\xi} =
\is{(I-A^*A)\eta}{A^*\xi}
   \\
& = \is{A \eta}{\xi} - \is{A \eta}{AA^*\xi} = \is{A
\eta}{\xi} - \is{A \eta}{\xi} =0,
   \end{align*}
which means that $b \perp \jd{I-AA^*}$. Hence $b\in
\overline{\ob{I-AA^*}}=\ob{I-AA^*}$.

   (iii)$\Rightarrow$(i) Let $\xi \in \hh$ be such
that $\|A\xi\|=\|\xi\|$. By \eqref{csin}, $\xi \in
\jd{I-A^*A}$. Since $b \in \ob{I-AA^*}$, there exists
$\eta \in \hh$ such that $b=(I-AA^*)\eta$. Thus, we
have
   \begin{align*}
\is{A\xi}{b} = \is{(I-AA^*)A\xi}{\eta} =
\is{A(I-A^*A)\xi}{\eta} = 0.
   \end{align*}
This completes the proof.
   \end{proof}
   Theorem \ref{s-b-norm} has two important
consequences. The first is related to spectral radius
of $C_{\varphi}$ (see Theorem \ref{s-b-prsp}), while
the other to normaloidity of $C_{\varphi}$ (see
Theorem \ref{semi-lid}). Theorems \ref{s-b-prsp} and
\ref{semi-lid}, as well as Corollary \ref{semi-nor},
are no longer true if $\bb_d$ is replaced by
$\exp(\hh)$ with $\dim \hh=\infty$ (cf.\ Example
\ref{optim}). Let us also mention that Theorem
\ref{s-b-prsp} has been recently proved by Trieu Le
using a different method (see \cite[Theorem
1.4.]{LeT}).
   \begin{thm} \label{s-b-prsp}
Let $\varphi\colon \cbb^d \to \cbb^d$ be a holomorphic
mapping $($$d \in \nbb$$)$. Assume $C_{\varphi} \in
\ogr{\bb_d}$. Then $r(C_{\varphi})= 1$.
   \end{thm}
   \begin{proof}
By Theorem \ref{s-b-norm}, there exist $A \in
\ogr{\cbb^d}$ and $b\in \cbb^d$ such that $\varphi = A
+ b$, $\|A\| \Le 1$ and $b \in \ob{I-AA^*}$. Since
$C_{\varphi}^n = C_{\varphi^n}$ and $\varphi^n = A^n +
b_n$ for every $n \in \nbb$, where $b_n = b + \ldots +
A^{n-1}b$ for $n \in \nbb$, we infer from Theorem
\ref{s-b-norm} that $b_n \in \ob{I-A^nA^{*n}}$ for
every $n\in \nbb$. If $\|A\| < 1$, then
   \begin{align*}
\|b_n\| \Le \|b\| \sum_{n=0}^{\infty} \|A\|^n \Le
\frac{\|b\|}{1-\|A\|}, \quad n \in \nbb,
   \end{align*}
and thus (use the C. Neumann's series expansion)
   \begin{align*}
\is{(I-A^nA^{*n})^{-1}b_n}{b_n} & \Le
\|(I-A^nA^{*n})^{-1}\| \|b_n\|^2
   \\
& \Le \frac{\|b_n\|^2}{1-\|A^nA^{*n}\|}
   \\
& \Le \frac{\|b\|^2}{(1-\|A\|)^3}, \quad n \in \nbb.
   \end{align*}
If $\|A\|=1$, then by\footnote{\;Lemma 3.2 in
\cite{D-St} remains valid in the complex case as well
(with the same proof).} \cite[Lemma 3.2(ii)]{D-St} the
sequence
$\{\is{(I-A^nA^{*n})^{-1}b_n}{b_n}\}_{n=1}^{\infty}$
is bounded. Hence, in both cases, $c:= \sup_{n \in
\nbb} \is{(I-A^nA^{*n})^{-1}b_n}{b_n} < \infty$. This
and \eqref{norm} applied to $\varphi^n$ yield
   \begin{align*}
1 \Le \|C_{\varphi}^n\|^{1/n} & = \|C_{\varphi^n}\|^{1/n}
   \\
& =
\exp\left(\frac{\is{(I-A^nA^{*n})^{-1}b_n}{b_n}}{2n}\right)
   \\
& \Le \exp\left(\frac{c}{2n}\right), \quad n \in \nbb.
   \end{align*}
Now, applying Gelfand's formula for spectral radius
completes the proof.
   \end{proof}
   \begin{thm} \label{semi-lid}
Assume $\varphi = A + b$ with $A\in \ogr{\cbb^d}$ and
$b\in \cbb^d$, and $C_{\varphi} \in \ogr{\bb_d}$
$($$d\in \nbb$$)$. Then the following conditions are
equivalent{\em :}
   \begin{enumerate}
   \item[(i)] $C_{\varphi}$ is normaloid,
   \item[(ii)]  $b=0$.
   \end{enumerate}
Moreover, if $C_{\varphi}$ is normaloid, then
$r(C_{\varphi})=\|C_{\varphi}\|=1$.
   \end{thm}
   \begin{proof}
(i)$\Rightarrow$(ii) Suppose $C_{\varphi}$ is normaloid. It
follows from Theorems \ref{s-b-norm} and \ref{s-b-prsp}
that $\is{c}{b}=0$, where $c=(I-AA^*)^{-1}b$. This implies
that $\is{c}{(I-AA^*)c}=0$. Since, by Theorem
\ref{s-b-norm}, $I-AA^* \Ge 0$, we deduce that $b =
(I-AA^*)c=0$.

The implication (ii)$\Rightarrow$(i) as well as the
``moreover'' part follow directly from Theorems
\ref{s-b-norm} and \ref{s-b-prsp} (see also Corollary
\ref{normaloid}(ii)).
   \end{proof}
The following corollary, which is a particular case of
Proposition \ref{cohyp}(iii), can be also deduced from
Corollaries \ref{wn3} and \ref{dim-fin}, and Theorem
\ref{semi-lid} because seminormal operators are
normaloid (cf.\ \cite[Theorem 1 in \S2.6.2]{Fur}).
   \begin{cor} \label{semi-nor}
Assume $\varphi = A + b$ with $A\in \ogr{\cbb^d}$ and
$b\in \cbb^d$, and $C_{\varphi} \in \ogr{\bb_d}$
$($$d\in \nbb$$)$. Then $C_{\varphi}$ is seminormal if
and only if $C_{\varphi}$ is normal. Moreover,
$C_{\varphi}$ is normal if and only if $A$ is normal
and $b=0$.
   \end{cor}
We conclude this section with an example of an
unbounded densely defined composition operator
$C_{\varphi}$ in $\bb_1$ with a holomorphic
symbol $\varphi$ which is not a polynomial.
      \begin{exa} \label{ndd}
Let $\varphi=\exp$. We will show that $C_{\varphi}$
is densely defined as an operator in $\bb_1$. Set
$e_n(\xi) = \xi^n$ for $\xi \in \cbb$ and $n\in
\zbb_+$. Since $\big\{\frac{1}{\sqrt{n!}}
e_n\big\}_{n=1}^{\infty}$ is an orthonormal basis
of $\bb_1$ (cf.\ \cite{Barg}), it suffices to show
that $e_n \in \dz{C_{\varphi}}$ for every $n\in
\zbb_+$. For this, fix $n\in \zbb_+$ and set
$\varOmega_n = \{\xi \in \cbb\colon |\xi| \Ge
4n\}$. Since
      \begin{align*}
2n \mathfrak{Re} \hspace{.1ex}\xi - |\xi|^2 \Le
\frac{1}{2} |\xi|^2 - |\xi|^2 = - \frac{1}{2} |\xi|^2
\quad \xi \in \varOmega_n,
      \end{align*}
we see that
      \begin{align*}
\int_{\varOmega_n} |e_n \circ \varphi|^2 \D \mu_1 & =
\frac{1}{\pi} \int_{\varOmega_n} \exp(2n \mathfrak{Re}
\hspace{.1ex} \xi - |\xi|^2) \D V_1(\xi)
      \\
& \Le \frac{1}{\pi}\int_{\cbb} \exp(-
\frac{1}{2}|\xi|^2) \D V_1(\xi) < \infty,
      \end{align*}
which completes the proof.
      \end{exa}
Arguing as in Example \ref{ndd}, we see that a
composition operator $C_{\varphi}$ in $\bb_1$ with a
polynomial symbol $\varphi$ of an arbitrary degree is
always densely defined (because $\limsup_{|z| \to
\infty} \frac{|\varphi(z)|}{|z|^{\deg \varphi}} <
\infty$ whenever $\varphi \neq 0$); moreover, if $\deg
\varphi \Ge 2$, then $C_{\varphi}$ is unbounded (see
Theorem \ref{s-b-norm}).
   \section{\label{Sek13}Composition operators in $L^2(\mu_d)$}
   Below we discuss the relationship between
composition operators acting in function spaces
$\bb_d$ and $L^2(\mu_d)$ respectively, whose symbols
are holomorphic. If $\varphi\colon \cbb^d \to \cbb^d$
is a Borel function, then $\tilde C_{\varphi}$ stands
for the operator in $L^2(\mu_d)$ defined by
   \begin{align*}
\dz{\tilde C_{\varphi}} & = \{f \in
L^2(\mu_d)\colon f \circ \varphi \in
L^2(\mu_d)\},
   \\
\tilde C_{\varphi} f &= f \circ \varphi, \quad f \in
\dz{\tilde C_{\varphi}}.
   \end{align*}
It is well-known that
   \begin{align} \label{w-d}
   \begin{minipage}{63ex}
{\em the operator $\tilde C_{\varphi}$ is well-defined
$($no matter what the domain of $\tilde C_{\varphi}$
is\/$)$ if and only if $\mu_d \circ \varphi^{-1} \ll \mu_d$
$($absolute continuity$)$, where $\mu_d \circ
\varphi^{-1}(\varDelta) = \mu_d (\varphi^{-1}(\varDelta))$
for every Borel subsubset $\varDelta$ of $\cbb^d$.}
   \end{minipage}
   \end{align}

We begin by extending \cite[Corollary 2.5]{D-St}.
   \begin{thm} \label{bl2G}
Suppose $\varphi\colon \cbb^d \to \cbb^d$ is a
holomorphic mapping $($$d\in \nbb$$)$. Then the
following two conditions are equivalent{\em :}
   \begin{enumerate}
   \item[(i)] $\tilde
C_{\varphi}$ is well-defined and $\dz{\tilde
C_{\varphi}}=L^2(\mu_d)$,
   \item[(ii)] there exist $A\in \ogr{\cbb^d}$
and $b\in \cbb^d$ such that $\varphi = A + b$, $A$ is
nonsingular, $\|A\| \Le 1$ and $b \in \ob{I-AA^*}$.
   \end{enumerate}
Moreover, if {\em (ii)} holds, then $\tilde C_{\varphi} \in
\ogr{L^2(\mu_d)}$ and
   \begin{align*}
\|\tilde C_{\varphi}\|^2 & = \frac{1}{|\det
A|^2}\exp(\is{(I-AA^*)^{-1}b}{b}),
   \\
r(\tilde C_{\varphi}) & = \frac{1}{|\det A|},
   \end{align*}
where $\det A$ stands for the determinant of a complex
matrix associated with $A$.
   \end{thm}
   \begin{proof}
It follows from \eqref{w-d} that $\tilde C_{\varphi}$
is well-defined if and only if $V_d \circ \varphi^{-1}
\ll V_d$. Using the change-of-variables formula and
the fact that $V_d(M)=0$ for every proper linear
subspace $M$ of $\cbb^d$, we deduce that if $\varphi =
A + b$ with $A\in \ogr{\cbb^d}$ and $b\in \cbb^d$,
then $\tilde C_{\varphi}$ is well-defined if and only
if $A$ is nonsingular.

(i)$\Rightarrow$(ii) By assumption, if $f \in \bb_d$,
then $f\circ \varphi \in L^2(\mu_d)$ and thus, because
$f\circ \varphi$ is holomorphic, $f \in
\dz{C_{\varphi}}$. This means that $\dz{C_{\varphi}} =
\bb_d$. By Proposition \ref{stale2} and the closed
graph theorem, $C_{\varphi} \in \ogr{\bb_d}$. It
follows from Theorem \ref{s-b-norm} that there exist
$A\in \ogr{\cbb^d}$ and $b\in \cbb^d$ such that
$\varphi = A + b$, $\|A\| \Le 1$ and $b \in
\ob{I-AA^*}$. Since $\tilde C_{\varphi}$ is
well-defined, $A$ is nonsingular.

The implication (ii)$\Rightarrow$(i) and the
``moreover'' part can be deduced from \cite[Corollary
2.5]{D-St}, \cite[Theorem 3.4(i)]{D-St} and
\cite[Remark 3$^{\circ}$ in Section 6]{D-St}.
   \end{proof}
   \begin{rem}  \label{mamde}
It follows from Theorems \ref{s-b-norm} and \ref{bl2G}
that if $\varphi\colon \cbb^d \to \cbb^d$ is a
holomorphic mapping such that the operator $\tilde
C_{\varphi}$ is well-defined, then $\tilde C_{\varphi}
\in \ogr{L^2(\mu_d)}$ if and only if $C_{\varphi} \in
\ogr{\bb_d}$. Moreover, by Theorem \ref{s-b-prsp}, if
$\tilde C_{\varphi} \in \ogr{L^2(\mu_d)}$, then
   \begin{align*}
\frac{\|\tilde C_{\varphi}\|}{\|C_{\varphi}\|} =
\frac{r(\tilde C_{\varphi})}{r(C_{\varphi})} =
\frac{1}{|\det A|},
   \end{align*}
where $\varphi = A + b$ with $A\in \ogr{\cbb^d}$ and
$b\in \cbb^d$. In particular, if $d=1$ and $b \in
\cbb$, then $\lim_{A \to 0}\|C_{A+b}\|^2 =
\exp(|b|^2)$ and $r(C_{A+b}) = 1$ for all $A \in \cbb$
with $|A| < 1$, while $\lim_{A \to 0}\|\tilde
C_{A+b}\| = \lim_{A \to 0} r(\tilde C_{A+b}) =
\infty$.

By Theorem \ref{s-b-norm} and Corollary
\ref{semi-nor}, bounded seminormal composition
operators $C_{\varphi}$ in $\bb_d$ with holomorphic
symbols $\varphi$ are always normal. The situation is
quite different for composition operators $\tilde
C_{\varphi}$ in $L^2(\mu_d)$. Since the measure
$\mu_d$ is finite, bounded hyponormal composition
operators $\tilde C_{\varphi}$ in $L^2(\mu_d)$ with
holomorphic symbols $\varphi$ are always unitary (see
Theorem \ref{bl2G} and \cite[Lemma 7 and Theorem
0]{ha-wh}). However, if $d > 1$, then there exists a
bounded composition operator $\tilde C_{A}$ in
$L^2(\mu_d)$ with a nonsingular $A \in \ogr{\cbb^d}$
such that $(\tilde C_{A})^*$ is hyponormal but not
subnormal (see \cite[Example 2.6 and (UE)]{js1}).
Moreover, if $d\in \nbb$, then there exists a bounded
composition operator $\tilde C_{A}$ in $L^2(\mu_d)$
with a nonsingular $A \in \ogr{\cbb^d}$ such that
$(\tilde C_{A})^*$ is subnormal but not normal.
Indeed, take a nonsingular nonunitary normal operator
$A\in \ogr{\cbb^d}$ such that $\|A\| \Le 1$. Then, by
\cite[Proposition 2.2 and Theorem 2.5]{js1} applied to
$\varphi=\exp$, $\tilde C_{A} \in \ogr{\bb_d}$ and
$(\tilde C_{A})^*$ is subnormal, while, by
\cite[Proposition 2.3]{js1}, $\tilde C_{A}$ is not
normal. Note that if $d=1$, then the adjoint of any
bounded composition operator $\tilde C_{A}$ in
$L^2(\mu_1)$ with $A \in \cbb \setminus \{0\}$ is
subnormal (cf.\ \cite[Theorem 2.5]{js1}).
   \end{rem}
   \section{\label{Sek14}Boundedness of $C_{A+b}$ in $\exp(\hh)$}
   In this section we will give necessary and
sufficient conditions for $C_{A+b}$ to be a bounded
operator on $\exp(\hh)$. We begin by proving two
lemmata that will be used in the proof of Theorem
\ref{glowne2} which is the main result of this
section.
   \begin{lem} \label{bl1}
Suppose $A\in \ogrp{\hh}$, $b \in \hh$ and $\dim
\ob{A} < \infty$. Then $C_{A+b} \in \ogr{\exp(\hh)}$
if and only if $\|A\| \Le 1$ and $b\in \ob{I-A^2}$.
Moreover, if $C_{A+b} \in \ogr{\exp(\hh)}$, then
   \begin{align*}
\|C_{A+b}\|^2 = \exp(\is{(I - A^2)^{-1}b}{b}).
   \end{align*}
   \end{lem}
   \begin{proof}
In view of Lemma \ref{n&sc4b}(i), there is no loss of
generality in assuming that $\|A\|\Le 1$. Note that $A
= 0_{\hh_0} \oplus A_1$, where $0_{\hh_0}$ is the zero
operator on $\hh_0=\jd{A}$ and $A_1$ is an injective
positive operator on $\hh_1=\ob{A}$. This implies that
$(I-A^s)^t = I_{\hh_0} \oplus (I_{\hh_1} - A_1^s)^t$
for all $s,t \in (0,\infty)$. Since $\dim \hh_1 <
\infty$, we get
   \begin{align} \label{dobr}
\ob{(I-A^s)^t} & = \hh_0 \oplus \ob{I_{\hh_1}-A_1^s} =
\ob{I-A^s}, \quad s,t \in (0,\infty).
   \end{align}

Suppose $C_{A+b} \in \ogr{\exp(\hh)}$. By Lemma
\ref{n&sc4b}(i) and \eqref{dobr}, $b \in \ob{I-A^2} =
\overline{\ob{I-A^2}}$ and $b=b_0 \oplus b_1$ for some
$b_0 \in \hh_0$ and $b_1 \in \ob{I_{\hh_1}-A_1^2}$. By
Theorem \ref{s-b-norm}, $C_{A_1+b_1} \in
\ogr{\exp(\hh_1)}$ and $\|C_{A_1+b_1}\|^2 =
\exp(\is{(I_{\hh_1}-A_1^2)^{-1}b_1}{b_1})$. In turn,
by Proposition \ref{nonzero}(iv), $C_{b_0} \in
\ogr{\exp(\hh_0)}$ and $\|C_{b_0}\|^2 =
\exp(\|b_0\|^2)$. Note that there exists a unitary
operator $U\colon \exp(\hh) \to \exp(\hh_0) \otimes
\exp(\hh_1)$ such that
   \begin{align*}
U K^{\exp,\hh}_{\xi} = K^{\exp,\hh_0}_{P_0\xi} \otimes
K^{\exp,\hh_1}_{P_1\xi}, \quad \xi \in \hh,
   \end{align*}
where $P_0$ and $P_1$ are orthogonal projections of
$\hh$ onto $\hh_0$ and $\hh_1$, respectively (cf.\
\cite[Proposition 1.31]{A-R-K-L-S}). Using Lemma
\ref{cfkz}(ii), we verify that
   \begin{align} \label{ucak}
U C_{A+b} K^{\exp,\hh}_{\xi} = (C_{b_0} \otimes
C_{A_1+b_1}) U K^{\exp,\hh}_{\xi}, \quad \xi \in \hh.
   \end{align}
This implies that $C_{A+b}$ is unitarily equivalent to
$C_{b_0} \otimes C_{A_1+b_1}$ and
   \begin{align*}
\|C_{A+b}\|^2 &= \|C_{b_0}\|^2 \|C_{A_1+b_1}\|^2
   \\
& = \exp(\|b_0\|^2)
\exp(\is{(I_{\hh_1}-A_1^2)^{-1}b_1}{b_1})
   \\
& = \exp(\is{(I-A^2)^{-1}b}{b}).
   \end{align*}

The ''if'' part can be derived from \eqref{dobr},
\eqref{ucak} and Theorem \ref{s-b-norm}.
   \end{proof}
   \begin{lem} \label{bl2}
Suppose $A\in \ogrp{\hh}$, $b \in \hh$ and $\pcal
\subseteq \ogr{\hh}$ is an upward-directed partially
ordered set of finite rank orthogonal projections such
that $\bigvee_{P\in \pcal} \ob{P} = \hh$. Then the
following conditions are equivalent{\em :}
   \begin{enumerate}
   \item[(i)] $C_{A+b}
\in \ogr{\exp(\hh)}$,
   \item[(ii)] $\|A\| \Le 1$,  $b\in \ob{I-APA}$ for every
$P\in \pcal$ and
   \begin{align*}
S(A,b):=\sup\{\is{(I-APA)^{-1}b}{b}\colon P \in
\pcal\} < \infty.
   \end{align*}
   \end{enumerate}
Moreover, if {\em (ii)} holds, then
   \begin{align*}
\|C_{A+b}\|^2 = \exp(S(A,b)).
   \end{align*}
   \end{lem}
   \begin{proof}
(i)$\Rightarrow$(ii) By Lemma \ref{n&sc4b}(i), $\|A\|
\Le 1$ and $b \in \ob{(I-A^2)^{1/2}}$. Take $P \in
\pcal$. Since $APA \Le A^2$, we see that $I-APA \Ge I
- A^2 \Ge 0$. By the Douglas theorem (cf.\
\cite[Theorem 1]{Doug}) and \eqref{dobr}, we have
   \begin{align*}
b \in \ob{(I - A^2)^{1/2}} \subseteq \ob{(I -
APA)^{1/2}} = \ob{I - APA}.
   \end{align*}
This, $\dim \ob{(APA)^{1/2}} < \infty$ and Lemma
\ref{bl1} yield $C_{(APA)^{1/2}+b} \in
\ogr{\exp(\hh)}$. Since $C_P$ is an orthogonal
projection (cf.\ Corollary \ref{opr}) and $C_{AP +
b}=C_PC_{A+b}\in \ogr{\exp(\hh)}$, we infer from Lemma
\ref{bl1} and Proposition \ref{fipsib} that
   \begin{align*}
\exp(\is{(I - APA)^{-1}b}{b}) &= \|C_{(APA)^{1/2} +
b}\|^2
   \\
& = \|C_{AP + b}\|^2 = \|C_PC_{A+b}\|^2 \Le
\|C_{A+b}\|^2.
   \end{align*}
This implies that $\exp(S(A,b)) \Le \|C_{A+b}\|^2$.

(ii)$\Rightarrow$(i) Take $P\in \pcal$. By Lemma
\ref{bl1} and Proposition \ref{fipsib}, $C_{AP+b} \in
\ogr{\exp(\hh)}$, $C_{(APA)^{1/2}+b} \in
\ogr{\exp(\hh)}$, $\|C_{AP+b}\|=\|C_{(APA)^{1/2}+b}\|$
and
   \begin{align*}
\|C_P C_{A+b} f\|^2 & = \|C_{AP+b} f\|^2
   \\
& \Le \|C_{(APA)^{1/2}+b}\|^2 \|f\|^2
   \\
&= \exp(\is{(I - APA)^{-1}b}{b}) \|f\|^2
   \\
& \Le \exp(S(A,b)) \|f\|^2, \quad f\in \dz{C_{A+b}}.
   \end{align*}
Applying Proposition \ref{suport}, we deduce that
$\|C_{A+b} f\|^2 \Le \exp(S(A,b)) \|f\|^2$ for all
$f\in \dz{C_{A+b}}$. By Proposition \ref{stale2}(i)
and Lemma \ref{cfkz}(i), this implies that $C_{A+b}
\in \ogr{\exp(\hh)}$ and $\|C_{A+b}\|^2 \Le
\exp(S(A,b))$.
   \end{proof}
Now we are in a position to characterize the
boundedness of $C_{A+b}$. The equivalence
(i)$\Leftrightarrow$(iii) and the equality
$\|C_{\varphi}\|^2 = \exp(\|(I-AA^*)^{-1/2}b\|^2)$
have been proved independently by Trieu Le using a
different approach (cf.\ \cite{LeT}).
   \begin{thm} \label{glowne2}
Let $\varphi\colon \hh \to \hh$ be a holomorphic
mapping and $\pcal \subseteq \ogr{\hh}$ be an
upward-directed partially ordered set of finite rank
orthogonal projections such that\/\footnote{\;Such a
$\pcal$ always exists, e.g., the set of all finite
rank orthogonal projections in $\hh$ has the required
properties (see also Example \ref{uww}).}
$\bigvee_{P\in \pcal} \ob{P} = \hh$. Then the
following conditions are equivalent{\em :}
   \begin{enumerate}
   \item[(i)] $C_{\varphi}
\in \ogr{\exp(\hh)}$,
   \item[(ii)] $\varphi=A+b$, where $A\in \ogr{\hh}$,
$\|A\| \Le 1$, $b \in \ob{I-|A^*|P|A^*|}$ for every
$P\in \pcal$ and $S(A,b):=
\sup\{\is{(I-|A^*|P|A^*|)^{-1}b}{b} \colon P \in
\pcal\} < \infty$,
   \item[(ii$^{\prime}$)] $\varphi=A+b$, where $A\in
\ogr{\hh}$, $\||A^*|P|A^*|\| \Le 1$ for every $P\in
\pcal$, $b \in \ob{I-|A^*|P|A^*|}$ for every $P\in
\pcal$ and $S(A,b) < \infty$,
   \item[(iii)] $\varphi=A+b$, where $A\in \ogr{\hh}$,
$\|A\| \Le 1$ and $b \in \ob{(I-AA^*)^{1/2}}$.
   \end{enumerate}
Moreover, if {\em (iii)} holds, then
   \begin{align} \label{Jur3}
\|C_{\varphi}\|^2 = \exp(\|(I-AA^*)^{-1/2}b\|^2)
=\exp(S(A,b)).
   \end{align}
   \end{thm}
   \begin{proof}
In view of Proposition \ref{pro1}, there is no loss of
generality in assuming that $\varphi = A + b$, where
$A\in \ogr{\hh}$ and $b\in \hh$. That the conditions
(i) and (ii) are equivalent and $\|C_{\varphi}\|^2
=\exp(S(A,b))$ follows from Proposition \ref{fipsib}
and Lemma \ref{bl2}.

(ii)$\Leftrightarrow$(iii) Assume $A\in \ogr{\hh}$ is
a contraction. Set $A_P=I - |A^*|P|A^*|$ for $P\in
\pcal$. Then $A_P\in\ogrp{\hh}$ for all $P\in \pcal$.
Since $\bigvee_{P\in \pcal} \ob{P} = \hh$, we see that
$\{P\}_{P\in \pcal}$ is a monotonically increasing net
which converges in the strong operator topology to
$I$. This implies that $\{A_{P}\}_{P \in \pcal}
\subseteq \ogrp{\hh}$ is a monotonically decreasing
net which converges in the strong operator topology to
$I - |A^*|^2$. Since $\dim \ob{|A^*|P|A^*|} < \infty$
for all $P\in \pcal$, we infer from \eqref{dobr} that
$\ob{A_P}$ is closed and $\ob{A_P}=\ob{A_P^{1/2}}$ for
all $P\in \pcal$. By Lemma \ref{a-12},
$\is{A_P^{-1}\xi}{\xi} = \|A_P^{-1/2}\xi\|^2$ for all
$\xi \in \ob{A_P}$ and $P\in \pcal$. Now applying
Lemma \ref{cfkz-2}, we deduce that the conditions (ii)
and (iii) are equivalent and
$\exp(\|(I-AA^*)^{-1/2}b\|^2) =\exp(S(A,b))$.

The conditions (ii) and (ii$^{\prime}$) are easily
seen to be equivalent.
   \end{proof}
Noting that $\jd{A^*} \subseteq \ob{I-AA^*}$, we get
the following corollary.
   \begin{cor}
If $A\in \ogr{\hh}$ and $b\in \jd{A^*}$, then $C_{A+b}
\in \ogr{\exp(\hh)}$ if and only if $\|A\| \Le 1$.
   \end{cor}
   The next result provides an estimate for the rate
of growth of the sequence
$\{\|(I-A^nA^{*n})^{-1/2}b_n\|\}_{n=1}^{\infty}$ whose
$n$-th term appears in the formula for the norm of the
$n$-th power of $C_{A+b}$. As shown in Example
\ref{optim} below, this estimate is optimal.
   \begin{pro}\label{jakwn}
Suppose $C_{\varphi} \in \ogr{\exp(\hh)}$, where
$\varphi = A + b$ with $A \in \ogr{\hh}$ and $b \in
\hh$. Set $b_n=(I + \ldots + A^{n-1})b$ for
$n\in\nbb$. Then the following holds{\em :}
   \begin{enumerate}
   \item[(i)] $\varphi^n = A^n+b_n$ and $b_n \in
\ob{(I-A^nA^{*n})^{1/2}}$ for all $n\in \nbb$,
   \item[(ii)] there exists a constant $M \in (0,\infty)$ such
that
   \begin{align} \label{optosz}
\|(I-A^nA^{*n})^{-1/2}b_n\| \Le M \sqrt{n}, \quad n
\in \nbb.
   \end{align}
   \end{enumerate}
   \end{pro}
   \begin{proof}
Since $C_{\varphi}^n=C_{\varphi^n}$ for all $n\in
\nbb$, the assertion (i) follows from Theorem
\ref{glowne2}. This theorem, combined with Gelfand's
formula for the spectral radius, yields
   \begin{align} \label{Jur2}
r(C_{\varphi}) = \lim_{n\to\infty} \|C_{\varphi}^n
\|^{1/n} = \lim_{n\to\infty}
\exp\Big(\frac{1}{2n}\|(I-A^nA^{*n})^{-1/2}b_n\|^2\Big).
   \end{align}
Hence, there exists $R\in (0,\infty)$ such that
$\|(I-A^nA^{*n})^{-1/2} b_n\| \Le R \sqrt{n}$ for $n$
large enough. This implies (ii).
   \end{proof}
As shown in Example \ref{optim} below, Theorem
\ref{s-b-prsp} is no longer true if $\bb_d$ is
replaced by $\exp(\hh)$, where $\hh$ is an infinite
dimensional Hilbert space. However, under some
circumstances, the conclusion of this theorem is still
valid (cf.\ \cite[Proposition 3.9]{LeT}).
   \begin{pro} \label{spr1}
Suppose $\varphi = A + b$, where $A \in \ogr{\hh}$, $b
\in \hh$ and $\|A\|<1$. Then $C_{\varphi} \in
\ogr{\exp(\hh)}$ and $r(C_{\varphi})=1$. Moreover, if
$b\neq 0$, then $C_{\varphi}$ is not normaloid.
   \end{pro}
   \begin{proof}
It follows from C. Neumann's theorem and Theorem
\ref{glowne2} that $C_{\varphi} \in \ogr{\exp(\hh)}$
and \eqref{Jur2} holds, where $\{b_n\}_{n=1}^{\infty}$
is as in Proposition \ref{jakwn}. Since $\|A\|<1$, we
deduce from C. Neumann's theorem that $(I-A)^{-1} \in
\ogr{\hh}$ and
   \begin{align} \label{Jur1}
b_n = (I-A^n)(I-A)^{-1}b, \quad n\in \nbb.
   \end{align}
Applying C. Neumann's theorem again, we see that
$(I-A^nA^{*n})^{-1} \in \ogr{\hh}$ for all $n\in \nbb$
and
   \allowdisplaybreaks
   \begin{align*}
\|(I-A^nA^{*n})^{-1/2}b_n\|^2 & =
\is{(I-A^nA^{*n})^{-1}b_n}{b_n}
   \\
&\hspace{-1.5ex} \overset{\eqref{Jur1}}\Le
\frac{\|(I-A^n)(I-A)^{-1}b\|^2}{1-\|A\|^{2n}}
   \\
& \Le \frac{4\|b\|^2}{(1-\|A\|^{2n})(1-\|A\|)^2},
\quad n\in \nbb.
   \end{align*}
This, together with \eqref{Jur2}, gives
$r(C_{\varphi})=1$. Finally, if $b\neq 0$, then by
\eqref{Jur3}, $\|C_{\varphi}\| > 1$. This completes
the proof.
   \end{proof}
   \begin{exa} \label{optim}
Let $\hh$ be an infinite dimensional Hilbert space,
$V\in \ogr{\hh}$ be an isometry and $b\in \hh$. Set
$\varphi=V+b$. By Theorem \ref{glowne2},
$C_{\varphi}\in \ogr{\exp(\hh)}$ if and only if $b\in
\jd{V^*}$. Suppose that $V$ is not unitary, i.e.,
$\jd{V^*} \neq \{0\}$. Take $b\in \jd{V^*}\setminus
\{0\}$. Then $\{V^n b\}_{n=0}^{\infty}$ is an
orthogonal sequence,
$\ob{(I-V^nV^{*n})^{1/2}}=\jd{V^{*n}}$ for every $n\in
\nbb$ and
   \begin{align*}
\|(I-V^nV^{*n})^{-1/2}b_n\|^2 = \|b_n\|^2 = \|b +
\ldots + V^{n-1} b\|^2 = \|b\|^2 n, \quad n\in \nbb,
   \end{align*}
which means that the inequality in \eqref{optosz}
becomes an equality with $M=\|b\|$. Now we show that
$\E^{-\|b\|^2/2} C_{\varphi}$ is a coisometry. Indeed,
by Theorem \ref{sprz}(i) and Lemma \ref{cfkz}(ii), we
have
   \begin{align*}
C_{\varphi} C_{\varphi}^* K_{\xi} = C_{\varphi}
K_{V\xi + b} = \E^{\is{b}{V\xi+b}} K_{V^{*}(V\xi + b)}
= \E^{\|b\|^2} K_{\xi}, \quad \xi \in \hh.
   \end{align*}
Since $\{K_{\xi}\}_{\xi \in \hh}$ is total in
$\exp(\hh)$, we deduce that $C_{\varphi} C_{\varphi}^*
= \E^{\|b\|^2} I$, which implies that $\E^{-\|b\|^2/2}
C_{\varphi}$ is a coisometry. In particular,
$C_{\varphi}$ is cohyponormal. Hence, $C_{\varphi}$ is
normaloid and consequently, by Theorem \ref{glowne2},
we have (cf.\ \cite[Proposition 3.8]{LeT})
   \begin{align} \label{nonc}
r(C_{\varphi}) = \|C_{\varphi}\| = \E^{\|b\|^2/2}.
   \end{align}
It follows from Corollary \ref{hypo} that
$C_{\varphi}$ is not normal. In view of \eqref{nonc},
the spectral radius of $C_{\varphi}$ can take any
value in the interval $(1,\infty)$ if $b$ ranges over
the set $\jd{V^*}\setminus \{0\}$. Summarizing, we see
that Theorems \ref{s-b-prsp} and \ref{semi-lid}, and
Corollary \ref{semi-nor} are no longer true if $\bb_d$
is replaced by $\exp(\hh)$ with $\dim \hh=\infty$.
   \end{exa}
Finally, we illustrate Theorem \ref{glowne2} and
Proposition \ref{jakwn} in the context of diagonal
operators.
   \begin{exa}\label{uww}
Let $\hh$ be an infinite dimensional separable complex
Hilbert space, $\{e_n\}_{n=0}^{\infty}$ be its
orthonormal basis and $A\in \ogr{\hh}$ be a positive
contractive diagonal operator (subordinated to
$\{e_n\}_{n=0}^{\infty}$) with diagonal
$\{\alpha_n\}_{n=0}^{\infty}$, i.e., $A e_n = \alpha_n
e_n$ and $\alpha_n \in [0,1]$ for every $n\in \zbb_+$.
Since for $t \in (0,\infty)$, $(I-A^2)^{t}$ is a
diagonal operator with diagonal
$\{(1-\alpha_n^2)^{t}\}_{n=0}^{\infty}$, we deduce
that for every $t \in (0,\infty)$,
   \allowdisplaybreaks
   \begin{gather}  \notag
\ob{(I-A^2)^{t}} = \bigg\{b \in \hh \colon
\sum_{n\colon \alpha_n <
1}\frac{|\is{b}{e_n}|^2}{(1-\alpha_n^2)^{2t}} < \infty
\text{ and } \is{b}{e_k} =0 \text{ if } \alpha_k=1
\bigg\},
   \\   \label{xyx}
\|(I-A^2)^{-t}b\|^2 = \sum_{n\colon \alpha_n <
1}\frac{|\is{b}{e_n}|^2}{(1-\alpha_n^2)^{2t}}, \quad b
\in \ob{(I-A^2)^{t}}.
   \end{gather}
Using the first of the above equalities (see also
\eqref{aa-1} and \eqref{aa-2}), one can find a
sequence $\{\alpha_n\}_{n=0}^{\infty} \subseteq [0,1]$
such that
   \begin{align*}
\dz{(I-A^2)^{-1}} = \ob{I-A^2} \varsubsetneq
\ob{(I-A^2)^{1/2}} = \dz{(I-A^2)^{-1/2}},
   \end{align*}
which emphasizes the difference between Theorems
\ref{s-b-norm} and \ref{glowne2}. Note also that the
set $\pcal := \{P_n\colon n \in \zbb_+\}$, where $P_n$
is the orthogonal projection of $\hh$ onto the linear
span $\hh_n$ of $\{e_k\}_{k=0}^n$, satisfies the
assumptions of Theorem \ref{glowne2}. Since for all
$n\in \zbb_+$, $\hh_n$ reduces $A$ and $I-AP_nA =
(I_{\hh_n}-(A|_{\hh_n})^2) \oplus I_{\hh\ominus
\hh_n}$, we see that $I-AP_nA$ is a diagonal operator
in $\hh$ with diagonal $(1-\alpha_0^2, \ldots,
1-\alpha_n^2, 1,1, \ldots)$.

Suppose now that $b\in \ob{(I-A^2)^{1/2}}$ (as above
$\{\alpha_n\}_{n=0}^{\infty}\subseteq [0,1]$). By
Theorem \ref{glowne2}, $C_{A+b} \in \ogr{\exp(\hh)}$.
Applying \eqref{xyx} to powers of $C_{A+b}$ and using
Lebesgue's monotone convergence theorem, we obtain
(with $b_k=(I + \ldots + A^{k-1})b$)
   \allowdisplaybreaks
   \begin{align*}
\|(I-A^{2k})^{-1/2}b_k\|^2 & = \sum_{n\colon \alpha_n
< 1}\frac{|\is{b_k}{e_n}|^2}{1-\alpha_n^{2k}}
   \\  \notag
& = \sum_{n\colon \alpha_n < 1}\frac{|\is{b}{e_n}|^2(1
+ \alpha_n + \ldots +
\alpha_n^{k-1})^2}{1-\alpha_n^{2k}}
   \\
& = \sum_{n\colon \alpha_n < 1}
\frac{|\is{b}{e_n}|^2(1 -
\alpha_n^{k})^2}{(1-\alpha_n)^2(1-\alpha_n^{2k})}
   \\
& = \sum_{n\colon \alpha_n < 1} \frac{|\is{b}{e_n}|^2
(1 - \alpha_n^{k})}{(1-\alpha_n)^2(1+\alpha_n^{k})} \;
\underset{(k \to \infty)}{\nearrow} \; \sum_{n\colon
\alpha_n < 1} \frac{|\is{b}{e_n}|^2}{(1-\alpha_n)^2}.
   \end{align*}
Set $\varPsi_b= \sum_{n\colon \alpha_n < 1}
\frac{|\is{b}{e_n}|^2}{(1-\alpha_n)^2}$. We show that
both cases $\varPsi_b < \infty$ and $\varPsi_b =
\infty$ can occur (still under the assumption that
$b\in \ob{(I-A^2)^{1/2}}$). Indeed, take $x \in
(1,\infty)$ and $y \in (0,\infty)$. Set $\alpha_n = 1
- \frac{1}{(n+1)^y}$ for all $n\in \zbb_+$. Then
$\{\alpha_n\}_{n=0}^{\infty} \subseteq [0,1)$ and
there exists $b\in \hh$ such that
$|\is{b}{e_n}|^2=\frac{1}{(n+1)^x}$ for all $n\in
\zbb_+$. It follows from \eqref{xyx} that $b\in
\ob{(I-A^2)^{1/2}}$ if and only if $x-y > 1$. In turn,
$\varPsi_b = \infty$ if and only if $x-2y \Le 1$.
Since the sets
   \begin{align*}
\{(x,y)\in (1,\infty) \times (0,\infty)\colon x-y
> 1, x-2y \Le 1\},
   \\
\{(x,y)\in (1,\infty) \times (0,\infty)\colon x-y
> 1, x-2y >  1\},
   \end{align*}
are nonempty, we are done.
   \end{exa}
   \appendix
\numberwithin{equation}{section}
\numberwithin{thm}{section}
   \section{\label{apap}Conjugations}
   Let $\hh$ be a complex Hilbert space. An
anti-linear map $Q\colon \hh \to \hh$ such that
$Q(Q\xi) = \xi$ and $\is{Q\xi}{Q\eta}=\is{\eta}{\xi}$
for all $\xi,\eta \in \hh$ is called a {\em
conjugation} on $\hh$. Note that such a map always
exists. Indeed, if $\{e_\omega\}_{\omega\in
\varOmega}$ is any orthonormal basis of $\hh$, then
there exists a unique conjugation $Q$ on $\hh$ such
that $Q e_\omega=e_\omega$ for all $\omega \in
\varOmega$. In fact, any conjugation is of this form
(cf.\ \cite[Lemma 1]{Gar-Put}). Given a conjugation
$Q$ on $\hh$, we denote by $\varXi_Q$ the selfmap of
$\ogr{\hh}$ defined by $\varXi_Q(A) = QAQ$ for $A \in
\ogr{\hh}$. The fixed points of $\varXi_Q$ are called
{\em $Q$-real} operators. As shown below, each
selfadjoint operator is $Q$-real with respect to some
conjugation $Q$.
   \begin{pro}
If $A\in \ogr{\hh}$ is selfadjoint, then there exists
a conjugation $Q$ on $\hh$ such that $A$ is $Q$-real.
   \end{pro}
   \begin{proof}
By the Kuratowski-Zorn lemma, there exists a family
$\{e_\omega\}_{\omega\in \varOmega}$ of vectors in $\hh$
such that $\hh = \bigoplus_{\omega \in \varOmega}
\hh_\omega(A)$, where $\hh_\omega(A) = \bigvee_{n\in
\zbb_+} A^n e_\omega$. Note that each $\hh_\omega(A)$
reduces $A$. It is now easily seen that for every $\omega
\in \varOmega$, there exists a (unique) conjugation
$Q_\omega$ on $\hh_\omega(A)$ such that $Q_\omega (A^n
e_\omega) = A^n e_\omega$ for all $n\in \zbb_+$. This
implies that $Q_\omega A_\omega = A_\omega Q_\omega$ for
all $\omega \in \varOmega$ with $A_\omega =
A|_{\hh_\omega(A)}$. Set $Q(\bigoplus_{\omega \in
\varOmega} h_\omega) = \bigoplus_{\omega \in \varOmega}
Q_\omega (h_\omega)$ for $\bigoplus_{\omega \in \varOmega}
h_\omega \in \bigoplus_{\omega \in \varOmega}
\hh_\omega(A)$. Then $Q$ is a conjugation on $\hh$ such
that $QA=AQ$, i.e., $\varXi_Q(A)=A$.
   \end{proof}
Now we show that any two conjugations on $\hh$ are
unitarily equivalent.
   \begin{pro}
Suppose $Q_1$ and $Q_2$ are conjugations on $\hh$.
Then there exists a unitary operator $U\in \ogr{\hh}$
such that $U^{-1}Q_2U = Q_1$. Moreover, we have
   \begin{align} \label{q1q2}
\varXi_{Q_1}(U^{-1} A U) = U^{-1} \varXi_{Q_2}(A) U,
\quad A \in \ogr{\hh}.
   \end{align}
   \end{pro}
   \begin{proof}
By \cite[Lemma 1]{Gar-Put}, there are orthonormal
bases $\{e^1_{\omega}\}_{\omega \in \varOmega}$ and
$\{e^2_{\omega}\}_{\omega \in \varOmega}$ of $\hh$
such that $Q_k e^k_{\omega}=e^k_{\omega}$ for all
$\omega \in \varOmega$ and $k=1,2$. Let $U\in
\ogr{\hh}$ be a unitary operator such that
$Ue^1_{\omega}=e^2_{\omega}$ for all $\omega \in
\varOmega$. Then $U^{-1}Q_2U e^1_{\omega}= Q_1
e^1_{\omega}$ for every $\omega \in \varOmega$ and
thus $U^{-1}Q_2U = Q_1$. As a consequence, the
equality \eqref{q1q2} holds.
   \end{proof}
Below, in Proposition \ref{abcon}, we collect some
basic properties of the selfmap $\varXi_Q$, all of
which are easy to prove. In particular, the selfmap
$\varXi_Q$ can be thought of as an abstract
conjugation on $\ogr{\hh}$.
   \begin{pro} \label{abcon}
If $Q$ is a conjugation on $\hh$, then
   \begin{enumerate}
   \item[(i)] $\varXi_Q$ is an antilinear bijection,
   \item[(ii)] $\|\varXi_Q(A)\|=\|A\|$ for all $A \in
\ogr{\hh}$,
   \item[(iii)] $\varXi_Q(\varXi_Q(A))=A$ for all $A \in
\ogr{\hh}$,
   \item[(iv)]
$\varXi_Q(AB)=\varXi_Q(A)\varXi_Q(B)$ for all $A,B \in
\ogr{\hh}$ and $\varXi_Q(I_{\hh}) = I_{\hh}$,
   \item[(v)]  $\varXi_Q(A^*) = \varXi_Q(A)^*$ for all $A \in
\ogr{\hh}$,
   \item[(vi)] if $A\in \ogr{\hh}$, then $A \Ge 0$ if and
only if $\varXi_Q(A) \Ge 0$.
   \end{enumerate}
   \end{pro}
The selfmap $\varXi_Q$ preserves a variety of
fundamental properties of Hilbert space operators.
Below, we collect some of them. Given $\varDelta
\subseteq \cbb$, we set
   \begin{align*}
\borel{\varDelta} & = \{\varDelta \cap
\varDelta^\prime\colon \varDelta^\prime \text{ is a
Borel subset of } \cbb\},
   \\
\varDelta^* & = \{z \in \cbb \colon \bar z \in
\varDelta\}.
   \end{align*}
If $A\in \ogr{\hh}$ is normal, then $E_A$ stands for
the spectral measure of $A$.
   \begin{thm}\label{qaq}
Suppose $Q$ is a conjugation on $\hh$ and $A \in
\ogr{\hh}$. Then
   \begin{enumerate}
   \item[(i)] if $A$ is paranormal
$($hyponormal, subnormal, quasinormal, normal,
isometric, a partial isometry, an orthogonal
projection$)$, then so is $\varXi_Q(A)$; the reverse
implication holds as well,\footnote{\;By Proposition
\ref{abcon}(v), (i) holds if $A$ and $\varXi_Q(A)$ are
replaced by $A^*$ and $\varXi_Q(A)^*$, respectively.}
   \item[(ii)] $\sigma(\varXi_Q(A))=\sigma(A)^*$ and
$\varXi_Q((\lambda I - A)^{-1}) = (\bar \lambda I -
\varXi_Q(A))^{-1}$ if $\lambda \in \cbb \setminus
\sigma(A)$,
   \item[(iii)] if $A = U|A|$ is a polar decomposition of
$A$, then $\varXi_Q(A) = \varXi_Q(U)\varXi_Q(|A|)$ is a
polar decomposition of $\varXi_Q(A)$; in particular,
$|\varXi_Q(A)|=\varXi_Q(|A|)$,
   \item[(iv)] if $A$ is normal, then $E_{\varXi_Q(A)}(\varDelta)
= \varXi_Q(E_A(\varDelta^*))$ for every $\varDelta \in
\borel{\sigma(A)^*}$,
   \item[(v)] if $A$ is normal and $f\colon \sigma(A) \to \cbb$
is an $E_A$-essentially bounded Borel function, then
$\varXi_Q(f(A)) = f^*(\varXi_Q(A))$, where $f^*(z) =
\overline{f(\bar z)}$ for $z \in \sigma(A)^*$,
   \item[(vi)] if $f$ is a holomorphic mapping defined on an
open set $\varOmega \subseteq \cbb$ that contains
$\sigma(A)$, then $\varXi_Q(f(A)) = f^*(\varXi_Q(A))$,
where $f^*(z) = \overline{f(\bar z)}$ for $z \in
\varOmega^*$.
   \end{enumerate}
   \end{thm}
   \begin{proof}
(i) We only consider the case of subnormality because
the other cases can be obtained directly from
Proposition \ref{abcon}. By Lambert's theorem (cf.\
\cite{lam0}, see also \cite[Theorem 7]{St-Sz}), $A$ is
subnormal if and only if $\{\|A^n \xi
\|^2\}_{n=0}^{\infty}$ is a Stieltjes moment sequence
for every $\xi\in \hh$, or equivalently, by the
Stieltjes theorem (cf.\ \cite[Theorem 6.2.5]{ber}),
$A$ is subnormal if and only if
   \begin{align*}
\sum_{i,j=0}^n \lambda_i \bar \lambda_j A^{*(i+j+k)}
A^{i+j+k} \Ge 0, \quad \lambda_1, \ldots, \lambda_n
\in \cbb, \, n \in \zbb_+, \, k=0,1.
   \end{align*}
This combined with Proposition \ref{abcon} implies
that $A$ is subnormal if and only if $\varXi_Q(A)$ is
subnormal.

(ii) Apply Proposition \ref{abcon}.

(iii) By Proposition \ref{abcon}, $\varXi_Q(|A|) \Ge 0$ and
$\varXi_Q(|A|)^2 = \varXi_Q(A)^*\varXi_Q(A)$, and thus
$|\varXi_Q(A)|=\varXi_Q(|A|)$. Since $Q^{-1}=Q$ and, by
(i), $\varXi_Q(U)$ is a partial isometry such that
$\jd{\varXi_Q(U)} = Q(\jd{U}) =
Q(\jd{A})=\jd{\varXi_Q(A)}$, the assertion (iii) is proved.

(iv) By (i), the map $F_A\colon \borel{\sigma(A)^*}
\ni \varDelta \mapsto
\varXi_Q(E_A(\tau^{-1}(\varDelta))) \in \ogr{\hh}$ is
a spectral measure, where $\tau\colon \sigma(A) \to
\sigma(A)^*$ is given by $\tau(z)=\bar z$ for $z \in
\sigma(A)$. Applying the measure transport theorem
(cf.\ \cite[Theorem 1.6.12]{Ash}), we obtain
   \begin{align*}
\is{\varXi_Q(A)\xi}{\eta} = \int_{\sigma(A)} \tau(z)
\, \is{E_A(\D z)Q\eta}{Q\xi} = \int_{\sigma(A)^*} z \,
\is{F_A(\D z)\xi}{\eta}, \quad \xi,\eta \in \hh,
   \end{align*}
which together with (ii) implies (iv).

(v) By (iv), $f^*$ is $E_{\varXi_Q(A)}$-essentially bounded
Borel function and
   \begin{align*}
\is{\varXi_Q(f(A))\xi}{\eta} = \int_{\sigma(A)} \bar
f(z) \is{\varXi_Q(E_A(\D z)) \xi}{\eta} =
\int_{\sigma(A)^*} f^*(z) \is{F_A(\D z)\xi}{\eta}
   \end{align*}
for all $\xi, \eta \in \hh$, which gives (v).

(vi) Let $\gamma\colon [0,1] \to \cbb$ be a contour that
surrounds $\sigma(A)$ in $\varOmega$ (see \cite[p.\
260]{Rud}). Set $\tilde \gamma(t) = \overline{\gamma(1-t)}$
for $t\in [0,1]$. Since $\mathrm{Ind}_{\tilde \gamma}(z) =
\mathrm{Ind}_{\gamma}(\bar z)$ for $z\in \cbb \setminus
\gamma([0,1])^*$, where $\mathrm{Ind}_{\gamma}(z) =
\frac{1}{2\pi \I} \int_{\gamma} \frac{\D \zeta}{\zeta-z}$
for $z \in \cbb \setminus \gamma([0,1])$ (see \cite[Theorem
10.10]{Rud1}), we deduce that the contour $\tilde \gamma$
surrounds $\sigma(A)^*$ in $\varOmega^*$. Hence, by (ii),
we have
   \begin{align*}
\varXi_Q(f(A)) & = \varXi_Q\Big(\frac{1}{2\pi \I} \int_0^1
f(\gamma(t)) (\gamma(t) I- A)^{-1} \gamma^\prime(t) \D t
\Big)
   \\
& = \frac{1}{2\pi \I} \int_0^1 f^*(\tilde\gamma(t))
(\tilde\gamma(t) I- \varXi_Q(A))^{-1} \tilde \gamma^\prime
(t) \D t = f^*(\varXi_Q(A)),
   \end{align*}
which completes the proof.
   \end{proof}
It follows from Theorem \ref{qaq}(ii) that, in general, the
operators $A$ and $\varXi_Q(A)$ are not unitarily
equivalent.

Now we show that a conjugation on $\hh$ can be lifted
to a conjugation on $\varPhi(\hh)$.
   \begin{pro} \label{lem1+}
Let $\varPhi \in \fscr$ and $Q$ be a conjugation on
$\hh$. Then the map $\widehat Q\colon \varPhi(\hh) \to
\varPhi(\hh)$ given by $(\widehat Qf)(\xi) =
\overline{f(Q(\xi))}$ for $\xi \in \hh$ is a
well-defined conjugation on $\varPhi(\hh)$ such that
$\widehat QK^{\varPhi}_{\xi} = K^{\varPhi}_{Q\xi}$ and
$\widehat Q C_A \widehat Q= C_{\varXi_Q(A)}$ for all
$\xi \in \hh$ and $A\in\ogr{\hh}$.
   \end{pro}
   \begin{proof}
If $f \in \varPhi(\hh)$, then, by \eqref{rep} and
\eqref{rep2}, we have
   \begin{align*}
\Big|\sum_{i=1}^n \lambda_i \overline{f(Q\xi_i)}\Big|^2 =
\Big|\Big\langle f, \sum_{i=1}^n \lambda_i
K^{\varPhi}_{Q\xi_i}\Big\rangle\Big|^2 & \Le \|f\|^2
\sum_{i,j=1}^n K^{\varPhi}(Q\xi_j,Q\xi_i) \lambda_i \bar
\lambda_j
   \\
& = \|f\|^2 \sum_{i,j=1}^n K^{\varPhi}(\xi_i,\xi_j)
\lambda_i \bar \lambda_j
   \end{align*}
for all finite sequences $\{\lambda_i\}_{i=1}^n
\subseteq \cbb$ and $\{\xi_i\}_{i=1}^n \subseteq \hh$,
and thus, by the RKHS test, $\widehat Qf \in
\varPhi(\hh)$ and $\|\widehat Qf\| \Le \|f\|$. Hence
$\widehat Q$ is a well-defined contractive anti-linear
map from $\varPhi(\hh)$ to $\varPhi(\hh)$ such that
$\widehat Q (\widehat Q f)=f$ for all $f \in
\varPhi(\hh)$. This implies that $\|\widehat Qf\| =
\|f\|$ for all $f\in \varPhi(\hh)$. Applying the
anti-linearity of $\widehat Q$ and the polarization
formula, we conclude that $\widehat Q$ is a
conjugation on $\varPhi(\hh)$. The remaining part of
the conclusion is easily seen to be true.
   \end{proof}
It follows from Proposition \ref{lem1+} that if $A\in
\ogr{\hh}$ is $Q$-real, then $C_A$ is $\widehat
Q$-real (i.e., $\widehat Q C_A \widehat Q= C_{A}$).
   \section{\label{apap2}Paranormality of tensor products}
   In view of Fock's type model for a composition
operator $C_A$ (cf.\ Theorem \ref{fmod}), the question
discussed below seems to be of independent interest.
It is well-known that the tensor product of two
bounded paranormal operators may not be paranormal
even if they are equal (cf.\ \cite[Section 3]{An}).
Our aim is to prove the following.
   \begin{pro}
Let $\hh_j$ be a complex Hilbert space and $A_j \in
\ogr{\hh_j}$ for $j=1, \ldots, n$, where $n \in \nbb$.
Suppose $A_1 \otimes \dots \otimes A_n$ is a nonzero
paranormal operator. Then $A_j$ is paranormal for
every $j \in \{1, \ldots, n\}$.
   \end{pro}
   \begin{proof}
There is no loss of generality in assuming that $n=2$.
Then
   \begin{align} \notag
\|A_1 f \|^2 \|A_2 g\|^2 & = \|(A_1\otimes
A_2)(f\otimes g)\|^2
   \\  \notag
& \Le \|f\otimes g\| \|(A_1\otimes A_2)^2(f\otimes
g)\|
   \\    \label{tens1}
& = \|f\| \|A_1^2 f\| \|g\| \|A_2^2 g\|, \quad f \in
\hh_1, \, g \in \hh_2.
   \end{align}
Since $A_1\otimes A_2 \neq 0$, we see that $A_1 \neq
0$ and $A_2 \neq 0$. Note that
   \begin{align}   \label{tens2}
\jd{A_1^2} \subseteq \jd{A_1}.
   \end{align}
Indeed, if $f \in \jd{A_1^2}$, then, by \eqref{tens1}
applied to $g \in \hh \setminus \jd{A_2}$, we deduce
that $f\in \jd{A_1}$. It follows from \eqref{tens2}
that $A_1^2 \neq 0$. This and \eqref{tens1} yield
   \begin{align} \label{tens3}
0 < c:=\sup_{f\in\hh \setminus\jd{A_1^2}} \frac{\|A_1
f \|^2}{\|f\| \|A_1^2 f\|} \Le \inf_{g\in\hh
\setminus\jd{A_2}} \frac{\|g\| \|A_2^2 g\|}{\|A_2
g\|^2} < \infty.
   \end{align}
Hence, by \eqref{tens2} and \eqref{tens3}, we have
   \begin{align} \label{tens4}
\|A_1 f \|^2 & \Le c \|f\| \|A_1^2 f\|, \quad f \in
\hh_1,
   \\ \label{tens5}
\|A_2 g\|^2 & \Le \frac{1}{c} \|g\| \|A_2^2 g\|, \quad
g \in \hh_2.
   \end{align}
If $c \in (0,1)$, then \eqref{tens4} implies that
$\|A_1\|^2 \Le c\|A_1^2\| \Le c \|A_1\|^2$ and thus $c
\Ge 1$, a contradiction. Similar argument based on
\eqref{tens5} excludes the case of $c \in (1,\infty)$.
Hence $c=1$, which completes the proof.
   \end{proof}
The special case of symmetric tensor powers is even
simpler to prove.
   \begin{pro} \label{parasym}
Let $\hh$ be a complex Hilbert space, $A \in
\ogr{\hh}$ and $n \in \nbb$. Suppose $A^{\odot n}$ is
paranormal operator. Then $A$ is paranormal.
   \end{pro}
   \begin{proof}
The definition of paranormality of $A^{\odot n}$ leads
to
   \begin{align*}
\|Af\|^{2n} = \|A^{\odot n} f^{\otimes n}\|^2 \Le
\|f^{\otimes n}\| \|(A^{\odot n})^2 f^{\otimes n}\| =
\|f\|^n \|A^2 f\|^n, \quad f \in \hh,
   \end{align*}
which completes the proof.
   \end{proof}
Applying Theorem \ref{fmod}, Proposition \ref{parasym}
and Theorem \ref{qaq}(i), we get the following.
   \begin{cor}
Suppose $\varPhi \in \fscr$ and $A \in \ogr{\hh}$. If
$C_A \in \ogr{\varPhi(\hh)}$ and $C_A^*$ is
paranormal, then $A$ is paranormal.
   \end{cor}
   \bibliographystyle{amsalpha}
   
   \end{document}